\documentclass[11pt]{amsart}

\usepackage{amsmath,amsfonts,amssymb,amsthm,mathrsfs,hyperref}
\usepackage{tikz}
\usetikzlibrary{matrix,arrows}
\usepackage{makecell}
\usepackage{todonotes}
\usepackage{xcolor}
\usepackage{fullpage}
\usepackage{paralist}
\usepackage{tikz-cd}

\usepackage{arydshln}
\makeatletter
  \renewcommand*\env@matrix[1][*\c@MaxMatrixCols c]{%
    \hskip -\arraycolsep
    \let\@ifnextchar\new@ifnextchar
  \array{#1}}
\makeatother

\newcommand{\II}{I}
\newcommand{\IIm}[1][m]{I_{(#1)}}

\newcommand{\SSign}[3][]{\Sigma_{#2}(#3; {#1})}
\newcommand{\SSig}[3][]{\Sigma_{#2}(#3)}
\newcommand{\SSigM}[2][]{\SSig[n]{#2}{\min\{k-1,n\}-1}}

\definecolor{forestgreen(traditional)}{rgb}{0.0, 0.27, 0.13}
\definecolor{forestgreen(web)}{rgb}{0.13, 0.55, 0.13}
\definecolor{airforceblue}{rgb}{0.36, 0.54, 0.66}

\hypersetup{
  colorlinks,
  citecolor=forestgreen(traditional),
  linkcolor=magenta,
  urlcolor=airforceblue}

\newcommand{\set}[1]{\left\{#1\right\}}

\def\CC{{\mathbb C}}

\def\GG{{\mathbb G}}

\def\NN{{\mathbb N}}
\def\P{{\mathbb P}}
\def\PP{{\mathbb P}}

\newcommand{\SL}{\mathbf{SL}}

\newcommand{\YF}{\mathrm{YF}}

\newcommand{ \cupdot}{\mathbin{\mathaccent\cdot \cup}}

\DeclareMathOperator{\codim}{codim}

\DeclareMathOperator{\Gr}{Gr}
\DeclareMathOperator{\Hilb}{Hilb}
\DeclareMathOperator{\Hom}{Hom}

\DeclareMathOperator{\im}{im}

\DeclareMathOperator{\Sing}{Sing}
\DeclareMathOperator{\Symm}{Sym}

\DeclareMathOperator{\Spec}{Spec}
\DeclareMathOperator{\Vertex}{Vertex}
\DeclareMathOperator{\rank}{rank}
\newcommand \brank {\ensuremath{\underline{\mathrm{rank}}}}

\newcommand{\Sym}{{S}^{\bullet}}
\newcommand{\Sd}{{S}}

\def\ds{\displaystyle}

\newcommand \ti[1]{\textit{#1}}
\newcommand \trm[1]{\textrm{#1}}

\newcommand \mbf[1]{\mathbf{#1}}

\newtheorem{theorem}{Theorem}
\newtheorem*{theorem*}{Theorem}
\newtheorem{lemma}[theorem]{Lemma}
\newtheorem*{lemma*}{Lemma}
\newtheorem{proposition}[theorem]{Proposition}
\newtheorem{corollary}[theorem]{Corollary}

\theoremstyle{definition}

\theoremstyle{remark}
\newtheorem{remark}[theorem]{Remark}
\newtheorem{question}[theorem]{Question}

\newtheorem{example}[theorem]{Example}

\newcommand\TT{\mathbb{T}}
\newcommand\PN{\PP^{N}}
\newcommand\Pa{\PN}
\newcommand\Pb{\PP^{\beta}}

\newcommand\Pm{\PP^{m}}
\newcommand\Pn{\PP^{n}}
\newcommand{\lin}[1]{\langle\, #1 \,\rangle}

\newcommand\mo[2][\mathbf{t}]{\mathsf{mono}[#1]_{\leq #2}}

\newcommand\XXk[2]{(#1)^{#2}}
\newcommand\Pmk[1][k]{\XXk{\PP^{m}}{#1}}
\newcommand\Pnk{\XXk{\PP^{n}}{k}}
\newcommand\Zk[1][k]{\XXk{Z}{#1}}
\newcommand\Xk{\XXk{X}{k}}
\newcommand\Lk{\XXk{L}{k}}

\begin{document}

\title[Singular loci of higher secant of Veronese embeddings]{On the singular loci of higher secant varieties of Veronese embeddings}

\author[K. Furukawa]{Katsuhisa Furukawa}
\address{Katsuhisa Furukawa, Department of Mathematics, Faculty of Science, Josai University, Saitama, Japan}
\email{katu@josai.ac.jp}

\author[K. Han]{Kangjin Han}
\address{Kangjin Han, 
Department of Liberal Arts and Sciences,
Daegu-Gyeongbuk Institute of Science \& Technology (DGIST),
Daegu 42988,
Republic of Korea}
\email{kjhan@dgist.ac.kr}

\thanks{The first named author was supported by JSPS KAKENHI Grant Number 22K03236.
  The second named author was supported by a National Research Foundation of Korea (NRF) grant funded by the Korean government (MSIT, no. 2017R1E1A1A03070765 and no. 2021R1F1A104818611).}

\date{}
% \date{\today}

\begin{abstract}
  The $k$-th secant variety of a projective variety $X \subset \mathbb{P}^N$, denoted by $\sigma_k(X)$, is defined to be the closure of the union of $(k-1)$-planes spanned by $k$ points on $X$. In this paper, we examine the $k$-th secant variety $\sigma_k(v_d(\mathbb{P}^n)) \subset \mathbb{P}^N$ of the image of the $d$-uple Veronese embedding $v_d$ of $\mathbb{P}^n$ to $\mathbb{P}^N$ with $N=\binom{n+d}{d}-1$, and focus on the singular locus of $\sigma_k(v_d(\mathbb{P}^n))$, which is only known for $k\le3$.
  To study the singularity for arbitrary $k,d,n$, we define \emph{the $m$-subsecant locus} of $\sigma_k(v_d(\mathbb{P}^n))$ to be the union of $\sigma_k(v_d(\mathbb{P}^m))$ with any $m$-plane $\mathbb{P}^m \subset \mathbb{P}^n$. By investigating the projective geometry of moving embedded tangent spaces along subvarieties and using known results on the secant defectivity and the identifiability of symmetric tensors, we determine whether the $m$-subsecant locus is contained in the singular locus of $\sigma_k(v_d(\mathbb{P}^n))$ or not.
  Depending on the value of $k$, these subsecant loci show an interesting trichotomy between generic smoothness, non-trivial singularity, and trivial singularity. In many cases, they can be used as a new source for the singularity of the $k$-th secant variety of $v_d(\mathbb{P}^n)$ other than the trivial one, the $(k-1)$-th secant variety of $v_d(\mathbb{P}^n)$. We also consider the case of the $4$-th secant variety of $v_d(\mathbb{P}^n)$ by applying main results and computing conormal space via a certain type of Young flattening. Finally, we present some generalizations and discussions for further developments.
\end{abstract}

\maketitle
\tableofcontents \setcounter{page}{1}

\section{Introduction}\label{intro}

Throughout the paper, we work over $\CC$, the field of complex numbers. Let $X \subset\PP^N$ be an embedded projective variety. The \ti{$k$-th secant variety} of $X$ is defined as:
\begin{gather}\label{eq:def-k-secant}
  \sigma_k(X)=\overline{\bigcup_{x_1,\dots,x_k\in X}\lin{x_1,\dots,x_k}} \subset \PP^N,
\end{gather}
where $\lin{x_1,\dots,x_k}  \subset \PN$ denotes the linear span of the points $x_1,\dots,x_k$ and the overline means the Zariski closure. In particular, $\sigma_1(X)=X$ and $\sigma_2(X)$ is often simply called the \ti{secant} or \ti{secant line} variety of $X$ in the literature. 

The construction of secant varieties (or more generally, \ti{join} construction of subvarieties) is not only one of the most famous methods in classical algebraic geometry, but also a very popular subject in recent years, especially in connection with fields of research such as tensor rank and decomposition, algebraic statistics, data science, geometric complexity theory and so on (see \cite{La,LaGCT} for more details).

Despite of a rather long history and the popularity, most of the fundamental questions on the higher secant varieties $\sigma_k(X)$ still remain open even for many well-known base varieties $X$. For instance, one can ask the secant defectivity question, which concerns the dimension of $\sigma_k(X)$. We say that $\sigma_k(X)$ is \emph{secant defective} (or simply \emph{defective}) if the dimension of $\sigma_k(X)$ is less than the expected one, $\min\{N,k\cdot \dim(X)+k-1\}$.
It is classically known that higher secant varieties of curves are non-defective (e.g. \cite[Corollary~1.2.3]{Rus}).
Due to the famous theorem of Alexander-Hirschowitz \cite{AH}, we know the dimensions of higher secant varieties of all Veronese varieties.
In other research, there are only a few cases where the dimension theorem for $\sigma_k(X)$ is fully determined
(see \cite{Rus}, \cite{Zak} for references).
Questions about defining equations of $k$-th secant varieties have also only been answered for very small $k$'s of a few cases and seem still far from understanding the essence of the sources for the equations (see also \cite[Chapter 5]{La} for a reference).

In this paper, we concentrate on the case of $X=v_d(\Pn)  \subset \PN$,
the Veronese variety,
which is the image of the Veronese embedding
$v_d: \Pn \hookrightarrow \PN$ with $N={n+d\choose d}-1$.
In particular, we study the singular locus of $\sigma_k(v_d(\Pn))$, an arbitrary higher secant variety of the Veronese variety.

The knowledge on singularities of higher secant varieties is fundamental and very important for its own sake in the study of algebraic geometry and also can be useful for problems in applications. For example, it can be used as a key condition to establish the \emph{identifiability} of structured tensors (e.g. the introduction in \cite{COV} and references therein).

For any irreducible variety $X \subset\PP^N$, it is classically well-known that 
\begin{equation}\label{basic_sing_containment}
\sigma_{k-1}(X) \subset \Sing(\sigma_k(X))
\end{equation} 
unless $\sigma_k(X)$ fills up the whole linear span $\lin{X}$ (see \cite[Proposition 1.2.2]{Rus}). In the paper, we say that a point $p\in\sigma_k(X)$ is a \ti{non-trivial} singular point if $p\notin \sigma_{k-1}(X)$ and $\sigma_k(X)$ is singular at $p$, while the points belonging to $\sigma_{k-1}(X)$ are called \ti{trivial} singular points.

There are some known results on the singular loci of $k$-th secant varieties $\sigma_k(X)$, mostly for very small $k$'s. The equality in (\ref{basic_sing_containment}) holds for determinantal varieties defined by minors of a generic matrix. It is completely described for the $2$-nd secant variety of the Segre product of projective spaces in \cite{MOZ}. It has recently been generalized to the case of $\sigma_2(X)$, where $X$ is a Segre-Veronese embedding by \cite{AMZ} and  $X$ is a Grassmannian by \cite{GS}. For the $3$-rd secant variety of the Grassmannian $\mathbb{G}(2,6)$, the dimension and some other properties of the singularity have been studied in \cite{AOP2}.

For the case of Veronese varieties, it is classical that $\Sing(\sigma_k(v_d(\Pn)))=\sigma_{k-1}(v_d(\Pn))$ holds both for the binary case (i.e., $n=1$) and for the case of quadratic forms (i.e., $d=2$) (see e.g. \cite[Chapter 1]{IK}). For $k=2$, it was proved in \cite{Kan} that the above equality holds also for any $d,n$. In these cases, the $k$-th secant variety has only the trivial singularity.
For $k=3$, the singular locus was completely determined by the second author in \cite{H18}; in particular, the non-trivial singularity occur if and only if $d=4$ and $n\ge3$. 

In the present paper, we explore the singular locus of any higher secant variety of the Veronese variety, and introduce a new main origin for the singularity other than the trivial singularity. We call this the `subsecant locus'.
As in our main results, these loci show an interesting \ti{trichotomy} phenomenon among generic smoothness, non-trivial singularity, and trivial singularity.

For any given point $p\in\sigma_k(v_d(\Pn))$, there exists an $m$-plane (i.e., $m$-dimensional linear subvariety) $\Pm$ of $\Pn$ with $1\le m\le k-1$ such that $p\in\sigma_k(v_d(\Pm))$;
it immediately follows for a general $p$, and even if $p$ is in the boundary of the closure in (\ref{eq:def-k-secant}), it is also true by considering  $(1,d-1)$-symmetric flattening (see \S\ref{sect_m=1} for details).
So, from now on, we say that $\sigma_k(v_d(\Pm)) \subset\sigma_k(v_d(\Pn))$ is
an \ti{$m$-subsecant variety} of $\sigma_k(v_d(\Pn))$ if $m< k-1$ and $m<n$,
and simply call it a \ti{subsecant variety} in case there is no confusion.
We also define \ti{the $m$-subsecant locus of $\sigma_k(v_d(\Pn))$},
\begin{gather}\label{eq:def-m-subsec}
  \SSig[n]{k,d}{m} \,\trm{~or~}\, \SSign[\Pn]{k,d}{m} =
  \bigcup_{\Pm \subset\Pn} \sigma_k(v_d(\Pm)).
\end{gather}
It naturally forms an increasing sequence of loci in the $k$-th secant variety as
\begin{equation*}
  \SSig[n]{k,d}{1} \subset\SSig[n]{k,d}{2} \subset\cdots \subset\SSig[n]{k,d}{\min\{k-1,n\}-1} \subset\sigma_k(v_d(\Pn))=\SSig[n]{k,d}{\min\{k-1,n\}}.
\end{equation*}
In particular, $\SSig[n]{k,d}{\min\{k-1,n\}-1}$
is the union of all proper subsecant varieties,
which we call the \ti{maximum subsecant locus} of the given $k$-th secant variety $\sigma_k(v_d(\Pn))$.
Any point of $\sigma_k(v_d(\Pn))$ outside the maximum subsecant locus and the previous secant variety $\sigma_{k-1}(v_d(\Pn))$ is called a point of the \ti{full-secant locus}.
Note that for $k=3$ of \cite{H18},
when $d=4$ and $n\ge3$,
the singular locus of $\sigma_3(v_d(\Pn))$ is given as the maximum subsecant locus $\SSig[n]{3,d}{1}$,
which is the only case where the singularity pattern of $3$-rd secant varieties becomes exceptional,
while for all the other $d, n$ the singularity is just the trivial one, $\sigma_2(v_d(\Pn))$ (see Remark~\ref{rem_main_result} (a)). Most of the previously known results on singular loci of secant varieties can be understood in this viewpoint (see Remark~\ref{remk_for_next_project}).

Thus, a basic question for our concern could be stated as follows:
For given $k,d,m,n$,
\begin{center}
  When is $\sigma_k(v_d(\P^n))$ singular at points of an $m$-subsecant locus $\SSig[n]{k,d}{m}$?
\end{center}

In principle, it is somewhat straightforward (despite the computational complexity) to check the singularity, once a complete set of equations for a higher secant variety is attained. But, as mentioned above, not much is known about the defining equations and they seem quite far from being fully understood at this moment, even for the Veronese case (see \cite{LO, FH3} for the state of the art).
Due to the lack of knowledge on the equations for the higher secant variety, it is very difficult to determine the singular locus in general. 

In this paper, without further understanding on the equations (!), we introduce a geometric way to pursue it for this kind of problems, which is based on a careful study on the behavior of embedded tangent spaces moving along a locus in the Veronese variety . For the case $m=1$, we first present the following result for the $1$-subsecant locus $\SSig[n]{k,d}{1} = \bigcup_{\PP^1 \subset \Pn} \sigma_k(v_d(\PP^1))$ of $\sigma_k(v_d(\P^n))$, which is a generalization of \cite[Theorem~2.1]{H18} (i.e., $k=3$ case) to any higher $k$-th secant varieties of Veronese varieties.

\begin{theorem}\label{thm_m1}
  Let $v_d:\P^n \rightarrow \P^N$
  be the $d$-uple Veronese embedding with $n\ge2$, $d\ge3$, and $N={n+d\choose d}-1$.
  Assume $k\ge3$.
  For $(k,d,n)\neq(3,4,2)$, the following holds:
\begin{enumerate}
\item[(i)] If $k\le\frac{d+1}{2}$, then $\sigma_k(v_d(\P^n))$ is smooth at every point in $\SSig[n]{k,d}{1} \setminus\sigma_{k-1}(v_d(\P^n))$.
\item[(ii)] If $k=\frac{d+2}{2}$ , then $\SSig[n]{k,d}{1} \subset\Sing(\sigma_k(v_d(\P^n)))$ but $\SSig[n]{k,d}{1} \not\subset\sigma_{k-1}(v_d(\P^n))$ (i.e., non-trivial singularity). This case occurs only if $d$ is even.
\item[(iii)] If $k\ge\frac{d+3}{2}$, then $\SSig[n]{k,d}{1} \subset\sigma_{k-1}(v_d(\P^n))$ (i.e., trivial singularity, unless $\sigma_{k}(v_d(\mathbb{P}^n))=\mathbb{P}^N$).

\end{enumerate}
For $(k,d,n)=(3,4,2)$, it holds that
\begin{enumerate}
\item[(iv)] $\sigma_3(v_4(\PP^2))$ is smooth at every point in $\SSig[2]{3,4}{1} \setminus\sigma_{2}(v_4(\PP^2))$.
\end{enumerate} 
\end{theorem} 

Concerning singular points of arbitrary $\sigma_k(v_d(\Pn))$ originated from subsecant loci, we prove the following general theorems on the $m$-subsecant locus $\SSig[n]{k,d}{m}$ with $m\ge2$ and $k\ge4$ as the main results.

\begin{theorem}\label{general_m}
  Let $v_d:\P^n \rightarrow \P^N$
  be the $d$-uple Veronese embedding
  with $n\ge3$, $d\ge3$, and $N={n+d\choose d}-1$.
  Let $k \geq 4$ and let $\Pm  \subset \Pn$ be an $m$-plane with $2\leq m < \min\{k-1,n\}$.
  Assume that $(d,m) \notin \mathcal{E} = \{(3,3),(3,4),(3,5),(4,2),(4,3),(4,4),(5,2),(6,2)\}$.

  For $(k,d,n)\neq(4,3,3)$, setting
  $\mu = \Bigr\lceil \frac{\binom{m+d}{m}}{m+1} \Bigl\rceil$,
  we have: 
\begin{enumerate}
\item[(i)] If $k < \mu$, then $\sigma_k(v_d(\P^n))$ is smooth at a general point in $\SSig[n]{k,d}{m} \setminus\sigma_{k-1}(v_d(\P^n))$. 
\item[(ii)] If $k = \mu$, then $\SSig[n]{k,d}{m} \subset\Sing(\sigma_k(v_d(\P^n)))$ but $\SSig[n]{k,d}{m} \not\subset\sigma_{k-1}(v_d(\P^n))$ (i.e., non-trivial singularity).
\item[(iii)] If $k>\mu$, then $\SSig[n]{k,d}{m} \subset\sigma_{k-1}(v_d(\P^n))$
  (i.e., trivial singularity, unless $\sigma_{k}(v_d(\mathbb{P}^n))=\mathbb{P}^N$).
\end{enumerate}
For $(k,d,n)=(4,3,3)$, it holds that
\begin{enumerate}
\item[(iv)] $\sigma_4(v_3(\P^3))$ is smooth at every point in $\SSig{4,3}{2} \setminus\sigma_{3}(v_3(\P^3))$.
\end{enumerate} 
\end{theorem}

\begin{theorem}\label{general_m2} In the same situation as Theorem~\ref{general_m}, for
  \begin{gather*}
    (d,m) \in \mathcal{E} = \{(3,3),(3,4),(3,5),(4,2),(4,3),(4,4),(5,2),(6,2)\},
  \end{gather*}
  if $k$ is in one of the ranges named (i), (ii), (iii)
  in Table \ref{singTable1} below,
  then the following property corresponding to the name of the range holds.

  \begin{enumerate}
  \item[(i)] $\sigma_k(v_d(\P^n))$ is smooth at a general point in $\SSig[n]{k,d}{m} \setminus\sigma_{k-1}(v_d(\P^n))$.
  \item[(ii)] $\SSig[n]{k,d}{m} \subset\Sing(\sigma_k(v_d(\P^n)))$ but $\SSig[n]{k,d}{m} \not\subset\sigma_{k-1}(v_d(\P^n))$.
  \item[(iii)] $\SSig[n]{k,d}{m} \subset\sigma_{k-1}(v_d(\P^n))$.
  \end{enumerate}
\end{theorem} 

\begin{table}[hbt]
  \begin{tabular}{|l|c|c|c|c|}
    \hline
    $(d,m)$ & $\frac{\binom{m+d}{m}}{m+1}$ & (i) & (ii) & (iii)\\ \hline
    %%%
    $(3,3)$ & $5$            & $k\le 5$ & None & $6\le k$\\ \hline % n-id-full
    $(3,4)$ & $7$            & $k\le 6$ & $k=7,8$ & $9\le k$\\ \hline % defect
    $(3,5)$ & ${28}/{3}$ & $k\le 8$ & $k=9, 10$ & $11\le k$\\ \hline % n-id-sub
    $(4,2)$ & $5$            & $k\le 4$ & $k=5,6$ & $7\le k$\\ \hline % defect
    $(4,3)$ & ${35}/{4}$ & $k\le 7$ & $k=8,9,10$ & $11\le k$\\ \hline % n-id-sub,defect
    $(4,4)$ & $14$           & $k\le 13$ & $k=14, 15$ & $16\le k$\\ \hline % defect
    $(5,2)$ & $7$            & $k\le 7$ & None & $8\le k$\\ \hline % n-id-full
    $(6,2)$ & ${28}/{3}$ & $k\le 8$ & $k=9, 10$ & $11\le k$\\ \hline % n-id-sub
  \end{tabular}

  \caption{(Non-)singularity of $\SSig[n]{k,d}{m}$ in Theorem \ref{general_m2}}
  \label{singTable1}
\end{table}

To understand the reason for considering the conditions that $(d,m)$ is in $\mathcal{E}$ or not,
we discuss the secant defectivity and the generic identifiability
of an $m$-subsecant variety $\sigma_k(v_d(\Pm)) \subset \sigma_k(v_d(\Pn))$.
Set $\Pb = \lin{v_d(\Pm)} \subset \PP^N$ with $\beta = \binom{m+d}{m}-1$.
where $\sigma_k(v_d(\Pm)) \subset \Pb$ is of dimension $\leq km+k-1$.

From \cite{AH},
the codimension of $\sigma_k(v_d(\P^m))$ in $\Pb$ is greater than
$\max\{\beta-(km+k-1),0\}$
(i.e., $\sigma_k(v_d(\P^m))$ does not fill $\Pb$ and is \emph{secant defective})
if and only if $d=2$ and $2 \leq k \leq m$, or
$(k,d,m) = (7,3,4),(5,4,2),(9,4,3),(14,4,4)$;
indeed, in the latter case,
the four $k$-th secant varieties $\sigma_k(v_d(\P^m))$'s are hypersurfaces of $\Pb$.
Except these defective cases,
we have $\sigma_k(v_d(\P^m))=\Pb$ if
$km+k-1 \geq \beta$, or equivalently $k \geq {\binom{m+d}{m}}/({m+1})$.

We say that a point $a \in \Pb$ is \emph{$k$-identifiable}
if there is a \emph{unique}
$k$-tuple of points
$x_1, \dots, x_k \in v_d(\Pm)$ such that $a \in \lin{x_1, \dots, x_k}$.
We also say that
$\sigma_k(v_d(\P^m))$
is \emph{generically identifiable} if 
a general point $a\in\sigma_k(v_d(\P^m))$ is $k$-identifiable.
From \cite[Theorem~1]{GM},
a general point $a \in \Pb$ is $k$-identifiable
\big(or $\sigma_k(v_d(\P^m))$ is generically identifiable in the case of $\sigma_k(v_d(\P^m)) = \Pb$\big)
if and only if $m = 1$ and $k=(d+1)/2$, or $(k,d,m) = (5,3,3),(7,5,2)$.
From \cite[Theorem~1.1]{COV2},
in the case when $\sigma_k(v_d(\P^m)) \subsetneq \Pb$ is not secant defective,
$\sigma_k(v_d(\P^m))$ is \emph{not} generically identifiable
if and only if $(k,d,m)=(9,3,5),(8,4,3),(9,6,2)$.

\begin{remark}\label{rem_main_result}
  We make some remarks on the theorems above.
\begin{enumerate}[(a)]

\item In the case of $k=3$ and $n \geq 3$, 
  the condition of (ii) in Theorem~\ref{thm_m1} holds if and only if $d=4$,
  and then 
  $\Sing(\sigma_3(v_4(\Pn))$ is equal to
  $\bigcup_{\PP^1  \subset \Pn} \lin{v_4(\PP^1)} = \SSig[n]{3,4}{1}$
  since
  $\sigma_3(v_4(\PP^1)) = \lin{v_4(\PP^1)}$ and
  $\sigma_2(v_4(\Pn))  \subset \bigcup_{\PP^1  \subset \Pn} \lin{v_4(\PP^1)}$ (see also Corollary \ref{coro_for_cone_sing}).
  This gives a geometric description of
  the only exceptional case for the singular loci of the $3$-rd secant varieties
  in \cite{H18}.

\item (i) of Theorem~\ref{thm_m1} is stronger than (i) of Theorems \ref{general_m} and \ref{general_m2}, since it claims smoothness for `every point' in the $m$-subsecant locus $\SSig[n]{k,d}{m}$ outside $\sigma_{k-1}(v_d(\P^n))$.
  The `general point' condition of (i) in Theorems \ref{general_m} and \ref{general_m2} cannot be deleted (see Example~\ref{sing_special_pt}). 
  We also note that, the ranges of $k$ of (i) and (ii) are slightly different between Theorems~\ref{thm_m1} and \ref{general_m}.
  If $m=1$ and $d$ is even, then the case $k = \lceil{\binom{m+d}{m}}/({m+1})\rceil = (d+2)/2$
  is (ii) of Theorem~\ref{thm_m1},
  which is similar to that of Theorem~\ref{general_m};
  for this $k$, $\sigma_k(v_d(\PP^1))$ is \emph{not} generically identifiable.
  However, if $d$ is odd,
  then the case $k = {\binom{m+d}{m}}/({m+1}) = (d+1)/2$ belongs to (i) of Theorem~\ref{thm_m1}.
  
\item
  (i) in Theorems~\ref{thm_m1}, \ref{general_m}, and \ref{general_m2}
  corresponds to the {$k$-identifiable} case of a general point $a\in\sigma_k(v_d(\P^m))$.
  From the viewpoint of the secant fiber in the incidence (\ref{main_incidence}) in \S\ref{sect_prep},
  it means that the fiber $p^{-1}(a)$ under the projection $p$ consists of only one element up to permuting $x_i$'s. On the other hand, (ii) of Theorems~\ref{thm_m1}, \ref{general_m}, and \ref{general_m2} corresponds to the case of generic {non-identifiability} of the subsecant variety $\sigma_k(v_d(\P^m))$ with the situation of $\sigma_{k-1}(v_d(\P^m)) \subsetneq \Pb=\lin{v_d(\P^m)}$,
  except $(k,d,m,n) = (3,4,1,2), (4,3,2,3)$.
  This non-identifiability of a general $a \in \sigma_k(v_d(\P^m))$ occurs if $\dim p^{-1}(a) > 0$, or if $\dim p^{-1}(a)=0$ and $\# p^{-1}(a)\ge2$ (modulo permutation).
  For $m = 1$, only the former occurs in the case $k = (d+2)/2$.
  For $m \geq 2$,
  the former corresponds to
  the case where $k$ is the ceiling of ${\binom{m+d}{m}}/({m+1}) \notin \NN$ with $\sigma_k(v_d(\P^m))=\Pb$,
  to the case where $\sigma_k(v_d(\P^m)) \subsetneq \Pb$ is secant defective,
  or to the case
  where $\sigma_{k-1}(v_d(\P^m)) \subsetneq \Pb$ is secant defective
  (i.e., $(k,d,m) = (8, 3, 4), (6, 4, 2), (10, 4, 3), (15, 4, 4)$),
  and the latter corresponds to
  the case where $k$ is the number ${\binom{m+d}{m}}/({m+1}) \in \NN$ with $\sigma_k(v_d(\P^m))=\Pb$ except $(k,d,m) = (5,3,3),(7,5,2)$,
  or to the case $(k,d,m)=(9,3,5),(8,4,3),(9,6,2)$.

\item $(k,d,m,n) = (3,4,1,2), (4,3,2,3)$ (i.e., (iv) of Theorems~\ref{thm_m1} and ~\ref{general_m}),
  \ti{a posteriori}, turn out to be the \ti{only} two exceptional cases which do not follow this trichotomy pattern;
  in other words,
  though $k$ belongs to the range of (ii)
  and the generic {non-identifiability} of $\sigma_k(v_d(\P^m))$ holds,
  $\SSig{k,d}{m}$ does not provide non-trivial singular points
  (see also Remark~\ref{m=1_exception_why}).
  Indeed, in {\cite{FH3}} we show
  that if $(k,d,n)=(3,4,2),(4,3,3)$, then $\sigma_k(v_d(\Pn))$ is a \emph{del Pezzo $k$-th secant variety}, that is, a $k$-th secant variety of \emph{next-to-minimal degree}.
  In this sense, these two cases also belong to a special class
  with respect to the degrees of higher secant varieties. (For basic definitions and results on such varieties, see \cite{CR}, \cite{CK}.)
\end{enumerate}
\end{remark}

As an application of our main results for $\SSig[n]{4,d}{2}$ and $\SSig[n]{4,d}{1}$, we obtain the following result on the singularity of the \ti{fourth} secant variety of any Veronese variety.

\begin{theorem}[Singular locus for $\sigma_4(v_d(\P^n))$]\label{sing_4th_Vero}
  Let $v_d:\P^n \rightarrow \P^N$
  be the $d$-uple Veronese embedding
  with $n\ge3$, $d\ge3$, and $N={n+d\choose d}-1$.
  Then the followings hold.
\begin{enumerate}
\item[(i)] $\sigma_4(v_d(\P^n))$ is smooth at every point outside $\SSig[n]{4,d}{2} \cup\sigma_{3}(v_d(\P^n))$. 
\item[(ii)] If $d\ge4$, then a general point in $\SSig[n]{4,d}{2} \setminus\sigma_{3}(v_d(\P^n))$ is also a smooth point of $\sigma_4(v_d(\P^n))$. When $d=3$ and $n=3$, all points in $\SSig[3]{4,d}{2} \setminus\sigma_{3}(v_d(\P^n))$ are smooth. If $d=3$ and $n\ge4$, then $\SSig[n]{4,d}{2} \subset\Sing(\sigma_4(v_d(\P^n)))$ but $\SSig[n]{4,d}{2} \not\subset\sigma_{3}(v_d(\P^n))$ (i.e., non-trivial singularity).
\item[(iii)]  For $d\ge 7$, all points in $\SSig[n]{4,d}{1} \setminus\sigma_{3}(v_d(\P^n)) \subset \sigma_4(v_d(\P^n))$ are smooth.
  When $d=6$, $\SSig[n]{4,d}{1} \subset\Sing(\sigma_4(v_d(\P^n)))$ but $\SSig[n]{4,d}{1} \not\subset\sigma_{3}(v_d(\P^n))$ (i.e., non-trivial singularity).
  When $d\le 5$, $\SSig[n]{4,d}{1} \subset\sigma_{3}(v_d(\P^n))$ (i.e., trivial singularity).
\end{enumerate}
\end{theorem} 

For a projective variety $X \subset \PN$,
we denote by $\Vertex(X)$ the set of vertices of $X$.
Then $X$ is a cone if and only if $\Vertex(X) \neq \emptyset$.

\begin{example}[Cases with a nice description]\label{eg_smallest_4thsec}
  The smallest case for the singular locus of $\sigma_4(v_d(\PP^n))$ beyond the classical results for $d=2$ or $n=1$ is $(d,n)=(3,2)$, but in this case there is nothing to check, because $\sigma_4(v_3(\PP^2))$ fills up the ambient space $\PP^9$ and then $\Sing(\sigma_4(v_3(\PP^2)))=\emptyset$. In the case $(d,n)=(3,3)$, by Theorem \ref{sing_4th_Vero} (i) and (ii), we have
\begin{equation*}
\Sing(\sigma_4(v_3(\PP^3)))=\sigma_3(v_3(\PP^3)).
\end{equation*}
For the case $(d,n)=(3,4)$, if $V$ denotes a $5$-dimensional $\CC$-vector space with $\PP^4 = \PP V$, Theorem~\ref{sing_4th_Vero} tells us that the singular locus of the $4$-th secant variety of $v_3(\PP V)$ in $\PP^{34}$ is precisely the locus of cubic hypersurfaces in $5$ variables which are \ti{cones with the vertex dimension $\ge1$} as follows:
\begin{multline*}
  \Sing(\sigma_4(v_3(\PP V)))
  =\sigma_3(v_3(\PP V)) \cup\SSign[\PP V]{4,3}{2} \\
  =\sigma_3(v_3(\PP V)) \cup\big\{\bigcup_{\PP^2 \subset\PP V}\sigma_4(v_3(\PP^2))\big\}
  =\bigcup_{\PP^2 \subset\PP V}\lin{v_3(\PP^2)} \\ 
  =\big\{f\in \PP S^3 V~|~\textrm{the cubic hypersurface $X \subset \PP V$ defined by $f$ is a cone with $\dim \Vertex(X)\ge1$}\big\},
  \end{multline*}
which is just the maximum subsecant locus $\SSign[\PP V]{4,3}{2}$, an irreducible $15$-dimensional locus in the $19$-dimensional variety $\sigma_4(v_3(\PP V))$. In the same argument, we can obtain 
\begin{equation*}
  \trm{$\Sing(\sigma_4(v_3(\PP^n)))=\SSign[\Pn]{4,3}{2}$ for any $n\ge4$.}
\end{equation*}
Such a simple description of the singular locus can be attained in a few more cases (see Corollary~\ref{coro_for_cone_sing} for details). 
\end{example}

The paper is structured as follows. In \S\ref{sect_prep}, as preparation, we first recall some preliminaries on $k$-th secant varieties and corresponding incidences.
Then, using projective techniques, such as Terracini's lemma, the trisecant lemma, descriptions of embedding tangent spaces, and tangential projections, we reveal several geometric properties of $m$-subsecant varieties in higher secant varieties of Veronese varieties, which are crucial for the proof of main theorems. In \S\ref{sect_m=1}, as an illustration of the whole picture and our main ideas, we treat the case $m=1$ and prove Theorem~\ref{thm_m1}. In \S\ref{sect_m>=2} we deal with the general case (i.e., $m\ge2$) and prove Theorem~\ref{general_m} and Theorem~\ref{general_m2} as generalizing the ideas used in the previous section. We would like to remark that this can be done because the dimension theorem \cite{AH} and the generic identifiability question \cite{COV2,GM} were settled for the case of Veronese varieties. In \S\ref{Sing4Vero} we focus on the singular locus of the $4$-th secant variety and prove Theorem~\ref{sing_4th_Vero} by dividing the case into two parts; `full-secant locus points (i.e., $m=3$)' treated in Theorem~\ref{sing_4th_Vero_gen} via Young flattening and conormal space computation and `the subsecant locus' by Corollary \ref{sing_4th_Vero_sub}. Finally, we make some generalizations and remarks on the material for further developments in \S\ref{sect_final_remark}.\\

\paragraph*{\textbf{Acknowledgements}} The authors would like to give thank to Korea Institute for Advanced Study (KIAS) for inviting them and giving a chance to start the project. We are grateful to Giorgio Ottaviani for a helpful comment on Example \ref{eg_smallest_4thsec} and to Luca Chiantini and Jarek Buczy{\'n}ski for related discussions as well. The first author also wishes to express his gratitude to Hajime Kaji and Hiromichi Takagi for useful discussions.
Finally, we would like to thank the anonymous referees very much for their careful reading, pointing out some error in the first version, and many suggestions which help us to improve the final exposition of the paper.

\section{Some geometric properties of subsecant varieties}\label{sect_prep}

\subsection{Projection from the incidence to the secant variety}

For a (reduced and irreducible) variety $X$, we denote $X\times\cdots\times X$, the (usual) product of $k$ copies of $X$, by $\Xk$.
We denote the $k$-fold symmetric product of $X$, $\Xk/S_n$, by $\Symm^k(X)$.

For the $d$-uple Veronese embedding $v_d: \Pn \rightarrow \PP^{N}$ with $N={n+d\choose d}-1$, we regard the incidence variety $\II = \IIm[n] \subset \PP^{N} \times \Pnk$
to be the Zariski closure of 
\begin{equation}\label{main_incidence}
  \II^0 = \IIm[n]^0 = \set{(a, {x_1', \dots, x_k'})
    \mid a \in \lin{x_1, \dots, x_k} \text{ and } \dim \lin{x_1, \dots, x_k} = k-1},
\end{equation}
where we write $x_i = v_d(x_i')$ for $x_i' \in \Pn$. 
Taking the first projection
$p : \II \rightarrow \PP^{N}$, we have $p(\II)= \sigma_k(v_d(\Pn))$ (see also \cite[Definition 1.1.3]{Rus}, \cite[Chapter I, \textsection{}1, Chapter V]{Zak}). For any $a\in\sigma_k(v_d(\Pn))$, $p^{-1}(a)$ is often called the \ti{secant fiber} of $a$. 
Note that $\II$ is invariant under permuting factors on $(\Pn)^k$ from the definition so that both $p$ and $q_i$ maps factor through $\PP^{N}\times\Symm^k(\Pn)$.

We also have $\dim \II = {nk + k - 1}$ by
considering general fibers of $q : \II \rightarrow \Pnk$. For each $1 \leq i \leq k$, let 
\[
  q_i : \II \rightarrow \Pnk \rightarrow \Pn
\]
be the composition of $q$ and the projection to the $i$-th factor $\Pnk \rightarrow \Pn$. 
Then the following commutative diagram is obtained:
\begin{equation*}
  \begin{tikzcd}
    & I \arrow[ld,"p"] \arrow[rd, "q"'] \arrow[rrd, shift left=0.5ex, "q_i"] \\
    \PN && \Pnk \arrow[r] & \Pn.
  \end{tikzcd}
\end{equation*}

\begin{remark}\label{rem:large-d-incidence:1}
  We have some remarks on the incidence variety $\II \subset \PP^{N} \times \Pnk$.
  \begin{enumerate}[(a)]
  \item $\II^0$ can be viewed as a $\PP^{k-1}$-bundle over a non-empty open subset $U$ of $\Pnk$, consisting of $k$-tuples of points with the expected spanning dimension, so that $\II$ is irreducible.
    Further, for $q^{-1}(U) = I \cap (\Pn \times U)$,  we have $\II^0=q^{-1}(U)$ since both are irreducible closed subsets in $\PP^{N}\times U$ having the same dimension. So, for any $(a, {x_1', \dots, x_k'})\in \II \setminus \II^0$ and for $x_i = v_d(x_i')$, it holds that $\dim \lin{x_1, \dots, x_k} < k-1$
    (in other words, there is no $(a, {x_1', \dots, x_k'})\in \II \setminus \II^0$ such that $\dim \lin{x_1, \dots, x_k}= k-1$ but $a\notin\lin{x_1, \dots, x_k}$).
    Finally, note that the Euclidean closure of $\II^0$ also coincides with $\II$ in this case (see e.g. \cite[Theorem~3.1.6.1]{LaGCT}).
    
  \item In the case of $\dim \sigma_k(v_d(\Pn)) = {nk + k - 1}$,
    it holds $p(\II  \setminus \II^0) \neq \sigma_k(v_d(\Pn))$,
    and then $p^{-1}(a) \subset I^0$ for general $a \in \sigma_k(v_d(\Pn))$.
    In addition, if $k \leq \binom{n+d-1}{n}$, then as setting
    \begin{gather*}
      D = \set{(x_1', \dots, x_k') \in \Pnk
        \mid \dim(\lin{v_{d-1}(x_1'), \dots, v_{d-1}(x_k')}) < k-1}\!,
    \end{gather*}
    we have $\dim (\II \cap (\PP^{N} \times D)) <\dim (\II)$, and hence
    $p(\II \cap (\PP^{N} \times D)) \neq \sigma_k(v_d(\Pn))$.
  \item (Alternative incidences for the $k$-th secant variety) In the literature, instead of $(X)^k$, other spaces such as the symmetric product $\Symm^k(X)$ (e.g. in \cite{COV2}), the Hilbert scheme of degree $k$ $0$-dimensional subschemes $\Hilb_k(X)$ (e.g. in \cite{BuBu}), and the Grassmannian $\mathbb{G}(k-1,N)$ (e.g. in \cite{Mel}) have also been used in the incidence to consider the $k$-th secant variety of $X \subset\PN$.  
  \end{enumerate}
\end{remark}

\begin{remark}\label{rem:def-y'} Let us fix an $m$-plane $L=\Pm  \subset \Pn$ and consider the $d$-uple Veronese embedding of $\Pm$, $v_d: L=\Pm \rightarrow \PP^{\beta}$ with $\beta={m+d\choose m}-1$.
  We often use the following notation.
  \begin{enumerate}[(a)]
  \item
    Let $\hat{L} \subset \Pn$ be any $(n-m-1)$-plane not intersecting $L$.
    Changing homogeneous coordinates
    $t_0, t_1, \dots, t_m, u_1, u_2, \dots, u_{m'}$ on $\Pn$
    with $m' = n-m$,
    we may assume that
    $L \subset \Pn$ is the zero set of $u_1, \dots, u_{m'}$ and
    $\hat{L} \subset \Pn$ is the zero set of $t_0, \dots, t_m$.
    For any $x_i' \in \Pn$, say
    \begin{gather*}
      x_i' = [x_{i, 0}':\dots : x_{i, m}': x_{i, m+1}' : \dots : x_{i,n}'] \in \Pn,
    \end{gather*}
    then we set $y_i' = [x_{i, 0}' : \dots : x_{i, m}': 0 : \dots : 0]$.
    Thus, $y_i'$ gives a point of the $m$-plane if $x_i'$ is not of the form $[0 : \dots : 0 :\ast:\dots:\ast]$. By abuse of notations, we denote by $y_i'$ both corresponding points in $\Pm$ and in $\Pn$.
    Further, as considering linear projections $\pi_1:\PP^n \dashrightarrow\PP^m$ from the center $\hat{L}$
    (eliminating the $u$-variables),
    and $\pi_2:\PP^N\dashrightarrow\Pb$
    (eliminating all the monomials of degree $d$ which involve $u$-variables),
    we have a natural commuting diagram as $v_d\circ\pi_1=\pi_2\circ v_d$
    \begin{equation}\label{eq:diag-PnPmPNPb}
      \begin{tikzcd}
        \Pn \arrow[r,hook,"v_d"] \arrow[d, dashed, "\pi_1"'] & \PP^{N} \arrow[d, dashed, "\pi_2"] \\
        \Pm \arrow[r,hook,"v_d"] & \PP^{\beta}
      \end{tikzcd}
      \raisebox{-2.2em}{\kern-1ex.}
    \end{equation}

    In particular, when $d=2$ and $n=m+1$ (i.e., $m'=1$), then $\pi_1$ is an (inner) projection from one point $a\in\Pn$ and $\pi_2$ corresponds to a tangential projection of $\P^N$ from $\TT_{v_{2}(a)} v_{2}(\Pn)$.
  \item On the affine open subset $\set{t_0 \neq 0}$,
    the $d$-uple Veronese embedding $v_d: \Pn \rightarrow \PP^{N}$
    is parameterized by monomials in $\mo[t,u]{d}$,
    where $\mo[t,u]{e}~\textrm{(resp. $\mo[t]{e}$)}$
    is defined to be the set of monomials in $\CC[t_1, \dots, t_m, u_1, \dots, u_{m'}]$ (resp. in $\CC[t_1, \dots, t_m]$) of degree $\leq e$ for an integer $e$.
  \end{enumerate}
\end{remark}

Let us study the behavior of some points in the boundary of $I$,
which belong to $I$ but do not belong to $I^0$, as follows.

\begin{lemma}\label{lem-I-I0}
  Let $L=\Pm  \subset \Pn$ be any $m$-plane with $m < n$, and consider $\Pb = \lin{v_d(L)}  \subset \P^N$.
  For $v_d: L=\Pm \hookrightarrow \PP^{\beta}$,
  let us take $\IIm  \subset \Pb \times \Pmk$ to be
  the closure of
  the set of points $(a, {x_1', \dots, x_k'}) \in \Pb \times \Pmk$
  such that
  $a \in \lin{x_1, \dots, x_k}$ and $\dim \lin{x_1, \dots, x_k} = k-1$
  (i.e., the incidence variety of the same kind as $\II = \IIm[n]$
  in (\ref{main_incidence})).
  In this setting, for any $(a, x_1', \dots, x_k') \in \II  \setminus \II^0$ with $a \in \Pb$,
  one of the following conditions holds:
  \begin{enumerate}[\quad (i)]
  \item
    $a \in p(\IIm  \setminus \IIm^0)$; or
  \item
    there is a subset $C  \subset \Pmk$ of dimension $>0$ such that $\set{a} \times C  \subset \IIm$.
  \end{enumerate}
\end{lemma}

\begin{proof}
  Let $\set{W_j}$ be the irreducible components of $\II  \setminus \II^0$.
  For each $j$, we take
  \begin{gather}\label{eq:axWj}
    (a_j,x_{j,1}', \dots, x_{j,k}') \in W_j \subset \PN \times \Pnk.
  \end{gather}
  Let $\hat{L} \subset \Pn$ be a general $(n-m-1)$-plane such that all $x_{j,i}' \notin \hat{L}$
  and $L \cap \hat{L} = \emptyset$.
  For the linear projection $\pi_1: \Pn \dashrightarrow \Pm$ from the center $\hat{L}$,
  we have a natural linear projection $\pi_2: \PN \dashrightarrow \Pb$ as in Remark~\ref{rem:def-y'} with the diagram~(\ref{eq:diag-PnPmPNPb}),
  and then define the projections
  \begin{gather*}
    \rho_1: \Pa \times \Pnk \dashrightarrow \Pa \times \Pmk,\quad 
    \rho_2: \Pa \times \Pmk \dashrightarrow \Pb \times \Pmk,
  \end{gather*}
  where $\overline{\rho_2(\rho_1(\II))} = \IIm$.
  Taking two Segre embeddings
  \begin{gather*}
    Seg_1: \Pa \times \Pnk \hookrightarrow \PP^{l_1},\quad 
    Seg_2: \Pa \times \Pmk \hookrightarrow \PP^{l_2},
  \end{gather*}
  we may regard $\rho_1$ as the restriction of a linear projection
  $\pi_R: \PP^{l_1} \dashrightarrow \PP^{l_2}$ whose center $R$
  is a certain linear subvariety of $\PP^{l_1}$.
  Let $\tilde{\PP}^{l_1}  \subset \PP^{l_1} \times \PP^{l_2}$ be
  the graph of $\pi_R$, which coincides with
  the blowing-up of $\PP^{l_1}$ with respect to $R$.
  Let $r_1: \tilde{\PP}^{l_1} \rightarrow \PP^{l_1}$ and
  $r_2: \tilde{\PP}^{l_1} \rightarrow \PP^{l_2}$ be projections.
  Then we have the following commutative diagram:
  \begin{equation}\label{eq:blup-proj-seg}
    \begin{tikzcd}
      \tilde{\PP}^{l_1} \arrow[d, "r_1"'] \arrow[dr, "r_2"] \\
      \PP^{l_1} \arrow[r, dashed, "\pi_R"] &\PP^{l_2} \\
      \PN \times \Pnk \arrow[u, hook, "Seg_1"] \arrow[r, dashed, "\rho_1"] &\PN \times \Pmk \arrow[u, hook, "Seg_2"'] \arrow[r, dashed, "\rho_2"] &\Pb \times \Pmk.
    \end{tikzcd}
  \end{equation}
  Under the embedding $\II  \subset \Pa \times \Pnk \hookrightarrow \PP^{l_1}$, taking
  $R \cap \II$ in $\PP^{l_1}$, we have $\codim(R \cap \II, \II) \geq 2$, as follows.
  For the center $\hat{L}$ of $\pi_1: \Pn \dashrightarrow \Pm$, we set
  \begin{gather*}
    \hat{L}(i) = \{\,(x_1', \dots, x_k') \in \Pnk \mid x_i' \in \hat{L}\,\}.
  \end{gather*}
  We consider the projection $\bar{q}: \PN \times \Pnk \rightarrow \Pnk$.
  Then a point $z \in \PN \times \Pnk$ belongs to the indeterminacy locus of $\rho_1$
  if and only if $\bar{q}(z) \in \hat{L}(i)$ for some $i$ with $1 \leq i \leq k$.
  Since $\pi_R|_I (= \pi_R \circ Seg_1|_{\II})$ coincides with $Seg_2 \circ \rho_1|_{\II}$,
  and since $R \cap \II$ is the indeterminacy locus of $\pi_R|_I$,
  we have
  \begin{gather*}
    R \cap \II = q^{-1}\bigl(\bigcup_{i=1}^k \hat{L}(i)\bigr),
  \end{gather*}
  where $q$ is equal to $\bar{q}|_{\II}: \II \rightarrow \Pnk$.
  In particular, since the dimension of the fibers of $q|_{\II^0}$ is constant,
  and since $\codim(\hat{L}(i), \Pnk) = m+1$,
  it follows that $\codim(R \cap \II^0, \II) = m+1$.
  Now, let $\set{Q_s}$ be the irreducible components of $R \cap \II$.
  Then $\codim(Q_s, \II) = m+1$ in the case $Q_s \cap \II^0 \neq \emptyset$,
  i.e., $Q_s \not\subset \II \setminus \II^0$.
  In order to consider the remaining case $Q_s \subset \II \setminus \II^0$,
  we use the irreducible components $\set{W_j}$ of $\II  \setminus \II^0$.
  For $(a_j,x_{j,1}', \dots, x_{j,k}') \in W_j$ given in (\ref{eq:axWj}),
  since $x_{j,i}' \notin \hat{L}$, it follows $W_j  \not\subset R$ for any $j$.
 
  If an irreducible component $Q_s$ of $R \cap \II$ satisfies $Q_s \subset \II \setminus \II^0$,
  then there is some $j$ such that $Q_s  \subset R \cap W_j  \subsetneq W_j  \subsetneq \II$,
  and then $\codim(Q_s, \II) \geq 2$.
  As a result, in any case,
  each irreducible component of $R \cap \II$ is of codimension $m+1$ or at least $\geq 2$,
  which implies $\codim(R \cap \II, \II) \geq 2$.
  \medskip

  Now, let
  $(a, x_1', \dots, x_k') \in \II  \setminus \II^0$
  satisfy
  $a \in \Pb$, where $\dim \lin{x_1, \dots, x_k} < k-1$ for $x_i = v_d(x_i')$
  as in Remark~\ref{rem:large-d-incidence:1} (a)
  (note that we do \emph{not} know $a \in \lin{x_1, \dots, x_k}$).
  We regard $(a, x_1', \dots, x_k')$ as a point of $\PP^{l_1}$ under
  the embedding $\II  \subset \PP^{l_1}$.
  \medskip

  If $(a, x_1', \dots, x_k') \notin R$,
  then $\rho_1(a, x_1', \dots, x_k') = (a, y_1', \dots, y_1')$ is determined in $\overline{\rho_1(\II)}  \subset \Pa \times \Pmk$,
  where $y_i' \in \Pm$ is the image of $x_i'$ under $\Pn \rightarrow \Pm$.
  Since $a \in \Pb$,
  \begin{gather*}
    \rho_2(\rho_1(a, x_1', \dots, x_k')) = (a, y_1', \dots, y_1')
  \end{gather*}
  is determined and is contained in $\IIm$.
  Then, $\dim \lin{y_1, \dots, y_k} < k-1$ for $y_i=v_d(y_i') \in \Pb$,
  which means $(a, y_1', \dots, y_1') \in \IIm  \setminus \IIm^0$.
  \medskip

  Assume $(a, x_1', \dots, x_k') \in R$. We consider
  the blowing-up $\tilde{\PP}^{l_1}$ of $\PP^{l_1}$ with respect to $R$,
  and the projections $r_1,r_2$ in the diagram~(\ref{eq:blup-proj-seg}).
  Note that, for the strict transformation
  $S  \subset \tilde{\PP}^{l_1}$ of $\Pa \times \Pnk  \subset \PP^{l_1}$,
  two composite morphisms
  \begin{gather*}
    S \xrightarrow{r_1|_S} \Pa \times \Pnk \rightarrow \Pa
    \,\text{ and }\,
    S \xrightarrow{r_2|_S} \Pa \times \Pmk \rightarrow \Pa
  \end{gather*}
  coincide since it holds on an open subset of $S$.
  Let $E  \subset \tilde{\PP}^{l_1}$ be the exceptional divisor, and let
  $\widetilde{\II}  \subset \tilde{\PP}^{l_1}$ be the strict transformation of
  $\II  \subset \PP^{l_1}$. Then $r_2(\widetilde{\II}) = \overline{\rho_1(\II)}$.
  Let
  \begin{gather*}
    E_1 = r_1^{-1}(a, x_1', \dots, x_k') \cap (E \cap \widetilde{\II}),
  \end{gather*}
  the fiber of $E \cap \widetilde{\II} \rightarrow R \cap \II$ at $(a, x_1', \dots, x_k')$.
  It follows from $\codim(R \cap \II, \II) \geq 2$
  that $\dim(E_1) \geq 1$. Since $r_1^{-1}(z) \simeq r_2(r_1^{-1}(z))  \subset \PP^{l_2}$ for each $z \in \PP^{l_1}$, we have $\dim(r_2(E_1)) \geq 1$.
  Since the image of $r_2(E_1)$ under $\Pa \times \Pmk \rightarrow \Pa$ is $\set{a}$,
  there is $C  \subset \Pmk$ of positive dimension such that
  $r_2(E_1) = \set{a} \times C  \subset \Pa \times \Pmk$.
  Since $a \in \Pb$ and
  \begin{gather*}
    \rho_2(r_2(E_1)) \subset \rho_2(r_2(\tilde{\II})) = \overline{\rho_2(\rho_1(\II))} = \IIm,
  \end{gather*}
  we have
  $\set{a} \times C = \rho_2(\set{a} \times C)  \subset \IIm$.
\end{proof}

\begin{lemma}[Non-triviality of subsecant varieties]\label{lem:sigma-k-1}
  For an $m$-plane $L = \Pm  \subset \Pn$, we have
  \begin{gather*}
    \sigma_{k}(v_d(L)) \cap \sigma_{k-1}(v_d(\Pn))  \subset \sigma_{k-1}(v_d(L)).
  \end{gather*}
  In particular, $\sigma_{k}(v_d(L))  \not\subseteq \sigma_{k-1}(v_d(\Pn))$ unless $\sigma_{k-1}(v_d(L))=\sigma_{k}(v_d(L))$.
\end{lemma}
\begin{proof}
  Let $a \in \sigma_{k}(v_d(L)) \cap \sigma_{k-1}(v_d(\Pn))$ (note that $a$ can be in the boundary of $\sigma_{k-1}(v_d(\Pn))$).
  For a general point
  $b_0 \in \sigma_{k-1}(v_d(\Pn))$,
  we take an irreducible curve $C  \subset \sigma_{k-1}(v_d(\Pn))$
  such that
  $a, b_0 \in C$.
  Let $\pi_2 : \Pa \dashrightarrow \Pb$ be the linear projection in Remark~\ref{rem:def-y'} (a), and let $C' = \overline{\pi_2(C)}  \subset \Pb$.
  Since $a \in \Pb$, we have $a = \pi_2(a) \in C'$.

  Since $b_0$ is general,
  for a general point $b \in C$,
  we have
  $b \in \lin{x_1, \dots, x_{k-1}}$ with $x_1, \dots, x_{k-1} \in v_d(\Pn)$.
  Take $x_i' \in \Pn$ with $x_i = v_d(x_i')$.
  Setting $y_i = \pi_2(x_i)$, we have
  $y_i = v_d(y_i')$ with $y_i' \in L$ for each $i=1, \dots, k-1$
  as in Remark~\ref{rem:def-y'} (a).
  It follows that $\pi_2(b) \in \lin{y_1, \dots, y_{k-1}}  \subset \sigma_{k-1}(v_d(L))$.
  As a result,
  $a \in C'  \subset \sigma_{k-1}(v_d(L))$ and the assertion follows.
\end{proof}

\begin{remark}\label{brpp_and_closedness} We have some consequences of Lemma~\ref{lem:sigma-k-1}.
\begin{enumerate}[(a)]
\item(Border Rank Preserving Pair) For any $\Pm \subset\Pn$ and for any $k,d > 0$, by Lemma~\ref{lem:sigma-k-1}, we can derive
\begin{equation}\label{Pm_brpp}
\sigma_k(v_d(\Pn))\cap \langle v_d(\Pm) \rangle=\sigma_k(v_d(\Pm))
\end{equation}
as a set.
Since one inclusion is obvious, let us prove $\sigma_k(v_d(\Pn))\cap \langle v_d(\Pm) \rangle \subset\sigma_k(v_d(\Pm))$. Suppose that it does not hold. Then, there exists a form $f\in \sigma_k(v_d(\Pn))\cap \langle v_d(\Pm) \rangle$ with $f\in\sigma_{k_0}(v_d(\Pm)) \setminus\sigma_{k}(v_d(\Pm))$ for some $k_0>k$. Then, $f\in \sigma_{k_0}(v_d(\Pm))\cap\sigma_{k_0-1}(v_d(\Pn))$ so that $f\in\sigma_{k_0-1}(v_d(\Pm))$ by Lemma~\ref{lem:sigma-k-1}. Similarly, as repeating the same `descent' argument, we have $f\in\sigma_{k}(v_d(\Pm))$, which is a contradiction. Thus, the equality in (\ref{Pm_brpp}) is true.

In other words, for any $d$-th Veronese embedding $X=v_d(\Pn)$ and the linear span $L=\langle v_d(\Pm)\rangle$ of any sub-Veronese variety $v_d(\Pm)$, we showed that $(X,L)$ is a \ti{border rank preserving pair} for any $k, d>0$ in the terminology of \cite[Definition 5.7.3.1]{La}
(we would also like to note that \cite[Theorem~1.1]{BGL} can imply the same result for any $d\ge k$ in case of $\Pm \subset\Pn$).
\item(Every $\SSig[n]{k,d}{m}$ is closed) Recall that the \ti{symmetric subspace variety} $Sub_m(S^d V)$ (\cite[section 7.1.3]{La}) is defined as
  \begin{equation*}
\{f\in \P S^d V~|~\exists ~W \subset V,~\dim W=m+1~,~f\in \P S^d W\}~.
\end{equation*}
Let $n=\dim\P V$. For any $m\le n$, we have $Sub_m(S^d V)=\bigcup_{\Pm \subset\P V} \langle v_d(\Pm)\rangle$. Now, we show that $\SSig[n]{k,d}{m}$ is the intersection of the whole $k$-th secant $\sigma_k(v_d(\P V))$ and the symmetric subspace variety $Sub_m(S^d V)$ set-theoretically. See that
\begin{align*}
\sigma_k(v_d(\P V))\cap Sub_m(S^d V)&=\sigma_k(v_d(\P V))\cap \ds\bigcup_{\Pm \subset\P V} \langle v_d(\Pm)\rangle=\ds\bigcup_{\Pm \subset\P V} \bigg( \sigma_k(v_d(\P V))\cap \langle v_d(\Pm)\rangle \bigg)\\
                                    &=\ds\bigcup_{\Pm \subset\P V} \sigma_k(v_d(\Pm))~~\big(=\SSig[n]{k,d}{m}\big)\quad\trm{by (\ref{Pm_brpp})~.}
\end{align*}
Therefore, we obtain that $\SSig[n]{k,d}{m}=\sigma_k(v_d(\P V))\cap Sub_m(S^d V)$ as a set and in particular every $m$-subsecant $\SSig[n]{k,d}{m}$ is a \ti{Zariski-closed} locus in $\sigma_k(v_d(\P V))$. 
\end{enumerate}
\end{remark}

\subsection{General secant fiber of a subsecant variety, entry locus, and its Veronese image}
\label{sec:degen-subv-small}

For an $m$-dimensional projective variety $Z  \subset \Pb$,
we take another incidence variety $J  \subset \Pb \times \Zk$
to be the Zariski closure of 
\begin{equation}\label{incidence_on_Z}
  J^0 = \set{(a, x_1, \dots, x_k)
    \mid a \in \lin{x_1, \dots, x_k} \text{ and } \dim \lin{x_1, \dots, x_k} = k-1},
\end{equation}
with the projections
\[
  p : J \rightarrow \sigma_k(Z)  \subset \Pb, \quad q_i : J \rightarrow \Zk \rightarrow Z,
\]
where $\Zk \rightarrow Z$ is the projection to the $i$-th factor
for $1 \leq i \leq k$.
Then we have $\dim J = mk + k - 1$ for any $k\le \dim\langle Z\rangle+1$ by
considering general fibers of $J \rightarrow \Zk$.

For $a\in\sigma_k(Z)$, the scheme-theoretic image $q_i(p^{-1}(a))$ in $Z$ is called the ($k$-th) \ti{entry locus of $Z$ with respect to $a$} in the literature. It is known that for a \ti{general} $a\in\sigma_k(Z)$ the locus $q_i(p^{-1}(a))$ is equidimensional and moreover, if $Z$ is smooth and in the characteristic 0, then $p^{-1}(a)$ is \ti{smooth} so that $q_i(p^{-1}(a))$ is reduced (see \cite[Definition 1.4.5]{Rus}).

Let $Z, X  \subset \PN$ be projective varieties of dimensions $m, n$.
Let $Z  \subset X$ and $Z  \subset \Pb$, where $\Pb$ is a $\beta$-plane of $\PN$ (i.e., $Z$ is degenerate in $\PN$). Now, a general point $a\in Z$ does \emph{not} have to be \ti{general} in $X$ any longer. If $\beta < km+k-1$, then the projection $p$ has positive dimensional fibers.

We begin with a consequence of Terracini's lemma in our setting and add two more lemmas concerning `the entry locus' $q_i(p^{-1}(a))$.

\begin{lemma}\label{lem:TTx-in-TTa}
  Assume that
  $\sigma_k(Z)  \not\subset \Sing(\sigma_k(X))$.
  Let $F$ be an irreducible component of $p^{-1}(a)$ for a general point $a \in \sigma_k(Z)$ in the incidence (\ref{incidence_on_Z}).
  Then, for a general point $x \in q_i(F)$ with $1\le i \le k$, we have
  \[
    \TT_x (X)  \subset \TT_a (\sigma_k(X)),
  \]
  where $\TT_x (X)  \subset \PN$ means the embedded tangent space to $X$ at $x$.
\end{lemma}
\begin{proof}
  Since $\sigma_k(Z)  \setminus \Sing(\sigma_k(X))$ is non-empty open in $\sigma_k(Z)$ and $a$ is general, it is a smooth point of $\sigma_k(X)$; hence the embedded tangent space $\TT_a (\sigma_k(X))$ is defined. In addition, $a$ is contained in the $(k-1)$-plane $\lin{x_1, \dots, x_k}$
  for general $x_1, \dots, x_k \in Z$ with $x_i = x$.
  Then the assertion follows by Terracini's lemma (cf. \cite[Corollary 1.4.2]{Rus}, \cite[Chapter II, 1.10, Chapter V, 1.4]{Zak}).
\end{proof}

\begin{lemma}\label{lem:q_1F-fin}
  For a projective variety $Z$, let $\lin{Z}=\Pb$
  and consider the incidence $J$ as (\ref{incidence_on_Z}) with $k \geq 2$.
  Suppose that the ($k-1$)-secant of $Z$ is not defective and not equal to $\Pb$. Then,
  for a general point $(a,x_1, \dots, x_k)\in J$,
  and for any irreducible component $F$ of $p^{-1}(a)$ containing $(a,x_1, \dots, x_k)$,
  $q_i|_F: F \rightarrow q_i(F)$ is generically finite.
\end{lemma}

\begin{proof}
  For simplicity, we set $i=1$. First, since $(a,x_1, \dots, x_k)$ is a general point of the incidence $J$, we may assume that $a$ is a general point of $\sigma_k(Z)$ and $x_1,\dots,x_k$ are general $k$ points on $Z$.

  Let $\pi_{x_1}: \Pb \dashrightarrow \PP^{\beta-1}$ be the linear projection from $x_1$.
  Since $Z$ is non-degenerate in $\Pb$ and $\sigma_{k-1}(Z) \neq \Pb$, 
  the map $\pi_{x_1}|_{\sigma_{k-1}(Z)}$ is generically finite onto its image
  (otherwise, a general point $x_1$ is contained in $\Vertex(\sigma_{k-1}(Z))$, a linear subvariety of $\sigma_{k-1}(Z)$, a contradiction).

  Let $J'  \subset \Pb \times \Zk[k-1]$
  be the incidence for $(k-1)$-secant of $Z$, i.e., the closure of
  \begin{gather*}
    \set{(b,\tilde{x}_2, \dots, \tilde{x}_k)
      \mid b \in \lin{\tilde{x}_2, \dots, \tilde{x}_k} \text{ and } \dim \lin{\tilde{x}_2, \dots, \tilde{x}_k} = k-2}.
  \end{gather*}
  Since $\sigma_{k-1}(Z)$ is not secant defective,
  the first projection $p_{J'}: J' \rightarrow \sigma_{k-1}(Z)$ is generically finite.
  Hence the composite map
  \begin{gather}\label{eq:rho-comp-map}
    \rho = \pi_{x_1} \circ p_{J'}: J' \rightarrow \pi_{x_1}(\sigma_{k-1}(Z))
  \end{gather}
  is generically finite.

  Let $J_1 = q_1^{-1}(x_1)  \subset J$.
  Then
  $\pi_{x_1} \circ p|_{J_1}: J_1 \rightarrow \pi_{x_1}(\sigma_{k-1}(Z))$
  is dominant
  (this is because, for general $b \in \sigma_{k-1}(Z)$,
  we take a general point
  $c \in \lin{x_1, b}$ and take
  $k-1$ points $\tilde{x}_2,\dots,\tilde{x}_k \in Z$ such that
  $b \in \lin{\tilde{x}_2,\dots,\tilde{x}_k}$; then
  $(c,x_1,\tilde{x}_2,\dots,\tilde{x}_k) \in J_1$, whose image under $\pi_{x_1} \circ p$
  is $\pi_{x_1}(c) = \pi_{x_1}(b)$).
  Since $(a,x_1, \dots, x_k) \in J_1$ is general in $J$ and
  by the dominance of $\pi_{x_1} \circ p|_{J_1}$,
  we may consider $\alpha = \pi_{x_1}(a)$ as a general point in $\pi_{x_1}(\sigma_{k-1}(Z))$.

  Let $F$ be an irreducible component of $p^{-1}(a)$ containing $(a , x_1, x_2, \dots, x_k)$,
  and suppose that $q_1|_{F}$ has general fibers of positive dimensions.
  Then, the fiber of $q_1|_{F}$ at $x_1 \in q_1(F)$,
  \begin{gather*}
    F \cap J_1
  \end{gather*}
  is of positive dimension, which means that
  we have $(a, x_1, \tilde{x}_2,\dots,\tilde{x}_k) \in F \cap J_1$
  for fixed $a,x_1$ and moving $(\tilde{x}_2,\dots,\tilde{x}_k)$ with positive dimension.

  For general $(a, x_1, \tilde{x}_2,\dots,\tilde{x}_k) \in F \cap J_1$,
  we have the intersection point $b(\tilde{x}_2,\dots,\tilde{x}_k)$ of
  the line $\lin{a, x_1}$ and
  the hyperplane $\lin{\tilde{x}_2,\dots,\tilde{x}_k}$ in $\lin{x_1, \tilde{x}_2,\dots,\tilde{x}_k} = \PP^{k-1}$.
  Note that
  \begin{gather*}
    \pi_{x_1}(b(\tilde{x}_2,\dots,\tilde{x}_k))=\pi_{x_1}(a)=\alpha.
  \end{gather*}
  The $k$-tuple $(b(\tilde{x}_2,\dots,\tilde{x}_k), \tilde{x}_2,\dots,\tilde{x}_k)$
  moves with positive dimension since so does ($k-1$)-tuple $(\tilde{x}_2,\dots,\tilde{x}_k)$.
  In other words, the following locus is of positive dimension,
  \[
    \set{(b(\tilde{x}_2,\dots,\tilde{x}_k), \tilde{x}_2,\dots,\tilde{x}_k)
      \mid (a, x_1, \tilde{x}_2,\dots,\tilde{x}_k) \in F \cap J_1}
    \subset \rho^{-1}(\alpha)  \subset J',
  \]
  which contradicts that $\rho$, given in (\ref{eq:rho-comp-map}), is generally finite.
\end{proof}

\begin{lemma}\label{lem:q_1F=Z}
  For a projective variety $Z \subset\lin{Z}=\Pb$, suppose that $\sigma_{k-1}(Z)$ is a hypersurface in $\Pb$ and $\sigma_{k}(Z)=\Pb$. Then, we have $q_i(p^{-1}(a))=Z$ for a general point $a\in\sigma_{k}(Z)$.
  In particular, there is an irreducible component $F$ of $p^{-1}(a)$ such that $q_i(F)=Z$.
\end{lemma}

\begin{proof}
  Since $p^{-1}(a)$ is invariant under permuting $x_i$-factors, we set $i=1$ for simplicity. Take a general $a\in\sigma_k(Z)=\Pb$.
  Since $\Vertex(\sigma_{k-1}(Z))$ is a linear subvariety of $\sigma_{k-1}(Z) \subsetneq\Pb$ and $Z$ is non-degenerate in $\Pb$,
  \begin{gather*}
    \Vertex(\sigma_{k-1}(Z))\cap Z \subsetneq Z.
  \end{gather*}
  This implies that for general $x\in Z$,
  $\dim\lin{x, \sigma_{k-1}(Z)} > \dim(\sigma_{k-1}(Z))$ so that $\lin{x, \sigma_{k-1}(Z)}=\Pb$. Thus, any given general point $a\in\Pb$ sits on a line $\lin{x,b}$ for any general $x\in Z$ and for some general $b\in\sigma_{k-1}(Z)$.
  Taking $k-1$ points $\tilde{x_2},\dots,\tilde{x_k} \in Z$ such that $b \in \lin{\tilde{x_2},\dots,\tilde{x_k}}$,
  we have $(a,x,\tilde{x_2},\dots,\tilde{x_k})\in p^{-1}(a)$, which means $q_1(p^{-1}(a))=Z$.
\end{proof}

Now let us focus on the case $Z = v_d(\Pm)  \subset \PP^{\beta=\binom{m+d}{m}-1}$,
the image of the $d$-uple Veronese embedding of $\Pm$. Here we prove a very useful proposition, which is of independent interest itself. In Proposition~\ref{prop:dim-qiF-k-1-plane}, we consider the entry locus of a general point in $Z$ and estimate the dimension of the linear span of its image under ($d-1$)-uple Veronese embedding $v_{d-1}: \Pm \rightarrow \PP^{\beta_{d-1} = \binom{m+d-1}{m}-1}$.

\begin{proposition}\label{prop:dim-qiF-k-1-plane}
  Let $Z = v_d(\Pm) \subset \Pb$ with
  $d \geq 3$, $2 \leq m \leq k-2$, and
  $\beta=\binom{m+d}{m}-1 < km+k-1$.
  Assume $\dim (\sigma_{k-1}(Z)) = (k-1)m+k-2 < \beta$ (in other words, the ($k-1$)-th secant of $Z$ is not defective
  and not equal to $\Pb$), where we have $(k-1)m+k \leq \binom{m+d}{m}$.
  Let $J  \subset \Pb \times \Zk$ be the Zariski closure of
  the incidence (\ref{incidence_on_Z}),
  let $(a,x_1, \dots, x_k) \in J$ be a general point,
  and let $F  \subset J$ be an irreducible component of $p^{-1}(a)$ containing
  $(a,x_1, \dots, x_k)$.
  Then the following holds.
  
  \begin{enumerate}[(i)]
  \item 
    If $(k-1)m+k < \binom{m+d}{m}$, 
    then we have
    \begin{gather*}
      \dim \lin{v_{d-1}(A)}
      \geq k + (km+k-1) - \dim \sigma_{k}(Z)
    \end{gather*}
    for the preimage $A  \subset \Pm$ of $q_i(F)  \cup \set{x_1, \dots, x_k}  \subset Z$
    under $v_d: \Pm \simeq Z$ and for each $1 \leq i \leq k$.
  \item 
    If $(k-1)m+k = \binom{m+d}{m}$, then $q_i(F) = Z$.
    In addition, if $(d,m) \neq (3,2)$, then $\dim \lin{v_{d-1}(\Pm)} \geq k + m$.

  \end{enumerate}
\end{proposition}

\begin{remark}\label{rem:cond-k-appear}

  \begin{enumerate}[(a)]
  \item
    Two inequalities
    $\beta=\binom{m+d}{m}-1 < km+k-1$
    and $(k-1)m+k \leq \binom{m+d}{m}$
    are equivalent to
    $\frac{\binom{m+d}{m}}{m+1} < k < \frac{\binom{m+d}{m}}{m+1}+1$;
    it occurs if and only if
    $\frac{\binom{m+d}{m}}{m+1} \notin \NN$
    and $k=\Bigr\lceil \frac{\binom{m+d}{m}}{m+1} \Bigl\rceil$.

  \item 
    In (ii) of Proposition~\ref{prop:dim-qiF-k-1-plane}, if $(d,m)=(3,2)$, then the condition $(k-1)m+k = \binom{m+d}{m}$ gives $k=4$. In this case, $q_i(F) = Z$ is still true (e.g., by Lemma~\ref{lem:q_1F-fin}), but $\dim \lin{v_{d-1}(\Pm)}=\dim \lin{v_{2}(\P^2)}=5$ is $k+m-1$, not greater or equal to $k+m$.
  \end{enumerate}
\end{remark}

To prove Proposition~\ref{prop:dim-qiF-k-1-plane}, we settle two lemmas, Lemmas \ref{lem:codim-vd1Pm} and \ref{lem:R-cap-vPm}; the former one is technical and the latter geometric.

\begin{lemma}\label{lem:codim-vd1Pm}
  Let $d, m, k$ be integers such that $d \geq 3$, $2 \leq m \leq k-2$.
  \begin{enumerate}[(i)]
  \item 
    If $(k-1)m+k < \binom{m+d}{m} < km+k$, 
    then $\binom{m+d-1}{m}-1 -2m-k \geq 0$.
  \item 
    If  $(k-1)m+k = \binom{m+d}{m}$ and $(d,m) \neq (3,2)$,
    then $\binom{m+d-1}{m}-1 -m-k \geq 0$.

  \item 
    If $\binom{m+d}{m} = km+k$, then $\binom{m+d-1}{m}-1 -m-k \geq 0$. 
  \end{enumerate}

\end{lemma}

Note that (iii) of Lemma~\ref{lem:codim-vd1Pm}
is applied in a discussion of the proof of Part (ii) of Theorem~\ref{general_m},
though it is not used in this section.

To show the lemma, we need some calculations as follows.

\begin{remark}\label{rem:small-m}
  (a)
  Let $m=2$ and $km+k > \binom{m+d}{m}$.
  Then
  $(k-1)m+k < \binom{m+d}{m}$ does not occur.
  Otherwise, we have
  $3k-2 < \binom{m+d}{m} = \frac{(d+2)(d+1)}{2} < 3k$,
  and then $\frac{(d+2)(d+1)}{2} = 3k-1$.
  Considering the congruence modulo $3$,
  we have $(d+2)(d+1) \equiv 6k-2 \equiv 1 \pmod{3}$.
  Then the possible values of $(d+2)(d+1)$ are
  \begin{gather*}
    (d+2)(d+1) \equiv
    \begin{cases}
      2\cdot 1 \equiv 2 & (d=0)\\
      0\cdot 2 \equiv 0 & (d=1)\\
      1\cdot 0 \equiv 0 & (d=2),
    \end{cases}
  \end{gather*}
  modulo $3$, which is absurd.

  (b)
  For $d=3,4,5$,
  we calculate numbers $m$ satisfying the conditions
  $\frac{\binom{m+d}{m}}{m+1} \notin \NN$,
  $k=\Bigr\lceil \frac{\binom{m+d}{m}}{m+1} \Bigl\rceil$,
  and
  $(k-1)m+k < \binom{m+d}{m}$.
  For $d=3$, the smallest $m=5$.
  For $d=4$, the smallest $m=3$ and the next smallest $m=7$.
  For $d=5$, the smallest $m=9$.
  The explicit values of $\delta = \binom{m+d-1}{m}-(1+k+2m)$ for them are obtained
  as follows:
  \begin{gather*}
    (d,m,k,\delta)=(3,5,10,0), (4,3,9,4), (4,7,42,63), (5,9,201,495).
  \end{gather*}
\end{remark}

\begin{proof}[Proof of Lemma~\ref{lem:codim-vd1Pm}]
  (i)
  From Remark~\ref{rem:small-m} (a), we may assume $m \geq 3$.
  Let $\delta = \binom{m+d-1}{m}-(1+k+2m)$.
  Since $\binom{m+d-1}{m} = \frac{d}{m+d} \binom{m+d}{m}$, 
  using $(k-1)m+k+1 \leq \binom{m+d}{m}$,
  we have
  \begin{gather*}
    \delta=\binom{m+d-1}{m}-(1+k+2m)
    \geq \frac{1}{m+d}\Bigl(d((k-1)m+k+1)-(m+d)(1+k+2m)\Bigr).
  \end{gather*}
  Setting $k = m+a$ with $a \geq 2$, we have
  \begin{multline}\label{delta-middle-eq}
    d((k-1)m+k+1)-(m+d)(1+k+2m) 
    = m(k(d-1)-3d-1-2m)
    \\
    = m((m+a)(d-1)-3d-1-2m)
    \geq m((m+2)(d-1)-3d-1-2m)
    \\
    =
    m(dm-3m-d-3)
    =m((d-3)(m-1)-6).
  \end{multline}
  Then, $\delta \geq 0$ holds in the following three cases: $d \geq 6$ and $m \geq 3$;
  $d = 5$ and $m \geq 4$; or $d = 4$ and $m \geq 7$.
  In addition, in Remark~\ref{rem:small-m} (b) (see also Remark~\ref{rem:cond-k-appear} (a)),
  we explicitly check that $\delta \geq 0$ if $d = 4$ and $m\leq 6$,
  and check that there is no $k$ in our range if $d = 5$ and $m=3$.

  On the other hand, when $d=3$, (\ref{delta-middle-eq}) implies
  \begin{gather*}
    d((k-1)m+k+1)-(m+d)(1+k+2m) = m(2(m+a)-10-2m) = m(2a-10).
  \end{gather*}
  Hence, $\delta \geq 0$ holds if $d=3$ and $a \geq 5$.
  For $d=3$ and $k = m+a$ with $a =2,3,4$, we have
  \begin{gather*}
    \delta = \frac{m^2-m-2k}{2} = \frac{m(m-3)-2a}{2} > 0
  \end{gather*}
  if $m \geq 5$.
  In addition, in Remark~\ref{rem:small-m} (b),
  we explicitly check that there is no $k$ in our range if $d = 3$ and $m\leq 4$.
  \medskip

  (ii)
  Let $\delta = \binom{m+d-1}{m}-(1+k+m)$.
  In the same way as (i),
  using $(k-1)m+k = \binom{m+d}{m}$, we have
  \begin{gather*}
    \delta=\binom{m+d-1}{m}-(1+k+m)
    = \frac{1}{m+d}\Bigl(d((k-1)m+k)-(m+d)(1+k+m)\Bigr),
  \end{gather*}
  and
  \begin{multline*}
    d((k-1)m+k)-(m+d)(1+k+m) 
    = km(d-1)-2dm-m-m^2- d
    \\
    \geq (m+2)m(d-1)-2dm-m-m^2- d 
    =(d-2)(m^2-1)-2-3m.
  \end{multline*}

  If $d\geq 3$ and $m \geq 4$, then
  since $(d-2)(m^2-1)-2-3m \geq m(m-3)-3 \geq 1$, it holds $\delta \geq 0$.

  If $d\geq 4$ and $m = 3$,
  then since $(d-2)(m^2-1)-2-3m \geq 5$,
  it similarly holds
  $\delta \geq 0$.
  \medskip

  On the other hand, if $d=3$ and $m=3$, then
  $(k-1)m+k = \binom{m+d}{m}$
  implies
  \begin{gather*}
    4k-3 = \frac{(d+3)(d+2)(d+1)}{3!} = 20;
  \end{gather*}
  this case does not occur since $k$ cannot be an integer.

  If $d\geq 4$ and $m=2$, then
  $3k-2 = \frac{(d+2)(d+1)}{2}$.
  In this case, 
  $\delta = \frac{(d+1)d}{2}-k-3 \geq 0$,
  because of
  \begin{gather*}
    3\left(\frac{(d+1)d}{2}-k-3\right) \geq \frac{3(d+1)d}{2}-\frac{(d+2)(d+1)}{2}-11
    = d^2-12 \geq 4.
  \end{gather*}

  (iii)
  Next, we assume $\binom{m+d}{m} = km+k$.
  Then,
  \begin{gather*}
    \delta = \binom{m+d-1}{m}-(1+k+m) = \frac{1}{m+d}\Bigl(d(km+k)-(m+d)(1+k+m)\Bigr).
  \end{gather*}
  Using $k = m+a = (m+1)+(a-1)$ with $a \geq 2$, we have
  \begin{multline*}
    d(km+k)-(m+d)(1+k+m)
    \\
    = (d-1)km-(m+d)(m+1)
    = (d-1)m(m+1) + (d-1)m(a-1) -(m+d)(m+1)
    \\
    = (m+1)((d-1)m-(m+d)) + (d-1)m(a-1)
    \\
    = (m+1)((d-2)(m-1)-2) + (d-1)m(a-1).
  \end{multline*}
  If $d \geq 3, m \geq 2$ and $(d,m) \neq (3,2)$, then we have $\delta \geq 0$.
  If $(d,m) = (3,2)$, then since
  \begin{gather*}
    (m+1)((d-2)(m-1)-2) + (d-1)m(a-1) = -3+4(a-1) \geq 1,
  \end{gather*}
  we also have $\delta \geq 0$.
\end{proof}

The next lemma concerns a general fact on linear sections of Veronese varieties, which has an independent interest itself.

\begin{lemma}\label{lem:R-cap-vPm}
  Let $k,m \geq 2$.
  Let $x_1', \dots, x_k' \in \Pm$ be general $k$ points,
  let $v_e: \Pm \rightarrow \PP^{\beta_e = \binom{m+e}{m}-1}$ be
  the $e$-uple Veronese embedding of $\Pm$,
  and let
  $M = \lin{v_{e}(x_1'), v_{e}(x_2'),\dots, v_{e}(x_k')} \subset \PP^{\beta_{e}}$.
  Then, for any $k$-plane $R  \subset \PP^{\beta_{e}}$ 
  containing the ($k-1$)-plane $M$,
  the following holds.

  \begin{enumerate}[(i)]
  \item
    Assume $k \leq \beta_e - m$, and assume that
    there is
    a curve $C  \subset {R \cap v_{e}(\PP^m)}$ passing through $v_{e}(x_1')$.
    Then it holds that
    \begin{gather}\label{eq:C-in-lin-x-TTxk}
      R \subset \lin{{v_{e}(x_2'), \dots, v_{e}(x_{k}')}, {\TT_{v_{e}(x_1')}v_{e}(\PP^m)}}.
    \end{gather}

  \item
    Assume $e \geq 3$ and $k\le\beta_e-2m$. Then
    we have
    \begin{gather*}
      \dim_{v_{e}(x_1')} ({R \cap v_{e}(\PP^m)}) = 0,
    \end{gather*}
    where the left hand side means dimension of component(s) passing through $v_{e}(x_1')$. In particular, the set of a point
    $\set{v_{e}(x_1')}$ is an irreducible component of
    ${R \cap v_{e}(\PP^m)}$.

  \item 
    Assume $e=2$ and $k-1\le\beta_2-2m$. Assume that there is a curve $C  \subset R \cap v_2(\Pm)$ such that $v_2(x_1') \in C$.
    Then, for any irreducible subset $D \subset R \cap v_{2}(\PP^m)$,
    it holds $D \subset v_2(\lin{x_1', x_l'})$
    for some $l = 2, \dots, k$,
    where $v_2(\lin{x_1', x_l'})$ is a conic curve in $\PP^{\beta_2}$ given as the image of the line $\lin{x_1', x_l'}  \subset \Pm$,
  \end{enumerate}
\end{lemma}

\begin{proof}
  Let $C  \subset {R \cap v_{e}(\PP^m)}$ be a curve passing through $v_{e}(x_1')$.
  Let 
  \begin{gather*}
    \pi = \pi_{\lin{v_{e}(x_2'), \dots, v_{e}(x_{k}')}}:
    \PP^{\beta_{e}} \dashrightarrow \PP^{\beta_{e} -k+1}
  \end{gather*}
  be the linear projection from the ($k-2$)-plane
  $\lin{v_{e}(x_2'), \dots, v_{e}(x_{k}')}$.
  If $k\le\beta_e-m$,
  then the generalized trisecant lemma (\cite[Proposition 1.4.3]{Rus}) implies
  \begin{gather*}
    M \cap v_{e}(\PP^m)
    = \set{v_{e}(x_1'), v_{e}(x_2'), \dots, v_{e}(x_k')}.
  \end{gather*}
  In particular, $\dim (M \cap v_{e}(\PP^m)) = 0$ and $C  \not\subset M$.
  We have  $\pi(v_{e}(x_1')) \in \overline{\pi(C)} \subset \PP^{\beta_{e} -k+1}$
  because of $v_{e}(x_1') \in C$.
  If $C$ is contracted to a point under $\pi$, then $\overline{\pi(C)} = \pi(v_{e}(x_1'))$, which means that $C \subset M$, a contradiction.
  Hence $\overline{\pi(C)}$ must be a curve.

  For the $k$-plane $R \subset \PP^{\beta_e}$, which contains the ($k-2$)-dimensional center of $\pi$,
  the image $\overline{\pi(R)}$ is a line in $\PP^{\beta_{e} -k+1}$.
  Thus $\overline{\pi(C)} = \overline{\pi(R)}$. Moreover,
  it follows
  \begin{gather*}
    \overline{\pi(C)}=\overline{\pi(R)} = \TT_{\pi(v_{e}(x_{1}'))}\overline{\pi(R)} \subset \TT_{\pi(v_{e}(x_{1}'))} \overline{\pi(v_{e}(\PP^m))}
  \end{gather*}
  where, by generic smoothness, the right hand side is equal to $\pi({\TT_{v_{e}(x_{1}')}v_{e}(\PP^m)})$
  since $x_{1}' \in \Pm$ is general.
  It follows that
  $R$ is contained in the preimage of $\pi({\TT_{v_{e}(x_{1}')}v_{e}(\PP^m)})$,
  which implies the inclusion~(\ref{eq:C-in-lin-x-TTxk}) of (i).
  \medskip

  The condition $k\le\beta_e-m$ holds if $k$ or $k-1$ is $\le\beta_e-2m$.
  Next, we consider a tangential projection
  \begin{gather*}
    \pi_{\TT_{v_{e}(x_{1}')}v_{e}(\PP^m)}:
    \PP^{\beta_{e}} \dashrightarrow \PP^{\beta_{e} -m-1}
  \end{gather*}
  from the $m$-plane ${\TT_{v_{e}(x_{1}')}v_{e}(\PP^m)} \subset \PP^{\beta_{e}}$,
  and its restriction $\tilde{\pi} = \pi_{\TT_{v_{e}(x_{1}')}v_{e}(\PP^m)}|_{v_{e}(\PP^m)}$
  on $v_{e}(\PP^m)$.
  Note that the Veronese variety $v_{e}(\PP^m)$ and any embedded tangent space to $v_{e}(\PP^m)$ intersect only at one point;
  in particular, $v_{e}(\PP^m) \cap {\TT_{v_{e}(x_{1}')}v_{e}(\PP^m)}= \set{v_{e}(x_{1}')}$.
  If $R \subset \Pb$ satisfies (\ref{eq:C-in-lin-x-TTxk}), we have
  \begin{gather}\label{tilde_pi_C}
    \tilde{\pi}(R \cap v_{e}(\PP^m))  \subset \tilde{\pi}(v_e(\Pm)) \cap \lin{\tilde{\pi}(v_{e}(x_2')), \dots, \tilde{\pi}(v_{e}(x_{k}'))}.
  \end{gather}
  By Terracini's lemma,
  for general $z \in v_{e}(\Pm)$,
  the linear variety $\lin{\TT_{v_e(x_1')}v_{e}(\Pm), \TT_{z}v_{e}(\Pm)}$
  coincides with an embedded tangent space to $\sigma_2(v_{e}(\PP^m))$,
  and is of dimension $\dim\sigma_2(v_{e}(\PP^m))$.
  Then,
  \begin{gather*}
    \pi_{\TT_{v_{e}(x_{1}')}v_{e}(\PP^m)}(\TT_{z}v_{e}(\Pm))
    = \TT_{\tilde{\pi}(z)}\tilde{\pi}(v_{e}(\Pm)) \subset \PP^{\beta_{e} -m-1}
  \end{gather*}
  is of dimension $\dim\sigma_2(v_{e}(\PP^m)) - m - 1$.
  It follows
  \begin{gather}\label{eq:dim-im-tilde-pi}
    \dim\tilde{\pi}(v_{e}(\Pm)) = \dim\sigma_2(v_{e}(\PP^m)) - m - 1.
  \end{gather}

  Let $e \geq 3$ and $k\le\beta_e-2m$. Suppose that $\dim_{v_{e}(x_1')} ({R \cap v_{e}(\PP^m)}) > 0$, which means the existence of a curve $C$ satisfying the condition of (i).
  Since
  \begin{gather*}
    \codim(\tilde{\pi}(v_{e}(\Pm)), \PP^{\beta_{e} -m-1})-(k-1)\ge(\beta_e-2m-1)-(k-1) \geq 0,
  \end{gather*}
  again the trisecant lemma implies that the right hand side in (\ref{tilde_pi_C}) is only the set of $k-1$ points
  ${\tilde{\pi}(v_{e}(x_2')), \dots, \tilde{\pi}(v_{e}(x_{k}'))}$.
  Thus, each irreducible subset $D \subset R \cap v_{e}(\PP^m)$ satisfies
  \begin{gather}\label{eq:C-contr-to-pt}
    \tilde{\pi}(D) = \tilde{\pi}(v_{e}(x_{l}'))
  \end{gather}
  for some $l = 2, \dots, k$. (At least, taking $D=C$, we exactly have (\ref{eq:C-contr-to-pt}).)
  From $e \geq 3$, $\dim \sigma_2(v_{e}(\PP^m)) = 2m+1$ (i.e., non-defective).
  In this case, by (\ref{eq:dim-im-tilde-pi}),
  $\dim\tilde{\pi}(v_{e}(\Pm)) = m$, i.e.,
  the map $\tilde{\pi}$ must be generically finite. Then we reach a contradiction since $v_{e}(x_{l}')$ is a general point. It implies that $\dim_{v_{e}(x_1')} ({R \cap v_{e}(\PP^m)}) = 0$.
  \medskip

  Finally, we consider the case of $e=2$.
  Then $\dim \sigma_2(v_{2}(\PP^m)) = 2m$ for $m \geq 2$ (i.e., defective),
  and by (\ref{eq:dim-im-tilde-pi}), $\dim\tilde{\pi}(v_{2}(\Pm)) = m-1$.
  It means that
  the tangential projection
  $\tilde{\pi}: v_2(\Pm) \dashrightarrow \tilde{\pi}(v_2(\Pm))$
  has fibers of dimension $1$.
  Moreover, as in Remark~\ref{rem:def-y'} (a) we know that
  \begin{gather*}
    \pi_{\TT_{v_{2}(x_{1}')}v_{2}(\PP^m)} : \PP^{\beta_{2}=\frac{(m+2)(m+1)}{2} -1} \dashrightarrow \PP^{\beta_{2} -m-1=\frac{(m+1)m}{2} -1}
  \end{gather*}
  satisfies the commutative diagram:
  \begin{equation*}
    \begin{tikzcd}
      \Pm \arrow[r,hook,"v_2"] \arrow[d, dashed, "\pi_{x_{1}'}"'] & \PP^{\frac{(m+2)(m+1)}{2} -1} \arrow[d, dashed, "\pi_{\TT_{v_{2}(x_{1}')}v_{2}(\PP^m)}"] \\
      \PP^{m-1} \arrow[r,hook,"v_2"] & \PP^{\frac{(m+1)m}{2} -1}\quad,
    \end{tikzcd}
  \end{equation*}
  where $\pi_{x_{1}'}: \Pm \dashrightarrow \PP^{m-1}$ is
  the linear projection from $x_{1}'$, and
  $v_2: \PP^{m-1} \hookrightarrow \PP^{\frac{(m+1)m}{2} -1}$
  is the Veronese embedding of $\PP^{m-1}$.
  Then
  \begin{gather*}
    \dim(\tilde{\pi}(v_{2}(\Pm))) = \dim(v_{2}(\PP^{m-1})) = m-1.
  \end{gather*}
  If $k-1\le\beta_2-2m$, then
  \begin{gather*}
    \codim(\tilde{\pi}(v_{2}(\Pm)), \PP^{\beta_{2} -m-1})-(k-1)\ge(\beta_2-2m)-(k-1) \geq 0;
  \end{gather*}
  hence
  the trisecant lemma implies $\tilde{\pi}(D) = \tilde{\pi}(v_{2}(x_{l}'))$ for some $l = 2, \dots, k$ as we discussed for (\ref{eq:C-contr-to-pt}).

  In the diagram above, for any $y' \in \lin{x_1', x_l'} \subset \Pm$ with $y' \neq x_1'$,
  we have
  \begin{gather*}
    \tilde{\pi}(v_2(y')) = v_2(\pi_{x_1'}(y')) = v_2(\pi_{x_1'}(x_l')) = \tilde{\pi}(v_2(x_l'));
  \end{gather*}
  indeed, $\tilde{\pi}^{-1}(\tilde{\pi}(v_2(x_l'))) = v_2(\lin{x_1', x_l'})$.
  Since $\tilde{\pi}(D) = \tilde{\pi}(v_{2}(x_{l}'))$, we have $D \subset v_2(\lin{x_1', x_l'})$.
\end{proof}

We give one calculation before proving Proposition~\ref{prop:dim-qiF-k-1-plane}.

\begin{remark}\label{d3-fib-dim-geq-2}
  For $d=3$ and $m \neq 2$, if
  $\mu_0 = \frac{\binom{m+d}{m}}{m+1} \notin \NN$
  and $k=\lceil \mu_0 \rceil$
  (as under the conditions of
  Proposition~\ref{prop:dim-qiF-k-1-plane} and Remark~\ref{rem:cond-k-appear} (a)),
  then we have $(km+k-1) - \beta_d \neq 1$.
  The reason is as follows. First, we may write
  $\mu_0 = {(m+3)(m+2)}/6 = M/3$ for some $M \in \NN$
  since $(m+3)(m+2)$ is divisible by $2$. In addition, dividing $M$ by $3$ with remainder,
  we have $M = 3Q + R$ for a quotient $Q$ and a remainder $R$.
  Since $\mu_0=M/3 \notin \NN$, $R$ must be $1$ or $2$.
  In this setting,
  $k=\lceil M/3\rceil = Q+1$. It follows that $(km+k-1) - \beta_d$ is equal to
  \begin{align*}
     & (km+k) - \tbinom{m+3}{3} = (m+1)\left({k-\tfrac{(m+3)(m+2)}{6}}\right) = (m+1)\left((Q+1)-(Q+\tfrac{R}{3})\right)
    = (m+1)\cdot \tfrac{3-R}{3}.
  \end{align*}
  If $(km+k-1) - \beta_d = 1$, then $3 = (m+1)(3-R)$.
  Since $3-R$ is $1$ or $2$ and $m \in \NN$, we get $m=2$.
\end{remark}
\smallskip

\begin{proof}[Proof of Proposition~\ref{prop:dim-qiF-k-1-plane}]
  (i)
  For simplicity, we set $i=1$; then $x_1 \in q_1(F)  \subset Z=v_d(\Pm)$.
  For $s=\dim \sigma_{k}(Z)$,
  an irreducible component $F$ of $p^{-1}(a)$
  is of dimension $\dim J - s=(km+k-1) - s$.
  From Lemma~\ref{lem:q_1F-fin}, we have $\dim q_1(F)=(km+k-1) - s$.

  Let $q_1(F)'  \subset \Pm$ be the preimage of $q_1(F)  \subset Z$ under
  $v_d: \Pm \simeq Z$, and let
  \begin{gather*}
    A = q_1(F)'  \cup \set{x_1', \dots, x_k'}  \subset \Pm.
  \end{gather*}
  Let $v_{d-1}: \Pm \rightarrow \PP^{\beta_{d-1}}$
  be the ($d-1$)-uple Veronese embedding.
  Then the ($k-1$)-plane
  \begin{gather*}
    \lin{v_{d-1}(x_1'), \dots, v_{d-1}(x_k')}  \subset \PP^{\beta_{d-1}}
  \end{gather*}
  is contained in the linear variety
  \begin{gather*}
    \lin{v_{d-1}(A)} = \lin{v_{d-1}(q_1(F)')  \cup \set{v_{d-1}(x_1'), \dots, v_{d-1}(x_k')}},
  \end{gather*}
  and is of codimension $c = \dim \lin{v_{d-1}(A)} - (k-1)$.
  By Lemma~\ref{lem:codim-vd1Pm} (i),
  $\beta_{d-1} -2m-k \geq 0$.
  So, by the generalized trisecant lemma,
  \begin{gather*}
    v_{d-1}(\Pm) \cap \lin{v_{d-1}(x_1'), \dots, v_{d-1}(x_k')}
    = \set{v_{d-1}(x_1'), \dots, v_{d-1}(x_k')}.
  \end{gather*}
  In particular,
  \begin{gather*}
    v_{d-1}(q_1(F)') \cap \lin{v_{d-1}(x_1'), \dots, v_{d-1}(x_k')}
    \subset \set{v_{d-1}(x_1'), \dots, v_{d-1}(x_k')}.
  \end{gather*}
  Since $\dim q_1(F)' \geq 1$, we may take a point $y' \in q_1(F)'$
  such that
  \begin{gather*}
    v_{d-1}(y') \notin \lin{v_{d-1}(x_1'), \dots, v_{d-1}(x_k')}.
  \end{gather*}

  Assume $d \geq 4$. 
  Applying Lemma~\ref{lem:R-cap-vPm} (ii) to
  \begin{gather*}
    R = \lin{v_{d-1}(x_1'), \dots, v_{d-1}(x_k'), v_{d-1}(y')}
    \subset \lin{v_{d-1}(A)},
  \end{gather*}
  we have
  \begin{gather*}
    \dim_{v_{d-1}(x_1')} ({R\cap v_{d-1}(\PP^m) }) = 0.
  \end{gather*}
  In particular,
  $\dim_{v_{d-1}(x_1')} (R\cap v_{d-1}(q_1(F)')) = 0$.
  Regarding it as an intersection of two irreducible subvarieties in $\lin{v_{d-1}(A)}$,
  we deduce that every irreducible component
  of $R\cap v_{d-1}(q_1(F)')$
  is of dimension $\geq \dim(v_{d-1}(q_1(F)')) - (c-1)$.
  Hence
  \begin{gather*}
    \dim(\lin{v_{d-1}(A)}) \geq k+\dim(v_{d-1}(q_1(F)')) = k+ (km+k-1) - s.
  \end{gather*}
  
  Next, let us consider the case of $d = 3$.
  For $l = 2, \dots, k$,
  since
  ${v_{2}(\lin{x_1', x_l'})}  \subset \PP^{\beta_{2}}$ is a conic,
  it follows that $\lin{v_{2}(\lin{x_1', x_l'})}$ is a $2$-plane, which is equal to
  $\lin{v_{2}(x_1'), v_{2}(x_l'), z}$ for some $z \in \PP^{\beta_{2}}$.
  Then
  \begin{gather*}
    \lin{v_{2}(x_1'), \dots, v_{2}(x_k'), v_{2}(\lin{x_1', x_l'})} = \lin{v_{2}(x_1'), \dots, v_{2}(x_k'), z}
  \end{gather*}
  is a linear subvariety of dimension $\leq k$.
  Since $v_{2}(q_1(F)') \cap \lin{v_{2}(x_1'), \dots, v_{2}(x_k')}$ is empty
  or is a set of points,
  the intersection
  \begin{gather*}
    v_{2}(q_1(F)') \cap \lin{v_{2}(x_1'), \dots, v_{2}(x_k'), v_{2}(\lin{x_1', x_l'})}
  \end{gather*}
  is of dimension $\leq 1$.
  On the other hand, since $m\neq 2$ by Remark~\ref{rem:small-m} (a),
  we have
  \begin{gather*}
    \dim q_1(F)'=(km+k-1)-\beta\geq 2
  \end{gather*}
  as in Remark~\ref{d3-fib-dim-geq-2}.
  For the union
  \begin{gather*}
    W=\bigcup_{l = 2, \dots, k} v_{2}(q_1(F)') \cap \lin{v_{2}(x_1'), \dots, v_{2}(x_k'), v_{2}(\lin{x_1', x_l'})}  \subset \PP^{\beta_{2}},
  \end{gather*}
  we see that $q_1(F)'  \setminus v_2^{-1}(W)\neq\emptyset$ and may take 
  $y' \in q_1(F)'  \setminus v_2^{-1}(W)$.
  
  Let $R = \lin{v_{2}(x_1'), \dots, v_{2}(x_k'), v_{2}(y')}
   \subset \lin{v_{2}(A)}$ and suppose that
  $\dim_{v_{2}(x_1')} (R\cap v_{2}(q_1(F)')) > 0$, that is to say,
  there is a curve $C  \subset R\cap v_{2}(q_1(F)')$ containing ${v_{2}(x_1')}$.
  Taking $D=C$ and applying Lemma~\ref{lem:R-cap-vPm} (iii),
  we have $C = v_{2}(\lin{x_1', x_l'})$ for some $l > 1$.
  If
  \begin{gather*}
    \dim\lin{v_{2}(x_1'), \dots, v_{2}(x_k'), v_{2}(\lin{x_1', x_l'})}=k,
  \end{gather*}
  then $R = \lin{v_{2}(x_1'), \dots, v_{2}(x_k'), v_{2}(\lin{x_1', x_l'})}$,
  contradicting to the definition of $W$ and our choice of $y'$.
  Else if
  \begin{gather*}
    \dim\lin{v_{2}(x_1'), \dots, v_{2}(x_k'), v_{2}(\lin{x_1', x_l'})}=k-1,
  \end{gather*}
  then
  $C=v_{2}(\lin{x_1', x_l'}) \subset v_{2}(q_1(F)') \cap \lin{v_{2}(x_1'), \dots, v_{2}(x_k')}$,
  also contradicting that the intersection is of dimension at most $0$.

  Hence $\dim_{v_{2}(x_1')} (R\cap v_{2}(q_1(F)')) = 0$. Then, in the same way as above,
  we have
  \begin{gather*}
    \dim(\lin{v_{2}(A)}) \geq k+ (km+k-1) - s.
  \end{gather*}

  (ii) 
  In the case when
  \begin{gather*}
    (k-1)m+k = \tbinom{m+d}{m},
  \end{gather*}
  we have $km+k-1-s \geq m$.
  It follows from Lemma~\ref{lem:q_1F-fin} and $\Pm\simeq Z$ that
  $q_i(F) = Z$. From Lemma~\ref{lem:codim-vd1Pm} (ii),
  if $(d,m) \neq (3,2)$, then $\PP^{\beta_{d-1}}=\lin{v_{d-1}(\Pm)}$
  is of dimension $\geq k + m$.
\end{proof}

We end this subsection by making one more important remark on the case when $\sigma_k(v_d(\P^m)) \subsetneq \PP^{\beta_d}$ is secant defective,
which will be used in the proof of of Part (ii) of Theorem~\ref{general_m2}.

\begin{remark}[Estimate in defective cases]\label{dim_fib_defective}
  For four defective cases
  \begin{gather*}
    (k,d,m)=(7,3,4),(5,4,2),(9,4,3),(14,4,4),
  \end{gather*}
  similarly as in Proposition~\ref{prop:dim-qiF-k-1-plane}, we can have an estimation 
  \[
    \dim\lin{v_{d-1}(A)}\ge k+\delta,
  \]
  where $A = v_d^{-1}(q_1(p^{-1}(a)))  \subset \Pm$, the preimage of the entry locus of $a$, and $\delta$ is the secant defect of $\sigma_k(v_d(\Pm))$; here $A$ is $\delta$-equidimensional, the $k$ general points $x_1',\dots,x_k'\in\Pm$ are contained in $A$, and it is well known that $\delta=2$ when $(k,d,m)=(9,4,3)$ and $\delta=1$ in all the other defective cases.

  For three cases
  \begin{gather*}
    (k,d,m)=(5,4,2),(9,4,3),(14,4,4),
  \end{gather*}
  we see that $\beta_{d-1} -2m \geq k$ and
  \begin{gather*}
    v_{d-1}(A) \cap \lin{v_{d-1}(x_1'), \dots, v_{d-1}(x_k')}
    \subset \set{v_{d-1}(x_1'), \dots, v_{d-1}(x_k')}
  \end{gather*}
  by the trisecant lemma so that we may take $y' \in A$
  such that $v_{d-1}(y') \notin \lin{v_{d-1}(x_1'), \dots, v_{d-1}(x_k')}$.
  By Lemma~\ref{lem:R-cap-vPm} (ii), we get
  \begin{gather*}
    \dim_{v_{d-1}(x_1')} (R\cap v_{d-1}(\PP^m)) = 0,
  \end{gather*}
  where $R = \lin{v_{d-1}(x_1'), \dots, v_{d-1}(x_k'), v_{d-1}(y')}$.
  Thus, by the intersection argument in $\lin{v_{d-1}(A)}$ (similar to Proposition~\ref{prop:dim-qiF-k-1-plane} (i)), we derive the estimation
  $$\dim\lin{v_{d-1}(A)}\ge \dim R+\dim v_{d-1}(A)= k+\delta~.$$

  For the remaining case $(k,d,m)=(7,3,4)$,
  it holds $\beta_{d-1} -2m = k-1$, 
  and we still can claim that
  \begin{gather*}
    \dim(\lin{v_{2}(A)})\ge k+\delta=7+1=8
  \end{gather*}
  as follows.
  For the $6$-dimensional subspace $M=\lin{v_{2}(x_1'), \dots, v_{2}(x_7')} \subset \lin{v_{2}(A)}$,
  the trisecant lemma implies
  \begin{gather*}
    M\cap v_2(A) \subset M\cap v_2(\PP^4) = \set{v_{2}(x_1'), \dots, v_{2}(x_7')},
  \end{gather*}
  $0$-dimensional intersection.
  Then $\dim(\lin{v_{2}(A)}) \geq 7$
  (otherwise, we get $M=\lin{v_{2}(A)}$ so that $M\cap v_2(A)=v_2(A)$,  a contradiction).
  Suppose that
  \begin{gather*}
    \dim(\lin{v_{2}(A)})=7,
  \end{gather*}
  and set $R=\lin{v_{2}(A)}$.
  We take the irreducible
  components of the $1$-equidimensional closed set $A$ as $A = \bigcup_{j=1}^{s} A_j$.
  Note that $v_2(A_j) \subset R \cap v_2(\Pm)$.
  Since $x_1' \in A$, there is a curve $A_{j_0}$ containing $x_1'$.
  Taking $C = v_2(A_{j_0})$ and applying Lemma~\ref{lem:R-cap-vPm} (iii),
  for any $j$ with $1 \leq j \leq s$,
  we have $v_2(A_j) = v_2(\langle x_1',x_{l_j}'\rangle)$ for some $l_j=2,\dots,k$.
  It is equivalent to $A_j = \langle x_1',x_{l_j}'\rangle$, a line in $\P^4$; in particular, $x_1' \in A_j$. In the same way, $A_j$ must contain $x_1', \dots, x_7'$.
  But this is a contradiction, because these points are chosen as general $7$ points in $\P^4$. Thus, $\dim(\lin{v_{2}(A)})\ge8$.
\end{remark}

\subsection{Estimate for the linear span of tangents moving along a subsecant variety}\label{subsec_est}

First, we give the following explicit description of the embedded tangent space
$\TT_x v_d(\Pn) \subset \Pa$ to $v_d(\Pn)$ at a point $x$ in $v_d(\Pn)$ or $v_d(\Pm)$.
Note that it is related to computations of Gauss maps (see \cite{FI}).

Recall that $\mo[t]{e}$ denotes the set of monomials $f \in \CC[t_1, \dots, t_m]$ with $\deg f \leq e$.
Then $1 \in \mo[t]{e}$ as the monomial of degree $0$.
As mentioned in Remark~\ref{rem:def-y'}, as changing homogeneous coordinates
$t_0, t_1, \dots, t_m, u_1, u_2, \dots, u_{m'}$ on $\Pn$
with $m' = n-m$,
we may assume that
$\Pm$ is the zero set of $u_1, \dots, u_{m'}$. On the affine open subset $\set{t_0 \neq 0}$,
the Veronese embedding $v_d: \Pn \rightarrow \Pa$
is parameterized by monomials of $\CC[t_1, \dots, t_m, u_1, \dots, u_{m'}]$ of degree $\leq d$. So it is expressed as
\begin{gather}\label{eq:coordi-t0=1-tu}
 [\, \mo[t]{d} : u_1 \cdot \mo[t]{d-1} :\quad\cdots\quad : u_{m'} \cdot \mo[t]{d-1} : \quad*\quad \,]\,,
\end{gather}
where $u_i \cdot \mo[t]{d-1}$ means
\begin{gather*}
  \set{u_i f \mid f \in \mo[t]{d-1}} = (u_i : u_i t_1 : u_i t_2 : \dots : u_i t_m^{d-1}),
\end{gather*}
and ``$*$'' means the remaining monomials.

Let $x = v_d(x')$ with $x' \in \set{t_0 \neq 0}  \subset \Pn$.
Then $\TT_x v_d(\Pn)  \subset \Pa$
coincides with 
\begin{equation}\label{eq:expr-TTx-0}
 \textit{$n$-plane spanned by the $(n+1)$ points corresponding to the row vectors of the matrix}
 \end{equation}
 {\small
 \begin{align*}  
   &\begin{bmatrix}
      v_d
      \\      
      \partial v_d / \partial t_1
      \\
      \vdots
      \\
      \partial v_d / \partial t_m
      \\
      \partial v_d / \partial u_1
      \\
      \vdots
      \\
      \partial v_d / \partial u_{m'}
    \end{bmatrix}
     (x')  
     =
     \left[\begin{array}{cccc:ccc}
             \mo[t]{d} & u_1 \cdot \mo[t]{d-1} & \dots & u_{m'} \cdot \mo[t]{d-1} && *&
             \\      
             (\mo[t]{d})_{t_1} & u_1 \cdot (\mo[t]{d-1})_{t_1} & \dots & u_{m'} \cdot (\mo[t]{d-1})_{t_1} && *&
             \\
             \vdots & \vdots && \vdots & &\vdots&
             \\
             (\mo[t]{d})_{t_m} & u_1 \cdot (\mo[t]{d-1})_{t_m} & \dots & u_{m'} \cdot (\mo[t]{d-1})_{t_m} && *&
             \\
             \mathsf{O} & \mo[t]{d-1} & \dots & \mathsf{O} && *&
             \\
             \vdots && \ddots && &\vdots&
             \\
             \mathsf{O} & \mathsf{O} & \dots & \mo[t]{d-1} && *&
           \end{array}\right] (x')~
 \end{align*}
 }%
using (\ref{eq:coordi-t0=1-tu}), where $(\mo[t]{e})_{t_i}$ means $\set{ \partial f / \partial t_i \mid f \in \mo[t]{e}}$ 
and $\mathsf{O}$ is a zero matrix with suitable size.

In particular, in case of $x' \in \Pm = \set{u_1 = \dots = u_{m'} = 0}$, we see that the matrix is of the form
\begin{equation}\label{eq:expr-TTx}
  \left[\begin{array}{cccc:cccccccc}
    \mo[t]{d} & \mathsf{O} & \dots  & \mathsf{O} & & &&&\mathsf{O}&&&
    \\      
    (\mo[t]{d})_{t_1} & \mathsf{O} & \dots  & \mathsf{O} && &&& \mathsf{O}&&&
    \\
    \vdots & \vdots && \vdots && &&& \vdots&&&
    \\
    (\mo[t]{d})_{t_m} & \mathsf{O} & \dots & \mathsf{O}& & &&& \mathsf{O}&&&
    \\
    \mathsf{O} & \mo[t]{d-1} & \dots & \mathsf{O} & & &&&\mathsf{O}&&&
    \\
    \vdots && \ddots &&& &&&\vdots&&&
    \\
    \mathsf{O} & \mathsf{O} & \dots & \mo[t]{d-1} & &&&& \mathsf{O}&&&
  \end{array}\right] (x')~.
\end{equation}

As a consequence, we settle a key proposition which estimates a lower bound of the dimension of the linear span of moving embedded tangent spaces along a subset of a given $\Pm$.

\begin{proposition}\label{prop:dim-linTT-geq-vdAA}
  Let $v_d: \PP^n \rightarrow \Pa$ be the $d$-uple Veronese embedding.
  For an $m$-plane $\Pm  \subset \Pn$, for a (possibly reducible) subset $A  \subset \Pm$,
  and for a linear variety $\Lambda  \subset \lin{v_d(\Pm)}$,
  the dimension of the linear variety
  \[
    \lin{\Lambda  \cup \bigcup_{x \in v_d(A)}\TT_x (v_d(\Pn))}  \subset \Pa
  \]
  is greater than or equal to
  \begin{equation}\label{eq:sum-vdAA}
    \dim \lin{\Lambda  \cup v_{d}(A)} + (n-m)\big\{1 + \dim\lin{v_{d-1,m}(A)}\big\},
  \end{equation}
  where $v_{e,m}: \Pm \rightarrow \PP^{\binom{m+e}{m}-1}$ is the $e$-uple Veronese embedding of $\Pm$.
\end{proposition}

\begin{proof}
  For a given $A  \subset \Pm$, we consider
  \[
    B_0 = v_d (A),
    B_1 = (\partial v_d / \partial u_1) (A), \,\dots\,,
    B_{m'} = (\partial v_d / \partial u_{m'}) (A)
  \]
  as subsets in $\P^N$,
  where $B_i$ is embedded by the parameterization of $(m+1+i)$-th row
  of the matrix of (\ref{eq:expr-TTx-0}) for $1\le i\le m'$.
  Note that for the homogeneous coordinates $[w_0: \dots: w_N]$ on $\PN$
  corresponding to (\ref{eq:coordi-t0=1-tu}),
  $\lin{v_d(\Pm)}=\PP^{\beta=\binom{m+d}{m}-1}$ is the zero set of $w_{\beta+1}, \dots, w_N$,
  and $\Lambda  \cup B_0$ is contained in the set.

  Since $A  \subset\set{u_1 = \dots = u_{m'} = 0}$,
  it follows from (\ref{eq:expr-TTx}) that the linear variety
  \begin{equation}\label{eq:lin-BB-inPa}
    \lin{\Lambda  \cup B_0, B_1, \,\dots\,, B_{m'}}  \subset \Pa
  \end{equation}
  is of dimension $\dim(\lin{\Lambda  \cup B_0}) + \dim(\lin{B_1}) + \dots + \dim (\lin{B_{m'}}) + m'$.
  
  Again, by (\ref{eq:expr-TTx}) 
  we see that
  $B_0 \simeq v_{d}(A)$ and $B_i \simeq v_{d-1,m}(A)$ for $1 \leq i \leq m'$.
  Since the linear variety~(\ref{eq:lin-BB-inPa}) is contained in
  $\lin{\Lambda  \cup \bigcup_{x \in v_d(A)}\TT_x (v_d(\Pn))}$,
  we have the assertion.
\end{proof}

\section{Case of $m=1$}\label{sect_m=1}

\subsection{Symmetric flattening and conormal space computation}
\label{sec:symm-flatt-conorm}

For the proof of Theorem~\ref{thm_m1}, we begin with some preliminaries on equations for secant varieties and conormal space computation via known sets of equations, whereas we mainly adopt the geometric viewpoint and techniques for the $m\ge2$ case in \S\ref{sect_m>=2}. 

Let $V$ be a $(n+1)$-dimensional $\CC$-vector space $\CC\langle x_0,x_1,\dots,x_n\rangle$. Let $f\in S^d V$ be a homogeneous polynomial of degree $d$ (or $d$-\ti{form}) and $[f]$ be the corresponding point in $\PP S^d V$. In this paper, we frequently abuse notations as denoting both a $d$-form in $S^d V$ and the point in $\PP S^d V$ just by $f$. For the Veronese variety $v_d(\P V)$, we have a natural 1-1 correspondence between points of the ambient space $\lin{v_d(\P V)}$ and equivalent classes of degree $d$-forms in $S=\CC[x_0,x_1,\dots,x_n]$. First of all, let us recall some notions related to this correspondence.

Given a form $f$ of degree $d$, the minimum number of linear forms $l_{i}$ needed to write $f$ as a sum of $d$-th powers is the so-called (Waring) \textit{rank} of $f$ and we denote it by $\rank(f)$. Note that one can define $\rank([f])$ by $\rank(f)$, because this rank is invariant under non-zero scaling. The (Waring) \textit{border rank} is given by the same notion in the limiting sense. In other words, if there is a family $\{f_{\epsilon}\mid \epsilon >0 \}$ of polynomials with constant rank $r$ and $\lim_{\epsilon \to 0}f_{\epsilon} = f$, then we say that $f$ has border rank at most $r$. The minimum such $r$ is called the border rank of $f$ and we denote it again by $\brank(f)$. Note that by definition $\sigma_k(v_d(\P V))$ is the variety of homogeneous polynomials $f$ of degree $d$ with border rank $\brank(f)\le k$.

Now, we recall that some part of defining equations for $\sigma_k(v_d(\P V))$ comes from so-called \textit{symmetric flattenings}. Consider the polynomial ring $S=S^\bullet V=\CC[x_0,\ldots,x_n]$ and consider another polynomial ring $T=S^\bullet V^\ast=\CC[y_0,\ldots,y_n]$, where $V^\ast$ is the \ti{dual} $\CC$-vector space of $V$. Define the differential action of $T$ on $S$ as follows: for any $g\in T_{d-a}, f\in S_d$, we set
\begin{equation*}
 g\cdot f=g(\partial_0,\partial_1,\ldots,\partial_n)f\in S_a,
\end{equation*}
where $\partial_i=\partial/\partial{x_i}$
Let us take bases for $S_a$ and $T_{d-a}$ as
\begin{align*}
\mbf{X}^{I}=\frac{1}{i_0 !\cdots i_n !}x_0^{i_0}\cdots x_n^{i_n}&\quad\trm{and}\quad
\mbf{Y}^{J}=y_0^{j_0}\cdots y_n^{j_n},
\end{align*}
with $|I|=i_0+\cdots+i_n=a$ and $|J|=j_0+\cdots+j_n=d-a$. For a given $f=\sum_{|I|=d}b_I\cdot \mbf{X}^I$ in $S_d$, we have a linear map
\[\phi_{d-a,a}(f):T_{d-a}\to S_a,\quad g\mapsto g\cdot f\] for any $a$ with $1\le a\le d-1$, which can be represented by the following ${a+n\choose n}\times{d-a+n\choose n}$-matrix:
\begin{equation*}
\left(\begin{array}{ccc}
&&\\
&b_{I,J}&\\
&&\end{array}\right) \quad \trm{with $b_{I,J}=b_{I+J}$},
\end{equation*}
in the bases defined above. We call this `the $(d-a,a)$-symmetric flattening (or \ti{catalecticant}) matrix' of $f$. It is easy to see that the transpose $\phi_{d-a,a}(f)^{T}$ is equal to $\phi_{a,d-a}(f)$.

It is obvious that if $f$ has (border) rank 1, then any symmetric flattening $\phi_{d-a,a}(f)$ has rank 1. By subadditivity of matrix rank, we also know that $\rank~\phi_{d-a,a}(f)\le r$ if $\brank(f)\le r$. So, we obtain a set of defining equations coming from $(k+1)$-minors of the matrix $\phi_{d-a,a}(f)$ for $\sigma_k(v_d(\P V))$. For some small values of $k$, it is known that these minors are sufficient to cut the variety $\sigma_k(v_d(\P V))$ scheme-theoretically (see \cite[Theorem~3.2.1]{LO}).

Let us recall some more basic terms and facts. Let $Z \subset\P W$ be a (reduced and irreducible) variety and $\hat{Z}$ be its affine cone in $W$. Consider a (closed) point $\hat{p}\in \hat{Z}$ and say $p$ the corresponding point in $\P W$. We denote the \ti{affine tangent space to $Z$ at $p$} in $W$ by $\hat{T}_{p}Z$ and we define the  \textit{(affine) conormal space to $Z$ at $p$}, $\hat{N}^{\ast}_{p}Z$ as the annihilator $(\hat{T}_{p}Z)^{\perp} \subset W^{\ast}$.
Since $\dim \hat{N}^{\ast}_{p}Z+\dim\hat{T}_{p}Z=\dim W$ and $\dim Z\le \dim \hat{T}_{p}Z-1$, we get that 
\begin{gather}\label{dim_conormal}
  \dim \hat{N}^{\ast}_{p}Z\le \codim(Z,\P W)
\end{gather}
and the equality holds if and only if $Z$ is smooth at $p$. This conormal space is quite useful to study the (embedded) tangent space $\TT_{p} Z$.

For any given form $f\in S^d V$, we call $\partial\in T_t$ \ti{apolar} to $f$ if the differentiation $\partial\cdot f$ gives zero (i.e., $\partial\in\ker\phi_{t,d-t}(f)$). And we define the \ti{apolar ideal} $f^{\perp} \subset T$ as
\[
  f^\perp=\{\partial\in T~|~\partial\cdot f=0\}.
\]
It is straightforward to see that $f^\perp$ is indeed an ideal of $T$. Moreover, it is well-known that the quotient ring $T_f=T/f^\perp$ is an \ti{Artinian Gorenstein algebra with socle degree $d$} (see e.g. \cite[Chapter 1]{IK}). In terms of this apolar ideal, we have a useful description of (a part of) conormal space as follows:

\begin{proposition}\label{conormal_prop} Suppose that $f\in S^d V$ corresponds to a (closed) point $[f]$ of $\sigma_k(v_d(\P V)) \setminus\sigma_{k-1}(v_d(\P V))$. Then, for any $a$ with $1\le a\le{\lfloor \frac {d+1}{2}\rfloor}$ with $\rank \phi_{d-a,a}(f)=k$ we have
\begin{equation*}
\hat{N}^{\ast}_{[f]}\sigma_k(v_d(\P V))\supseteq(f^\perp)_a\cdot(f^\perp)_{d-a} 
\end{equation*}
as a subspace of $T_d=S^d V^\ast$.
\end{proposition}

\begin{proof}
Let $X \subset \P W$ be any variety. For any linear embedding $W\hookrightarrow A\otimes B$ and the induced embedding $X \subset \P W\hookrightarrow \P(A\otimes B)$, it is well-known that for any $[f]\in \sigma_k(X) \subset\P(A\otimes B)$, considering $\sigma_k(X)$ as a subvariety of $\P(A\otimes B)$, we have 
\[
  \hat{N}^{\ast}_{[f]}\sigma_k(X)\supseteq \ker(f)\otimes\im(f)^\perp=\hat{N}^{\ast}_{[f]}\sigma_p(Seg(\P A\times \P B))
\]
in $A^\ast \otimes B^\ast$ provided that $X \subseteq\sigma_p(Seg(\P A\times \P B))$, $X\nsubseteq\sigma_{p-1}(Seg(\P A\times \P B))$ and $f$ has rank $k\cdot p$ as a linear map in $\Hom(A^\ast,B)$ (see e.g. \cite[section 2.5]{LO}). Here, $Seg(\P A\times \P B)$ means the Segre variety in $\P(A\otimes B)$.

Further, since $X \subset \P W  \subset \P(A\otimes B)$, then as a subvariety of $\P W$ it holds that 
$$\hat{N}^{\ast}_{[f]}\sigma_k(X)\supseteq \pi(\ker(f)\otimes\im(f)^\perp)=\hat{N}^{\ast}_{[f]}(\sigma_p(Seg(\P A\times \P B))\cap \P W)~,$$
where $\pi:A^\ast \otimes B^\ast \to W^\ast$ is the dual map of the given inclusion $W\hookrightarrow A\otimes B$.

The assertion is immediate when we apply this fact to a partial polarization $S^d V \hookrightarrow S^a V\otimes S^{d-a}V$, because $X=v_d(\P V)$ is contained in $Seg(\P S^a V\times \P S^{d-a}V) \subset \P(S^a V\otimes S^{d-a}V)$ (i.e., $p=1$ case) and $\rank \phi_{d-a,a}(f)=k$, $\ker\phi_{d-a,a}(f)=(f^\perp)_{d-a}$, $\im(\phi_{d-a,a}(f))^\perp=(f^\perp)_{a}$.
\end{proof}

\subsection{Proof of Theorem~\ref{thm_m1}}

Now we study singularity and non-singularity of the subsecant variety
$\sigma_k(v_d(\PP^1))  \subset \sigma_k(v_d(\Pn))$
in each range of $k,d$ as Theorem~\ref{thm_m1}.

\begin{proof}[Proof of Theorem~\ref{thm_m1}] (i) Let $f$ be \ti{any} form belonging to $\sigma_k(v_d(\P^1)) \setminus\sigma_{k-1}(v_d(\P^n))$. Set $X=v_d(\P^n)  \subset \PN$, the Veronese variety.
  Consider $f$ as a polynomial in $\CC[x_0,x_1]$ as in \S\ref{sec:symm-flatt-conorm}. Then, by \cite[Theorem~1.44]{IK} we know that $T/f^\perp$ is an Artinian Gorenstein algebra with socle degree $d$ and that $f^\perp$ is a complete intersection of two homogeneous polynomials $F, G$ of each degree $a$ and $b$ $(a\le b)$ with $a+b=d+2$ \ti{as an ideal of $\CC[y_0,y_1]$}, where the Hilbert function of $T/f^\perp$ is
  \begin{gather}\label{H-shape of f}
    (1,2,\ldots, a-1,a,\ldots, a, a-1,\ldots,2,1).
  \end{gather}
  We claim that $\rank~\phi_{k,d-k}(f)=k$ (i.e., $a=k$). If $a<k$, then by the shape (\ref{H-shape of f}) we see that $\rank~\phi_{t,d-t}(f)<k$ for all $t$. In particular, all $k$-minors of $\phi_{t,d-t}(f)$ vanish for any $t$. Since the $k$-minors of catalecticant $\phi_{t,d-t}$ for each $k-1\le t\le d-(k-1)$ give the ideal of $\sigma_{k-1}(v_d(\P^1))$ (e.g. [Ibid., Theorem 1.45]), this implies $f\in \sigma_{k-1}(v_d(\P^1)) \subset\sigma_{k-1}(v_d(\P^n))$, which is a contradiction.

  Hence, we have that $f^\perp=\big(F,G,y_2,\ldots,y_n\big)$ as an ideal in $T=\CC[y_0,y_1,\ldots,y_n]$ for some polynomial $F$ of degree $k$ and $G$ of degree $(d-k+2)$ in $\CC[y_0,y_1]$.

  Now, let us show that $\sigma_k(X)$ is smooth at $f$ by computing the dimension of conormal space. In general, by (\ref{dim_conormal}) we have
  \begin{equation}\label{upp_bd_cor}
    {n+d\choose d}-kn-k\ge\dim_\CC \hat{N}^{\ast}_{[f]}\sigma_k(X),
  \end{equation}
  where the left hand side is given by the expected codimension of the $k$-th secant variety. By Proposition~\ref{conormal_prop}, we also have
  \begin{equation}\label{low_bd_cor}
    \dim_\CC\hat{N}^{\ast}_{[f]}\sigma_k(X)\ge\dim_\CC(f^\perp)_{k}\cdot(f^\perp)_{d-k}\,.
  \end{equation}
  Thus, $f$ is a smooth point of $\sigma_k(X)$ if the lower bound for the dimension of conormal space in (\ref{low_bd_cor}) is equal to the expected codimension in (\ref{upp_bd_cor}).

  Since $k\le\frac{d+1}{2}$ by the assumption, note that $d-k\ge k$ unless $d$ is odd and $k=\frac{d+1}{2}$, where $d-k=\frac{d-1}{2}<k$. 

  a) If $d$ is odd and $k=\frac{d+1}{2}$, then we have
  \begin{align*}
    (f^\perp)_{k}\cdot(f^\perp)_{d-k}&=\big(F, y_2,\ldots,y_n\big)_k\cdot\big(y_2,\ldots,y_n\big)_{d-k}\\
                                     &=\bigg\langle\big(\{y_i y_j~|~2\le i,j\le n\}\big)_d  \cupdot F\cdot\{y_2,\ldots,y_n\}\cdot\{y_0^{d-k-1},y_0^{d-k-2}y_1,\ldots,y_1^{d-k-1}\}\bigg\rangle \\
                                     &=\CC[y_0,y_1,\ldots,y_n]_d \setminus\bigg(\{y_0^{d},y_0^{d-1}y_1,\ldots,y_1^{d}\} \cupdot\{y_2,\ldots,y_n\}\cdot\{y_0^{d-1},\ldots,y_1^{d-1}\} \bigg)\\
                                     &\qquad \cupdot F\cdot\{y_2,\ldots,y_n\}\cdot\{y_0^{d-k-1},y_0^{d-k-2}y_1,\ldots,y_1^{d-k-1}\},
  \end{align*}
  where $ \cupdot$ means the `disjoint union' of sets of forms of degree $d$.

  So, we obtain
  \begin{align*}
    \dim \hat{N}^{\ast}_{[f]}\sigma_k(X)&\ge\dim_\CC (f^\perp)_{k}\cdot(f^\perp)_{d-k}\\
                                        &={n+d\choose d}-(d+1)-d(n-1)+(n-1)(d-k) \quad(\trm{note that $k=\frac{d+1}{2}$})\\
                                        &={n+d\choose d}-kn-k,
  \end{align*}
  which tells us that $\sigma_k(X)$ is smooth at $f$.

  b) When $d$ is odd and $k<\frac{d+1}{2}$ or $d$ is even,  we have $k\le d-k$ and
  \begin{align*}
    (f^\perp)_{k}\cdot(f^\perp)_{d-k}&=\big(F, y_2,\ldots,y_n\big)_k\cdot\big(F,y_2,\ldots,y_n\big)_{d-k}\\
                                     &=\bigg\langle\big(\{y_i y_j~|~2\le i,j\le n\}\big)_d  \cupdot F\cdot\{y_2,\ldots,y_n\}\cdot\{y_0^{d-k-1},y_0^{d-k-2}y_1,\ldots,y_1^{d-k-1}\}\\
                                     & \qquad \cupdot F^2\cdot\{y_0^{d-2k},y_0^{d-2k-1}y_1,\ldots,y_1^{d-2k}\}\bigg\rangle.
  \end{align*}
  Thus, by a similar dimension counting as case a), we see that
  \begin{align*}
    \dim \hat{N}^{\ast}_{[f]}\sigma_k(X)&\ge{n+d\choose d}-(d+1)-d(n-1)+(n-1)(d-k)+(d-2k+1) \\
                                        &={n+d\choose d}-kn-k,
  \end{align*}
  which coincides with the expected codimension as desired. Thus, $f$ is a smooth point of $\sigma_k(X)$.
  \medskip

  (ii) First note that $\dim \sigma_k(v_d(\PP^1)) = \min\set{2k-1, d}$ and the incidence $I$ has dimension $2k-1$. In the case $d \leq 2k-2$, each fiber of $p : I \rightarrow \PP^d$ is of dimension $\geq 1$, so for a general $a\in\sigma_k(v_d(\PP^1))$, it holds $q_i(p^{-1}(a)) = v_d(\PP^1)$ for some $i$ with $1 \leq i \leq k$ in the incidence (\ref{incidence_on_Z}) in \S\ref{sec:degen-subv-small}.

  Now, let $n\ge3, k=3$ or $n\ge2, k\ge4$ and $d = 2k-2$.
  Suppose $\sigma_k(v_d(\PP^1))  \not\subset \Sing(\sigma_k(v_d(\Pn)))$. Then a general point $a \in \sigma_k(v_d(\PP^1)) = \PP^d$ is a smooth point of $\sigma_k(v_d(\Pn))$. Since $q_i(p^{-1}(a)) = v_d(\PP^1)$ for some $i$, 
  it follows from Lemma~\ref{lem:TTx-in-TTa} that
  for $M=\TT_a \sigma_k(v_d(\Pn))$, we have
  the inclusion $\TT_x (v_d(\Pn))  \subset M$ for a general $x \in v_d(\PP^1)$,
  and then the inclusion holds for any $x \in v_d(\PP^1)$.
  This is because, for the Gauss map $\gamma: v_d(\Pn) \rightarrow \GG(n, \PN)$
  sending $\gamma(z) = \TT_z (v_d(\Pn))$ (it is a morphism since $v_d(\Pn)$ is smooth),
  considering the closed set
  $G_M = \set{W \in \GG(n, \PN) \mid W \subset M}$,
  we have $\gamma(U) \subset G_M$ for a certain non-empty open subset
  $U \subset v_d(\PP^1)$, and then
  $\gamma(v_d(\PP^1)) \subset G_M$. Therefore,
  \begin{equation}\label{containment_m=1}
    \lin{\bigcup_{x \in v_d(\PP^1)}\TT_x (v_d(\Pn))}  \subset \TT_a \sigma_k(v_d(\Pn)).
  \end{equation}
  Taking $m = 1$, $\Lambda=\emptyset$, and $A = \PP^1$ in Proposition~\ref{prop:dim-linTT-geq-vdAA},
  the number (\ref{eq:sum-vdAA}), a lower bound for dimension of left hand side of (\ref{containment_m=1}), is equal to $dn$. Thus we have
  \[
    (2k-2)n = dn \leq k(n+1) -1 \; \big(\!=\dim\TT_a \sigma_k(v_d(\Pn))\big),
  \]
  which is equivalent to the formula $n \leq (k-1)/(k-2)$.
  It follows that $n \leq 2$ if $k=3$, and $n = 1$ if $k \geq 4$,
  contrary to our assumption.
  
  Finally, since $\sigma_{k-1}(v_d(\PP^1)) \subsetneq\sigma_k(v_d(\PP^1))$ when $d \geq 2k-2$ (note that $\dim \sigma_{k-1}(v_d(\PP^1)) = 2k-3 < d$), the $\sigma_k(v_d(\PP^1))$ is a non-trivial singular locus of $\sigma_k(v_d(\Pn))$, which means that $\sigma_k(v_d(\PP^1))  \not\subset \sigma_{k-1}(v_d(\Pn))$, by Lemma~\ref{lem:sigma-k-1}.
  \medskip

  (iii) By assumption,
  $\dim \sigma_{k-1}(v_d(\PP^1)) = \min\set{2k-3, d} = d$, that is to say,
  \begin{gather*}
    \sigma_{k-1}(v_d(\PP^1)) = \sigma_{k}(v_d(\PP^1))= \lin{v_d(\PP^1)} = \PP^d;
  \end{gather*}
  hence the assertion follows.
  \medskip

  (iv) For $(n,k)=(2,3)$, smoothness of all points in
  $\sigma_3(v_d(\P^1)) \setminus\sigma_{2}(v_d(\PP^2))$
  for $d\ge4$ was already proved in \cite[Theorem~2.14]{H18}.
  This is included for completeness to the statement.
\end{proof}

\begin{remark}\label{m=1_exception_why}
  (iv) is the exception to the trichotomy in Theorem~\ref{thm_m1}.
  Under the condition $(k,d,m,n) = (3,4,1,2)$ of (iv),
  the arithmetic deduced from the inclusion assumption (\ref{containment_m=1}) of moving tangents in the proof does not make any contradiction.
  The situation is similar in the other exceptional case to the trichotomy, $(k,d,m,n) = (4,3,2,3)$ (Part (iv) of Theorem~\ref{general_m}).
\end{remark}

\section{Proof of main results}\label{sect_m>=2}

In this section, we prove Theorem~\ref{general_m} and Theorem~\ref{general_m2}.
We will first discuss the non-singularity result and then the results for the singular loci.

\subsection{Generic smoothness}

We begin with a lemma which deals with
a secant fiber of a general point in an $m$-subsecant variety $v_d(\Pm)$ in $v_d(\Pn) \subset \PN$.

\begin{lemma}\label{lem:large-d-incidence:1}
  Assume
  $\dim \sigma_k(v_d(\Pn)) = {nk + k - 1}$ and
  $\dim \sigma_k(v_d(\Pm)) = {mk + k - 1}$ (i.e., $v_d(\Pn)$ and $v_d(\Pm)$ are non-defective). Let $k \leq \tbinom{m+d-1}{m}$.
  Fix $L=\Pm \subset \Pn$ to be an $m$-plane,
  and take $a \in \sigma_k(v_d(L))$ to be a general point.
  Then, in the incidence $\II \subset \PN \times (\Pn)^k$
  with the first projection $p: I \rightarrow \PN$
  as (\ref{main_incidence}) in \S\ref{sect_prep}, we have the following inclusion scheme-theoretically,
  \begin{gather*}
    p^{-1}(a)  \subset \set{a} \times \Lk.
  \end{gather*}
\end{lemma}

\begin{proof}

  (i)
  Consider any $(a, x_1', \dots, x_k') \in p^{-1}(a) \subset \II$.
  Let $\IIm \subset \PN \times (L)^k$ be the another incidence as in Lemma~\ref{lem-I-I0}.
  Since $a \in \sigma_{k}(v_d(L))$ is general,
  it follows $a \notin \sigma_{k-1}(v_d(L))$ and $a \notin p(\IIm  \setminus \IIm^0)$ by Remark~\ref{rem:large-d-incidence:1} (b).
  Since $\dim \sigma_k(v_d(L)) = {mk + k - 1}$, the secant fiber of $\IIm \rightarrow \PN$ at $a$ also consists of finite points. So, by Lemma~\ref{lem-I-I0}, we have
  $(a, x_1', \dots, x_k') \in \II^0$.
  From Lemma~\ref{lem:sigma-k-1}, it is also true that
  $a \notin \sigma_{k-1}(v_d(\Pn))$.
  Thus, we may write $a = \sum_{i=1}^k c_i x_i$ for some $c_i \in \CC$ as regarding $a$ and $x_i = v_d(x_i')$ as vectors in the affine space $\CC^{N+1}$,
  where
  $c_i \neq 0$ for all $1 \leq i \leq k$.

  As in Remark~\ref{rem:def-y'},
  let us set $y_i' = [x_{i, 0}': \dots : x_{i, m}': 0 : \dots : 0]$.
  For $y_i = v_d(y_i')$,
  the diagram~(\ref{eq:diag-PnPmPNPb}) implies $a = \sum_{i=1}^k c_i y_i$,
  where $y_i' \neq 0$ for $1 \leq i \leq k$.
  For the affine open subset $V_0 = \set{t_0 \neq 0} \subset \Pn$,
  we may assume $x_i' \in V_0$ for all $i$.
  Since $v_{d}: \Pn \rightarrow \PP^{N}$
  is parameterized on $V_0$ by $\mo[t,u]{d}$, and $a \in \lin{v_d(L)}$, it holds that
  \begin{gather*}
    0 = \sum_{i=1}^k c_i \cdot \{(u_{1} \cdot \mo[t]{d-1}) (x_i')\}
    = \sum_{i=1}^k c_i \cdot x_{i,m+1}'\cdot \{\mo[t]{d-1}(y_i')\},
  \end{gather*}
  where
  for $Mono = u_{1} \cdot \mo[t]{d-1}, \mo[t]{d-1}$ and $pt = x_i', y_i'$,
  the symbol $\{Mono(pt)\}$ means the vector obtained by evaluating monomials in $Mono$ at the value of $pt$.
  
  Since $k \leq \tbinom{m+d-1}{m}$,
  applying Remark~\ref{rem:large-d-incidence:1} (b) to $\sigma_k(v_d(L))$,
  we may assume
  \begin{gather}\label{eq:dim-d-1-indep}
    \dim\lin{v_{d-1}(y_1'), \dots, v_{d-1}(y_k')} = k-1,
  \end{gather}
  which gives $c_i \cdot x_{i,m+1}' = 0$,
  thus $x_{i,m+1}' = 0$ for all $1 \leq i \leq k$
  (more precisely, the linear independence of (\ref{eq:dim-d-1-indep})
  means a $k$-minor of the corresponding matrix is non-zero,
  and $c_i \cdot x_{i,m+1}' = 0$ is obtained by multiplying the inverse of the $k \times k$ submatrix).
  Similarly, we can obtain $x_{i,j}' = 0$ for each $j > m$ and for all $1 \leq i \leq k$, which gives the linear defining equations for $\Lk$ in $\Pnk$.
  Hence, $x_1', \dots, x_k' \in L$.
  \medskip

  (ii)
  Let $U \subset \Pnk$ be the open subset used in Remark~\ref{rem:large-d-incidence:1},
  where $I^0$ is the $\PP^{k-1}$-bundle over $U$.
  We define a morphism $\Phi: \PP^{k-1} \times U \rightarrow \PN \times U$ by
  \begin{gather*}
    \Phi((c_1: \dots: c_k), (x_1', \dots, x_k')) = \Bigl(\,\sum_{i=1}^k c_i v_d(x_i'),\,(x_1', \dots, x_k')\Bigr).
  \end{gather*}
  Note that, by the linear independence of $v_d(x_1'), \dots, v_d(x_k')$
  for $(x_1', \dots, x_k') \in U$,
  $\sum_{i=1}^k c_i v_d(x_i') = \sum_{i=1}^k \tilde{c}_i v_d(x_i') \in \PN$
  if and only if
  $(c_1: \dots: c_k) = (\tilde c_1: \dots: \tilde c_k) \in \PP^{k-1}$.
  Then $\Phi(\PP^{k-1} \times U) = I^0$, and moreover, we have the isomorphism
  $\PP^{k-1} \times U \simeq I^0$ under $\Phi$.

  Let $U_0 = U \cap (V_0)^{k} \subset \Pnk$, where $(V_0)^{k}$ is an affine variety
  and its affine coordinates ring is $A=\CC[\{x_{i,j}'\}]$.
  In addition,
  for each $k$-minor $\xi$ of the matrix whose $i$-th column consists
  of monomials of $m$ variables $x_{i,1}', \dots, x_{i,m}'$ of degrees $\leq d-1$,
  we set $(V_0)^k_{\xi} = \set{\xi \neq 0}$,
  an open subset of $(V_0)^k$ whose coordinates ring is $A_{\xi}$.
  Let $W \subset \set{c_1 \neq 0} \subset \PP^{k-1}$ be the affine open subset such that
  all the coordinates $c_1, \dots, c_k$ are nonzero,
  where the coordinates ring of $W$ is $\CC[c_2, \dots, c_k]_{c_2\dotsm c_k}$
  by regarding $c_1 = 1$.

  We may assume $p^{-1}(a) \subset I^0 \cap (\set{a} \times (V_0)^k)$.
  To consider the scheme-theoretic structure of $p^{-1}(a)$,
  for the composite morphism $\Phi_1 = p \circ \Phi: \PP^{k-1} \times U_0 \rightarrow \PN$,
  we take the fiber
  \begin{gather*}
    \Phi_1^{-1}(a) \subset \PP^{k-1} \times U_0 \subset \PP^{k-1} \times (V_0)^{k}.
  \end{gather*}
  Since $a \notin \sigma_{k-1}(v_d(\Pn))$,
  $\Phi_1^{-1}(a) \subset W \times (V_0)^{k}$.
  Since $a$ is general in $\sigma_{k}(v_d(L))$,
  and by (\ref{eq:dim-d-1-indep}),
  $\Phi_1^{-1}(a)$ is contained in
  the union of affine open subsets $W \times (V_0)^k_{\xi}$ with all $k$-minors $\xi$.

  We take $F_{\xi} = \Phi_1^{-1}(a) \cap (W \times (V_0)^k_{\xi})$ for each $\xi$,
  and consider the ideal
  $I(F_{\xi})$ in $A_{\xi}[c_2, \dots, c_k]_{c_2\dotsm c_k}$,
  the affine coordinates ring of $W \times (V_0)^k_{\xi}$.
  For $\beta = \binom{m+d}{m}-1$,
  we may write
  \begin{gather*}
    a = (1:a^{(1)}:\dots:a^{(\beta)}:0:\dots:0) \in \lin{v_d(L)} \subset \PN
  \end{gather*}
  with $a^{(1)}, \dots, a^{(\beta)} \in \CC$ and $a^{(\ell)} = 0$ if $\ell > \beta$.
  Then the expression
  $a = \sum_{i=1}^k c_i \cdot v_d(x_i')$
  means that $a^{(\ell)} \sum_{i=1}^k c_i \cdot v_d(x_i')^{(0)} - \sum_{i=1}^k c_i \cdot v_d(x_i')^{(\ell)} \in {I(F_{\xi})}$ for $1 \leq \ell \leq N$,
  where $v_d(x_i')^{(\ell)}$ is the $\ell$-th coordinate of $v_d(x_i') \in \PN$.
  In particular,
  $\sum_{i=1}^k c_i \cdot v_d(x_i')^{(\ell)} \in {I(F_{\xi})}$ for $\ell > \beta$.
  Using the discussion of (i),
  we have $x_{i,j}' \in I(F_{\xi})$
  for all $1 \leq i \leq k$ and $j > m$,
  which means that $I(F_{\xi})$ contains the defining ideal of $(\PP^{k-1} \times (L)^k) \cap (W \times (V_0)^k_{\xi})$.
  Thus, scheme-theoretically, it follows
  $F_{\xi} \subset \PP^{k-1} \times (U_0 \cap (L)^k)$
  for any $\xi$,
  and hence $\Phi_1^{-1}(a) \subset \PP^{k-1} \times (U_0 \cap (L)^k)$.
  Therefore,
  $p^{-1}(a) \subset \set{a} \times (L)^k$.
\end{proof}

\begin{remark}\label{rem:COV2-1} We recall some known results on the $k$-the secant variety and its incidence in terms of $k$-fold symmetric product of $\Pn$.
\begin{enumerate}[(a)]
\item It is known that $\Symm^k(\Pn)$ is non-singular at $(x_1', \dots, x_k')$ if $x_i' \neq x_j'$ whenever $i\neq j$. Thus, the subset of all distinct $k$-points of $\Pn$ is a smooth open subscheme of $\Symm^k(\Pn)$ (see e.g. \cite[Lemma 7.1.4]{BK05}).
Then we also consider the incidence variety in this setting as
\begin{gather*}
  \widetilde{I}  \subset \PP^{N} \times \Symm^k(\Pn),
\end{gather*}
where $\widetilde{I}$ corresponds to $I$ in (\ref{main_incidence}) under the natural map $\PP^{N} \times \Pnk\rightarrow \PP^{N} \times \Symm^k(\Pn)$ and 
$\widetilde{p}: \widetilde{I} \rightarrow \sigma_k(v_d(\Pn))  \subset \PP^{N}$
is the first projection.
\item Assume $k(n+1) < \binom{n+d}{n}$.
  Then, we know from \cite[Theorem~1.1]{COV2} that the projection $\widetilde{p}: \widetilde{I} \rightarrow \sigma_k(v_d(\Pn))$ is birational except $(k,d,n)=(9,6,2),(8,4,3),(9,3,5)$, because it is a dominant and generically injective morphism.
\end{enumerate}
\end{remark}

Now, we are ready to prove Part (i) of Theorem~\ref{general_m} and Theorem~\ref{general_m2}.

\begin{proof}[Proof of Part (i) of Theorem~\ref{general_m} and Part (i) of Theorem~\ref{general_m2}] 
  For an $m$-plane $\Pm \subset \Pn$ with $m \geq 2$, we take
  the $m$-subsecant variety $Z=\sigma_k(v_d(\Pm))$ of $Y=\sigma_k(v_d(\Pn))$.
  From \cite{AH}, for $d \geq 3$,
  $Z$ does not fill $\lin{Z}$ and is secant defective
  if and only if $(k,d,m) = (7,3,4),(5,4,2),(9,4,3),(14,4,4)$.
  Thus, by the assumptions of Theorems \ref{general_m} and \ref{general_m2},
  we know that
  \begin{gather*}
    \dim Y = {nk + k - 1}, \quad 
    \dim Z = {mk + k - 1} \le \dim \lin{Z}=\tbinom{m+d}{m}-1,
  \end{gather*}
  that is, $Y, Z$ are non-defective.
  In this case,
  $Z = \lin{Z}$ if and only if $k=\frac{\binom{m+d}{m}}{m+1} \in \NN$.
  In particular, under the assumption $k < \mu$ of Theorem~\ref{general_m} (i),
  we have $Z \subsetneq \lin{Z}$.

  If $Z \subsetneq \lin{Z}$,
  then
  since $(k,d,m)=(9,3,5),(8,4,3),(9,6,2)$ are excluded
  from Theorem~\ref{general_m} and (i) of Table~\ref{singTable1},
  it follows from \cite[Theorem~1.1]{COV2} that
  $Z$ is generically identifiable.
  If $Z = \lin{Z}$, then $(k,d,m) = (5,3,3),(7,5,2)$ of (i) of Table~\ref{singTable1} only occur,
  and in these cases, it follows from \cite[Theorem~1]{GM}
  that $Z$ is generically identifiable.

  Let $a \in Z$ be a general point and consider  $\widetilde{p}: \widetilde{I} \rightarrow Y$.
  Note that $k \leq \tbinom{m+d-1}{m}$ for each $(k, d, m)$ of our range,
  which is an assumption of Lemma~\ref{lem:large-d-incidence:1}.
  Applying Lemma~\ref{lem:large-d-incidence:1}, Remark~\ref{rem:COV2-1} (b), and the generic identifiability of $Z$,
  we may assume that the scheme-theoretic fiber $\widetilde{p}^{-1}(a)$ is one point $\mathbf{x}=(a,x_1',\dots,x_k') \in \widetilde{I}\cap (\PP^{N} \times\Symm^k(L))$ and $\mathbf{x}$ is a non-singular point in $\widetilde{I}$, because $\mathbf{x}$ is contained in a smooth Zariski open subset of $\widetilde{I}$ (i.e., a projective bundle over a smooth open base, see Remarks~\ref{rem:large-d-incidence:1} (a) and \ref{rem:COV2-1} (a)). 

  Now, we restrict the projective morphism $\widetilde{p}: \widetilde{I} \rightarrow Y$ onto a non-empty affine open neighborhood $a\in W=\Spec A \subset Y$ and another open subset $\mathbf{x}\in U=\Spec B \subset \widetilde{I}$, and take the injective ring homomorphism $A\hookrightarrow B$ corresponding to $\widetilde{p}|_{U}:U\to W$. Also, let $m_a$ (resp. $m_{\mathbf{x}}$) be the maximal ideal of $a$ in $A$ (resp. of $\mathbf{x}$ in $B$). Note that we may take $U$ so that $\widetilde{p}|_{U}:U\to W$ is a finite morphism (cf. \cite[Chapter II, Ex.3.22 (d)]{H77}). 

Since $A/m_a = B/m_{\mathbf{x}}\simeq\CC$ and
\begin{gather*}
  \widetilde{p}^{-1}(a)\simeq \Spec(B\otimes_{A} A/m_a)\simeq\Spec(B/m_aB)
\end{gather*}
is isomorphic to one simple point $\Spec{B/m_{\mathbf{x}}}$, we have $m_aB=m_{\mathbf{x}}$ in $B$.
Let $B_{m_a} = B \otimes_A A_{m_a}$, whose member can be expressed as
$b/s$ with $b \in B$ and $s \in A \setminus m_a$.
We have
$m_a B_{m_a}=m_{\mathbf{x}} B_{m_a}$ in $B_{m_a}$, and then,
\begin{gather*}
  (A_{m_a}+m_a B_{m_a})/m_a B_{m_a}\simeq A_{m_a}/m_a A_{m_a}\simeq B_{m_a}/m_a B_{m_a},
\end{gather*}
which implies
$A_{m_a}+m_a B_{m_a} = B_{m_a} + m_a B_{m_a} =B_{m_a}$
as $A_{m_a}$-module.
By Nakayama lemma (see e.g. \cite[Corollary of Theorem 2.2]{Mat}), it follows $A_{m_a}=B_{m_a}$.
In particular, $B_{m_a}$ is a local ring, whose maximal ideal is $m_{\mathbf{x}}B_{m_a}$.
Thus, we have
\[
  A_{m_a}=B_{m_a}
  = (B_{m_a})_{m_{\mathbf{x}}B_{m_a}}
  = B_{m_{\mathbf{x}}},
\]
which implies that $a$ is a smooth point in $Y$.
\end{proof}

We present an example which shows that one can not extend this generic smoothness result to arbitrary point in the locus
$\sigma_k(v_d(\P^m)) \setminus\sigma_{k-1}(v_d(\Pn))$.

\begin{example}[Singularity can occur at a special point in Part (i) of Theorem~\ref{general_m2}]\label{sing_special_pt} 
  Let
  \begin{gather*}
    V=\CC\langle x,y,z,w\rangle\supset W=\CC\langle x,y,z\rangle
  \end{gather*}
  and $f=x^2y^2+z^4$ be a form of degree $4$. Then, $f$ represents a point in $\sigma_4(v_4(\PP W)) \setminus\sigma_3(v_4(\PP V))$. Note that $\rank\phi_{2,2}(f) =4>3$ where $\phi_{a,d-a}:S^d V\to S^a V\otimes S^{d-a}V$ is the symmetric flattening. Theorem~\ref{general_m2} (i) shows that a \ti{general} form in $\sigma_4(v_4(\PP W)) \setminus\sigma_3(v_4(\PP V))$ is a smooth point. But, here we show that $f$ is a singular point of $\sigma_4(v_4(\PP V))$. We know that the form $f_D=x^2 y^2$ has Waring rank 3 so that $f_D=\ell_1^4+\ell_2^4+\ell_3^4$ for some $\ell_i\in\CC[x,y]_1$. By \cite[Theorem~2.1]{H18} $f_D$ is also a singular point of $\sigma_3(v_4(\PP V))$. Since $f\in\langle f_D, z^4\rangle$, by Terracini's lemma, we see that $\TT_{z^4} v_4(\PP V) \subset\TT_{f}\sigma_{4}(v_4(\PP V))$ and $\TT_{\ell_i} v_4(\PP V) \subset\TT_{f}\sigma_{4}(v_4(\PP V))$ for any $i$. Further, because $\sigma_3(v_4(\PP^1))=\langle v_4(\PP^1)\rangle$ and $f_D$ has 1-dimensional secant fiber in its incidence, one can move $\ell_i$ along this $\PP^1$. Thus, we have
  \begin{equation}\label{eg:x2y2+z4}
    \TT_{f}\sigma_{4}(v_4(\PP V))\supseteq\langle \bigcup_{\ell_i\in\PP^1}\TT_{\ell_i^4} v_4(\PP V), \TT_{z^4} v_4(\PP V)\rangle.
  \end{equation}
  Note that using the parameterization (\ref{eq:coordi-t0=1-tu}) we can estimate the dimension of right hand side of (\ref{eg:x2y2+z4}). Take an affine open subset $\{[1:t:u_1:u_2]\}$ of $\PP^3$ and (with a change of coordinates) let $z^4$ be $[1:0:1:1]$ and $\ell_i\in\PP^1$ be represented by $[1:t:0:0]$ for $t\in\CC$. Then, by (\ref{eq:expr-TTx-0}) the embedded tangent space to $v_4(\PP V)$ at $[1:t:u_1:u_2]$ is given as the row span of
  \[
    \Small
    \begin{bmatrix}[ccccc:cccc:cccc:ccccccc]
      1 & t&t^2&t^3&t^4& u_1 & u_1 t &u_1 t^2 & u_1 t^3 & u_2 & u_2 t &u_2 t^2 & u_2 t^3 & u_1^2 &u_1^2 t&\cdots &u_1^3 &\cdots&u_2^2&\cdots 
      \\      
        & 1 &2t&3t^2&4t^3&   & u_1  &2 u_1 t & 3u_1 t^2 &  & u_2  &2u_2 t & 3u_2 t^2 &  &u_1^2 &\cdots & &\cdots&&\cdots  \\
        &  &&&& 1  & t  &t^2 & t^3 &  &  & &  &  2u_1&2u_1 t &\cdots & 3u_1^2&\cdots&&\cdots  \\
        &  &&&& &  &  & &1  & t  &t^2 & t^3   &  & &\cdots & &\cdots &2u_2&\cdots 
    \end{bmatrix}.
  \]
  On $[1:t:0:0]$ ($\forall t\in\CC$), this matrix turns into the shape as
  \[
    \begin{bmatrix}[ccccc:cccc:cccc:ccccccc]
      1 & t&t^2&t^3&t^4&  &  & &  &  &  & & & & & &&& 
      \\      
        & 1 &2t&3t^2&4t^3&   & &&&  &  &&  &  && &\mathsf{O} &&&  \\
        &  &&&& 1  & t  &t^2 & t^3 &  &  & &  &  & & & &&&  \\
        &  &&&& &  &  & &1  & t  &t^2 & t^3   &  & & & & &&
    \end{bmatrix}
  \]
  and at $[1:0:1:1]$ it is equal to 
  \[
    \begin{bmatrix}[ccccc:cccc:cccc:ccccccc]
      1 & &&&& 1 & & &  & 1 &  & &  & 1 && &1 &\cdots&1&\cdots 
      \\      
        & 1 &&&&   & 1  & &  &  & 1  & &  &  &1 & & &\cdots&&\cdots  \\
        &  &&&& 1  &   && &  &  & &  &  2& & & 3&\cdots&&\cdots  \\
        &  &&&& &  &  & &1  &  &&    &  & & & &\cdots &2&\cdots 
    \end{bmatrix}~,
  \]
  which shows that
  $\dim\langle \bigcup_{\ell_i\in\PP^1} \TT_{\ell_i^4} v_4(\PP V) \rangle\ge12$,
  $\dim \TT_{z^4} v_4(\PP V)=3$
  and $\langle \bigcup_{\ell_i\in\PP^1}\TT_{\ell_i^4} v_4(\PP V) \rangle\cap\TT_{z^4} v_4(\PP V)=\emptyset$.
  Thus, by (\ref{eg:x2y2+z4}) we obtain $\dim\TT_{f}\sigma_{4}(v_4(\PP V))\ge16=12+3+1$, greater than the expect dimension. Hence $f$ is a \ti{singular} point of $\sigma_{4}(v_4(\PP V))$, whereas $\sigma_{4}(v_4(\PP V))$ is smooth at a general point of $\sigma_{4}(v_4(\PP W))$. 
\end{example}

\subsection{Singularity}

In this subsection we will prove parts (ii) and (iii) both in Theorem~\ref{general_m} and Theorem~\ref{general_m2}, which show the \emph{singularity} of the $m$-subsecant loci $\SSig[n]{k,d}{m}$ in the $k$-th secant variety $\sigma_k(v_d(\Pn))$. 
Since $\SSig[n]{k,d}{m}$ is the union of all the $m$-subsecant varieties $\sigma_k(v_d(\Pm))$'s in $\sigma_k(v_d(\Pn))$ as (\ref{eq:def-m-subsec}),
it is enough to prove the statements for any
$\sigma_k(v_d(\Pm)) \subset\sigma_k(v_d(\Pn))$. 

\begin{proof}[Proof of Part (ii) of Theorem~\ref{general_m} and Part (ii) of Theorem~\ref{general_m2}]

  As we noted above, it is enough here to show that $\sigma_k(v_d(\Pm))  \subset \Sing(\sigma_k(v_d(\Pn)))$ and $\sigma_k(v_d(\Pm))  \not\subset \sigma_{k-1}(v_d(\Pn))$ for
  each $m$-subsecant variety $\sigma_k(v_d(\Pm)) \subset\sigma_k(v_d(\Pn))$.

  We will first prove that $\sigma_k(v_d(\Pm))  \subset \Sing(\sigma_k(v_d(\Pn)))$
  under the condition in Theorem~\ref{general_m} with $\frac{\binom{m+d}{m}}{m+1} \notin \NN$,
  next for Theorem~\ref{general_m2} with $(k,d,m) \neq (9,3,5),(8,4,3),(9,6,2)$,
  and finally for $\frac{\binom{m+d}{m}}{m+1} \in \NN$ or $(k,d,m)=(9,3,5),(8,4,3),(9,6,2)$.
  Basically, we use the same idea for the proof, but a detailed way of estimation will be slightly different according to each case (due to secant defectivity and non-identifiability). The non-triviality of the singular locus, i.e., $\sigma_k(v_d(\Pm))  \not\subset \sigma_{k-1}(v_d(\Pn))$ can be directly obtained at the end by Lemma~\ref{lem:sigma-k-1}.
  \medskip
  
  Let us take a general point $(a,x_1, \dots, x_k)$ in the incidence $J$ as (\ref{incidence_on_Z}) for $Z=v_d(\Pm) \subset\Pb$ and take an irreducible component $F$ of $p^{-1}(a)$ containing $(a,x_1, \dots, x_k)$. Then, $a \in \sigma_k(v_d(\Pm))$ is a general (so, smooth) point in $\sigma_k(v_d(\Pm))$. 
  
  Suppose $\sigma_k(v_d(\Pm))  \not\subset \Sing(\sigma_k(v_d(\Pn)))$. Then, we may assume that $a$ is also a smooth point in $\sigma_k(v_d(\Pn))$. In particular, we have
  \begin{gather*}
    \TT_a (\sigma_k(v_d(\Pm))) \subset\TT_a (\sigma_k(v_d(\P^n))).
  \end{gather*}
  Terracini's lemma implies $\TT_{x_i}v_d(\Pn)  \subset \TT_{a} \sigma_k(v_d(\Pn))$
  for $i=1, \dots, k$, and Lemma~\ref{lem:TTx-in-TTa} implies
  $\TT_x (v_d(\Pn))  \subset \TT_a (\sigma_k(v_d(\Pn)))$
  for a general point $x \in q_i(F)$. Thus, we have
  \begin{equation}\label{containment_large_m}
    \bigl\langle\;
    \TT_a (\sigma_k(v_d(\Pm)))
    \,\cup\bigcup_{x \in q_i(F)  \cup \set{x_1, \dots, x_k}}
    \TT_x (v_d(\Pn))
    \;\bigr\rangle
    \,\subset\, \TT_a (\sigma_k(v_d(\Pn))).
  \end{equation}

  First of all, let us consider (ii) of Theorem~\ref{general_m} with
  $\frac{\binom{m+d}{m}}{m+1} \notin \NN$.
  Set $k=\Bigr\lceil \frac{\binom{m+d}{m}}{m+1} \Bigl\rceil$.
  Then $\beta = \binom{m+d}{m}-1 < km+k-1$ and $(k-1)m+k \leq \binom{m+d}{m}$ as in
  Remark~\ref{rem:cond-k-appear}.
  We have $\Pb=\TT_a (\sigma_k(v_d(\Pm)))$ since $\sigma_k(v_d(\Pm))$ fills up the whole $\Pb$.
  It is enough to discuss
  the following three cases
  \begin{enumerate}[\quad ({a}1)]
  \item $(k-1)m+k < \binom{m+d}{m}$,
  \item $(k-1)m+k = \binom{m+d}{m}$ and $(d,m) \neq (3,2)$,
  \item $(k-1)m+k = \binom{m+d}{m}$ and $(d,m) = (3,2)$. 
  \end{enumerate}

  For the case (a1) (i.e., $(k-1)m+k < \binom{m+d}{m}$),
  we take $A = v_d^{-1}(q_i(F)  \cup \set{x_1, \dots, x_k})$ in $\Pm$ and $\Lambda=\Pb$.
  From Proposition~\ref{prop:dim-qiF-k-1-plane} (i),
  it follows
  $\dim \lin{v_{d-1,m}(A)} \geq k + (km+k-1) - \beta$ for the ($d-1$)-uple Veronese embedding
  $v_{d-1,m}$ of $\Pm$.
  From Proposition~\ref{prop:dim-linTT-geq-vdAA},
  the dimension of the left hand side in (\ref{containment_large_m}) is greater than or equal to
  the number~(\ref{eq:sum-vdAA}), which is
  \begin{gather*}
    \dim \lin{\Lambda  \cup v_{d}(A)} + (n-m)\big\{1 + \dim\lin{v_{d-1,m}(A)}\big\}
    \geq 
    \beta + (n-m)(1+k+(km+k-1)-\beta).
  \end{gather*}
  From the inclusion (\ref{containment_large_m}), we obtain
  \[
    \beta + (n-m)(1+k+(km+k-1)-\beta) \leq kn+k-1,
  \]
  which implies $(n-m)(1+(km+k-1)-\beta) \leq (km+k-1)-\beta$.
  It is a contradiction, because $n > m$ and $(km+k-1)-\beta > 0$.

  Now, assume $(k-1)m+k = \binom{m+d}{m}$ (equivalently, $km+k-1-\beta=m$).
  Then, in the same way as above,
  using Proposition~\ref{prop:dim-qiF-k-1-plane} (ii) and
  taking $A = \Pm=v_d^{-1}(Z)$ and $\Lambda=\lin{v_d(\Pm)}=\Pb$,
  we have
  \begin{equation}\label{eq:fib-dim-m-sum-vdAA}
    \beta + (n-m)(1+\dim\lin{v_{d-1,m}(\Pm)}) \leq kn+k-1.
  \end{equation}
  For $(d,m) \neq (3,2)$ (i.e., the case of (a2)),
  Proposition~\ref{prop:dim-qiF-k-1-plane} implies that $\dim \lin{v_{d-1,m}(\Pm)} \geq k + m$.
  Then
  $(n-m)(m+1) \leq km+k-1-\beta=m$, contrary to $n > m$.

  For $(d,m) = (3,2)$ (i.e., the case (a3)),
  we get $\beta = 9$, $\dim\lin{v_{d-1,m}(\Pm)} = 5$, and $k = \lceil {\binom{m+d}{m}}/({m+1}) \rceil = 4$. 
  The condition $(k,d,n) \neq (4,3,3)$ implies $n \geq 4$.
  Then we also have a contradiction since (\ref{eq:fib-dim-m-sum-vdAA}) does not hold.
  Hence we show that $\sigma_k(v_d(\P^m))  \subset \Sing(\sigma_k(v_d(\Pn)))$.
  \medskip
  
  Secondly, let us regard (ii) of Theorem~\ref{general_m2}.
  For $(k,d,m)=(10,3,5), (10,6,2)$,
  we have the same result as Theorem~\ref{general_m}
  since $\sigma_k(v_d(\Pm)) = \Pb$ and it satisfies (a1),(a2) respectively.
  Then, except $(k,d,m)=(9,3,5),(8,4,3),(9,6,2)$,
  the remaining part of (ii) of Theorem~\ref{general_m2} consists of the following two cases
  \begin{enumerate}[\quad ({b}1)]
  \item 
    $(k,d,m)=(7,3,4),(5,4,2),(9,4,3),(14,4,4)$ (i.e., the case of $\sigma_k(v_d(\Pm))$ being \ti{defective}),

  \item 
    $(k,d,m)=(8,3,4),(6,4,2),(10,4,3),(15,4,4)$ (i.e., \ti{just after the defective case}).
  \end{enumerate}

  By the same reason, we also have the inclusion (\ref{containment_large_m}) for these cases provided that $\sigma_k(v_d(\Pm))  \not\subset \Sing(\sigma_k(v_d(\Pn)))$.
  
  For the case (b1), i.e., the defective case, it is known that all the $\sigma_k(v_d(\Pm))$'s are hypersurfaces in $\Pb$ (see \cite{AH}). So, as taking $A = v_d^{-1}(q_i(p^{-1}(a)))  \subset \Pm$ corresponding to the entry locus of $a$, by Proposition~\ref{prop:dim-linTT-geq-vdAA} a inclusion of the same kind as (\ref{containment_large_m}) implies
  \begin{equation}\label{containment_defect}
    \beta-1 + (n-m)(1+\dim\lin{v_{d-1}(A)}) \leq kn+k-1,
  \end{equation}
  where $\beta$ is equal to $km+k-\delta$ and $\delta$ is the secant defect of $\sigma_k(v_d(\Pm))$. (\ref{containment_defect}) is equivalent to
  \[
    n-m\le\frac{\delta}{1+\dim\lin{v_{d-1}(A)}-k}
    \le \frac{\delta}{1+\delta}<1,
  \]
  which contradicts $n-m \ge 1$,
  because $\dim\lin{v_{d-1}(A)}\ge k+\delta$ by Remark~\ref{dim_fib_defective}.
  
  For the case (b2), i.e., {just after the defective case} (b1),
  the $k$-th secant variety $\sigma_k(v_d(\Pm))$ fills up $\Pb$,
  and hence $\TT_a (\sigma_k(v_d(\Pm)))=\Pb$.
  Then, we can also get a contradiction in a similar way, as follows.
  Since the $(k-1)$-secant variety $\sigma_{k-1}(v_d(\Pm))$ is a hypersurface in $\Pb$,
  by Lemma~\ref{lem:q_1F=Z}
  we have $q_i(F)=v_d(\Pm)$ for an irreducible component $F$ of $p^{-1}(a)$
  for general $a \in \Pb$
  so that we can take $A=\Pm$.
  By Proposition~\ref{prop:dim-linTT-geq-vdAA}, the inclusion (\ref{containment_large_m}) implies
  \begin{equation*}
    \beta + (n-m)(1+\dim\lin{v_{d-1}(\Pm)})= \binom{m+d}{m}-1+(n-m)\binom{m+d-1}{m}\leq kn+k-1,
  \end{equation*}
  which fails to hold in (b2);
  more precisely,
  for $(k,d,m)=(8,3,4),(6,4,2),(10,4,3),(15,4,4)$,
  the value $\binom{m+d}{m}-1+(n-m)\binom{m+d-1}{m} - (kn+k-1)$ is equal to
  \begin{gather*}
    7n - 33,4n - 11,10n - 35,20n - 85,
  \end{gather*}
  respectively, which must be greater than $0$ because of the condition $n\ge m+1$.
  Thus, we obtain $\sigma_k(v_d(\P^m))  \subset \Sing(\sigma_k(v_d(\Pn)))$. 
  \medskip
  
  Now, we discuss the following two cases
  \begin{enumerate}[\quad ({c}1)]
  \item 
    $k=\frac{\binom{m+d}{m}}{m+1} \in \NN$ of Theorem~\ref{general_m};
    then, since we exclude $(k,d,m) = (5,3,3),(7,5,2)$,
    a general point $a \in \Pb = \sigma_k(v_d(\Pm))$ is \emph{not $k$-identifiable} and
    the secant fiber $p^{-1}(a)$ consists of two or more points (see \cite[Theorem~1]{GM}),

  \item 
    $(k,d,m)=(9,3,5),(8,4,3),(9,6,2)$ of Theorem~\ref{general_m2};
    then, a general point $a \in \sigma_k(v_d(\Pm)) \subsetneq \Pb$ is \emph{not $k$-identifiable} and
    $p^{-1}(a)$ consists of two points (see \cite[Theorem~1.1]{COV2}).
  \end{enumerate}

  In these cases, even though they do not have positive dimensional secant fibers,
  we can still get a proof by contradiction using a different estimate, as follows.

  Similarly, suppose that $\sigma_k(v_d(\Pm)) \not\subset \Sing(\sigma_k(v_d(\Pn)))$ and take a general point $a\in\sigma_k(v_d(\Pm))$ so that $a$ is a smooth point in both $\sigma_k(v_d(\Pm))$ and $\sigma_k(v_d(\P^n))$.
  We take general $k$ points $x_1,\dots , x_k \in v_d(\Pm)$ with
  $a\in\lin{x_1,\dots ,x_k}$.
  By the non-identifiability,
  we have another set of $k$ points $y_1,\dots, y_k \in v_d(\Pm)$
  with $a\in\lin{y_1,\dots ,y_k}$
  such that 
  $(a, x_1,\dots , x_k)$ and $(a, y_1,\dots, y_k)$ are \emph{distinct} in the secant fiber $p^{-1}(a) \subset I \subset \Pb \times (v_d(\Pm))^k$ (modulo permutation on $(v_d(\Pm))^k$).
  Let $x_i'\in\P^m$ (resp. $y_j'$) be the preimage of $x_i$,
  that is $v_d(x_i')=x_i$ (resp. of $y_j$ with $v_d(y_j')=y_j$). 

  Setting $A=\{x_1',\dots,x_k',y_1',\dots,y_k'\}$, we have an inclusion, similar to (\ref{containment_large_m}),
  \begin{equation}\label{containment_3cases}
    \bigl\langle\;
    \TT_a (\sigma_k(v_d(\P^m)))
    \,\cup\bigcup_{x \in v_d(A)}
    \TT_x (v_d(\Pn))
    \;\bigr\rangle
    \,\subset\, \TT_a (\sigma_k(v_d(\Pn))).
  \end{equation}

  For the $(d-1)$-uple Veronese embedding
  $v_{d-1}=v_{d-1,m}: \Pm \hookrightarrow \PP^{\beta_{d-1}}$ with
  $\beta_{d-1} = \binom{m+d-1}{m}-1$,
  we have $\dim \lin{v_{d-1}(x_1'),\ldots,v_{d-1}(x_k')} = k-1$
  since $x_1',\dots,x_k'$ are general in $\P^m$.
  The $(k-1)$-plane $\lin{v_{d-1}(x_1'),\ldots,v_{d-1}(x_k')}$ is contained in $\lin{v_{d-1,m}(A)}$.
  On the other hand, the codimension of $v_{d-1}(\Pm) \subset \PP^{\beta_{d-1}}$,
  that is $\binom{m+d-1}{m}-1 -m$, is greater than or equal to $k$;
  it follows from (iii) of Lemma~\ref{lem:codim-vd1Pm} in the case (c1),
  and from explicit calculations in the case (c2).

  Then, we have $\dim\lin{v_{d-1,m}(A)}\ge k$, as follows.
  Otherwise, $\dim\lin{v_{d-1,m}(A)}\le k-1$ implies
  $\lin{v_{d-1,m}(A)}=\lin{v_{d-1}(x_1'),\ldots,v_{d-1}(x_k')}$.
  Since $y_1', \dots, y_k' \in A \subset \Pm$, it follows
  \begin{gather*}
    v_{d-1}(y_1'), \dots, v_{d-1}(y_k')
    \in \lin{v_{d-1}(x_1'),\ldots,v_{d-1}(x_k')} \cap v_{d-1}(\Pm),
  \end{gather*}
  where the right hand side must be $\{v_{d-1}(x_1'),\ldots,v_{d-1}(x_k')\}$
  because of the generalized trisecant lemma (\cite[Proposition~1.4.3]{Rus}),
  which gives a contradiction.

  Again by (\ref{containment_3cases}) and Proposition~\ref{prop:dim-linTT-geq-vdAA}, we get
  \begin{align*}
kn+k-1&\ge\dim\langle\TT_a \sigma_k(v_d(\P^m))  \cup v_d(A)\big\rangle+(n-m)\{1+\dim\langle v_{d-1,m}(A)\rangle\}\\
&\ge \dim\TT_a \sigma_k(v_d(\P^m))+(n-m)(1+k)=km+k-1+(n-m)(1+k)\\
&=kn+k+(n-m-1),
   \end{align*}
   which is a contradiction since $n-m-1\ge0$. Thus, it also holds that $\sigma_k(v_d(\P^m))  \subset \Sing(\sigma_k(v_d(\Pn)))$ in these generic non-identifiable cases.

   Note that for $(k,d,m)=(10,3,5),(9,4,3),(10,6,2)$, i.e.,
   just after the non-identifiable case (c2),
   the singularity is already shown in the second part of this proof,
   where $(k,d,m)=(9,4,3)$ is also in the defective case (b1).
   The case $(k,d,m) = (5,3,3),(7,5,2)$, which is excluded from (c1),
   belongs to (i) of Theorem~\ref{general_m2};
   in this sense, the non-trivial singularity does not appear for $(d,m)=(3,3),(5,2)$.
   \medskip

   Finally, since $\sigma_{k-1}(v_d(\Pm)) \subsetneq\sigma_k(v_d(\Pm))$ for the $k$'s of the range in this Part (ii), the $\sigma_k(v_d(\Pm))$ is a non-trivial singular locus, which means $\sigma_k(v_d(\Pm))  \not\subset \sigma_{k-1}(v_d(\Pn))$, by Lemma~\ref{lem:sigma-k-1}.
\end{proof}

We finish this section as proving Part (iii) of Theorems \ref{general_m} and \ref{general_m2} and Part (iv) of Theorem~\ref{general_m}.

\begin{proof}[Proof of Part (iii) of Theorem~\ref{general_m} and Part (iii) of Theorem~\ref{general_m2}] 
  By the conditions in the Part (iii) of these two theorems, we see that
  \begin{gather*}
    k-1\ge \Bigl\lceil\tfrac{{m+d\choose m}}{m+1}\Bigr\rceil
  \end{gather*}
  if $\sigma_k(v_d(\Pm))$ is never defective, or
  \begin{gather*}
    k-1\ge \Bigl\lceil\tfrac{{m+d\choose m}}{m+1}\Bigr\rceil+1
  \end{gather*}
  else if $(d,m) \in \{(3,4),(4,2),(4,3),(4,4)\}$, the defective list of Alexander-Hiroschowitz.
  In any case, we have $\sigma_{k-1}(v_d(\Pm)) =  \lin{v_d(\Pm)}$. Hence,
  \begin{gather*}
    \sigma_{k}(v_d(\Pm)) = \sigma_{k-1}(v_d(\Pm))  \subset \sigma_{k-1}(v_d(\Pn))
  \end{gather*}
  and the assertion follows.
\end{proof}

\begin{proof}[Proof of Part (iv) of Theorem~\ref{general_m}]
  It is shown in \cite{FH3}
  by giving the defining equations of $\sigma_4(v_3(\P^3))$ explicitly.
\end{proof}

\section{Case of 4-th secant variety of Veronese embedding}\label{Sing4Vero}

In this section we aim to prove Theorem~\ref{sing_4th_Vero} as an investigation of the singular loci of the \ti{fourth} secant variety (i.e., $k=4$) of any Veronese variety. This theorem consists of one part dealing with the (non-)singularity of points in full-secant loci (i.e., $m=3$) and the other part for points in the maximum subsecant loci $\SSigM[n]{4,d}$. So, we will obtain Theorem~\ref{sing_4th_Vero} by proving Theorem~\ref{sing_4th_Vero_gen} (Part (i) of Theorem~\ref{sing_4th_Vero}) and Corollary \ref{sing_4th_Vero_sub} (Part (ii) and (iii) of Theorem~\ref{sing_4th_Vero}). 

\subsection{Equations by Young flattening}

In \cite{LO} another source of equations for secant varieties of Veronese varieties was introduced via the so-called \ti{Young flattening}. Here we briefly review the construction of a certain type of Young flattening and use it to compute the conormal space of a given form. 

Let $V = \CC^{n+1}$ and $d = d_1+d_2+1$. For $1\le a\le n$, we consider a map
\begin{gather*}
  \YF^a_{d_1, d_2, n}: S^d V \rightarrow
  S^{d_1}V \otimes S^{d_2}V \otimes \bigwedge^{a}V^\ast \otimes \bigwedge^{a+1} V
\end{gather*}
which is obtained by first embedding $S^d V \hookrightarrow
S^{d_1}V \otimes S^{d_2}V \otimes V$ via co-multiplication, then tensoring with $\mathrm{Id}\in \bigwedge^{a}V \otimes \bigwedge^{a} V^\ast$, and finally skew-symmetrizing and permuting.

For any $f\in S^d V$, we identify $\YF^a_{d_1, d_2, n}(f) \in S^{d_1}V \otimes S^{d_2}V \otimes \bigwedge^{a}V^\ast \otimes \bigwedge^{a+1} V$
as a linear map
\begin{equation}\label{YF_map}
  S^{d_1}V^* \otimes \bigwedge^{a}V \rightarrow S^{d_2}V \otimes \bigwedge^{a+1} V.
\end{equation}
Let $\alpha_1, \dots, \alpha_{\binom{n+1}{a}}$ give a basis of $\bigwedge^{a}V$.
For a decomposable $w^d\in S^d V$, $\YF^a_{d_1, d_2, n}$ maps as
\[
  w^d \mapsto \frac{d!}{d_1! d_2!}w^{d_1} \otimes w^{d_2} \otimes \big(\sum_I \alpha_I^\ast \otimes (\alpha_I \wedge w)\big)
\]
and if we take $z_0, \dots, z_n$, a basis of $V$ (now, $w = \sum c_jz_j \in V$ for some $c_j$'s and $\alpha_I=z_{i_1}\wedge\cdots\wedge z_{i_a}$ for some distinct $i_1,\ldots,i_a$), then we have
\[
  \YF^a_{d_1, d_2, n}(w^d) = 
  {\small\frac{d!}{d_1! d_2!}}\sum^n_{j=0} c_j \sum_{i_1, \dots, i_a \neq j} w^{d_1}
  \otimes w^{d_2}  \otimes (z_{i_1}\wedge \dots \wedge z_{i_a})^\ast\otimes (z_{i_1}\wedge \dots \wedge z_{i_a} \wedge z_j),
\]
which shows $\YF^a_{d_1, d_2, n}(w^d)$ has rank $\binom{n}{a}$ as the linear map (note that the rank does not depend on the choice of $w$ and just consider the case $w=z_0$). Further, for $k \leq \binom{n+d'}{d'}$ with $d' = min\set{d_1, d_2}$ it is also immediate to see that $\rank (\YF^a_{d_1, d_2, n}(f)) = k\binom{n}{a}$ for a general $k$ sum of $d$-th power $f = \sum_{i=1}^k w_i^d$.

Thus, from $k\binom{n}{a}+1$ minors of the matrix $\YF^a_{d_1, d_2, n}(f)$ we obtain a set of equations for $\sigma_k(v_d(\PP V))$ for this range of $k$ (for some values of $k,d,d',a$, it is known that these minors cut $\sigma_k(v_d(\P V))$ as an irreducible component (see \cite[Theorem~1.2.3]{LO})).

We can also use this Young flattening to compute conormal space of secant varieties of Veronese.

\begin{proposition}\label{conormal_YF}
Let $V=\CC^{n+1}$ and $f$ be any (closed) point of $\sigma_k(v_d(\P V)) \setminus\sigma_{k-1}(v_d(\P V))$ in $\P S^d V$. Suppose $\YF_{d_1,d_2,n}^a(f)$ has rank $k{n\choose a}$ as a linear map in $\Hom(S^{d_1}V^\ast \otimes \bigwedge^{a}V, S^{d_2}V \otimes \bigwedge^{a+1} V)$. Then, we have
\begin{equation*}
\hat{N}^{\ast}_{[f]}\sigma_k(v_d(\P V))\supseteq (\ker \YF_{d_1,d_2,n}^a(f))\cdot (\im\YF_{d_1,d_2,n}^a(f))^\perp~, 
\end{equation*}
where the right hand side is to be understood as the image of the multiplication
\begin{gather*}
  S^{d_1}V^\ast \otimes \bigwedge^{a}V\otimes S^{d_2}V^\ast \otimes \bigwedge^{a+1} V^\ast\to S^{d}V^\ast.
\end{gather*}
\end{proposition}

\begin{proof}
  This proposition follows directly from the same idea as Proposition~\ref{conormal_prop} by applying it to a linear embedding
  \begin{gather*}
    S^d V \hookrightarrow
    S^{d_1}V \otimes \bigwedge^{a}V^\ast \otimes S^{d_2}V \otimes \bigwedge^{a+1} V.
  \end{gather*}
  Since $\rank\YF_{d_1,d_2,n}^a(f)=k{n\choose a}$ and, as observed before,
  $v_d(\P V)$ is contained in
  \begin{gather*}
    \sigma_{n\choose a}(Seg(\P(S^{d_1}V \otimes \bigwedge^{a}V^\ast)\times \P(S^{d_2}V \otimes \bigwedge^{a+1} V)))
    \subset \P(S^{d_1}V \otimes \bigwedge^{a}V^\ast\otimes S^{d_2}V \otimes \bigwedge^{a+1} V)
  \end{gather*}
  and not in the previous secants of the same Segre variety, it is straightforward from the proof of Proposition~\ref{conormal_prop} (i.e., the case $p={n\choose a}$).
\end{proof}

\subsection{Singularity and non-singularity}

Using Proposition~\ref{conormal_YF},
we have the non-singularity of $\sigma_4(v_d(\Pn))$ at any point outside $\SSig[n]{4,d}{2} \cup\sigma_3(v_d(\Pn))$.

\begin{theorem}[from full-secant locus]\label{sing_4th_Vero_gen}
  Let $v_d:\P^n \rightarrow \P^N$
  be the $d$-uple Veronese embedding
  with $n\ge3$, $d\ge3$, and $N={n+d\choose d}-1$.
  Suppose that $f\in \sigma_4(v_d(\P^n)) \setminus \sigma_3(v_d(\P^n))$ and $f$ does not belong to any $2$-subsecant $\sigma_4(v_d(\P^2))$ of $\sigma_4(v_d(\P^n))$.
  Then $\sigma_4(v_d(\P^n))$ is smooth at every such $f$.
\end{theorem}

\begin{proof}
First, note that for every $f$ in the statement there exists a unique 4-dimensional subspace $U$ such that $f\in\sigma_4(v_d(\P U))$, which is determined by the kernel of the symmetric flattening $\phi_{1,d-1}$. This gives a fibration as
\begin{equation*}
  \pi:\sigma_4(v_d(\P^n)) \setminus\big(\SSign[\Pn]{4,d}{2} \cup\sigma_3(v_d(\P^n))\big)\to \Gr(3,\P^n)~
\end{equation*}
whose fibers $\pi^{-1}(\P U)$'s are all isomorphic to
$\sigma_4(v_d(\P U)) \setminus(\SSign[\PP U]{4,d}{2} \cup\sigma_3(v_d(\P U)))$,
recalling that $\SSign[\Pn]{4,d}{2} \subset \sigma_4(v_d(\P^n))$ is the maximum subsecant locus, i.e., the union of all $\sigma_4(v_d(\P^1))$'s and $\sigma_4(v_d(\P^2))$'s in $\sigma_4(v_d(\P^n))$. So, we can reduce the proof of theorem to the case of $n=3$.

In case of $n=3$, there is a list of normal forms in $\sigma_4(v_d(\P^3)) \setminus(\SSigM[3]{4,d} \cup\sigma_3(v_d(\P^3)))$ due to Landsberg-Teitler (see \cite[Theorem~10.9.3.1]{La}) such as 
\begin{enumerate}[(i)]
\item $f_1={x_0}^d+{x_1}^d+{x_2}^d+{x_3}^d$, 
\item $f_2={x_0}^{d-1} x_1+{x_2}^d+{x_3}^d$, 
\item $f_3={x_0}^{d-1} x_1+{x_2}^{d-1} {x_3}$\,, 
\item $f_4={x_0}^{d-2} {x_1}^2+{x_0}^{d-1} x_2+{x_3}^d$, 
\item $f_5=x_0^{d-3} x_1^3+x_0^{d-2} x_1 x_2+x_0^{d-1} x_3$\,.
\end{enumerate}

Case (i) $f_1={x_0}^d+{x_1}^d+{x_2}^d+{x_3}^d$ (Fermat-type). It is well-known that this Fermat-type $f_1$ belongs to an almost transitive $\SL_4(\CC)$-orbit, which corresponds to a general point of $\sigma_4(v_d(\P^3))$. Hence, $f_1$ is a smooth point of $\sigma_4(v_d(\P^3))$.
\smallskip

Case (ii) $f_2={x_0}^{d-1} x_1+{x_2}^d+{x_3}^d$. Say
\begin{gather*}
  U=\CC\langle x_0,x_1,x_2,x_3\rangle.
\end{gather*}
Consider the Young flattening
\begin{gather*}
  \YF_{d-2,1,3}^1(f_2)\in S^{d-2} U\otimes U\otimes U^\ast\otimes\wedge^2 U\simeq\Hom(S^{d-2} U^\ast\otimes U, U\otimes\wedge^2 U)
\end{gather*}
defined in (\ref{YF_map}). For simplicity, we will denote this type of Young flattening by $\phi$ throughout the proof.
Then, $\phi(f_2)$ is
{\small
\begin{align*}
  &\alpha x_0^{d-2}\otimes x_0\otimes\big(\sum_{j=0}^{3} y_j\otimes x_j\wedge x_1\big)
  +\beta x_0^{d-3}x_1\otimes x_0\otimes\big(\sum_{j=0}^{3} y_j\otimes x_j\wedge x_0\big)
  +\gamma x_0^{d-2}\otimes x_1\otimes\big(\sum_{j=0}^{3} y_j\otimes x_j\wedge x_0\big)\\
  &+\delta x_2^{d-2}\otimes x_2\otimes\big(\sum_{j=0}^{3} y_j\otimes x_j\wedge x_2\big)
  +\epsilon x_3^{d-2}\otimes x_3\otimes\big(\sum_{j=0}^{3} y_j\otimes x_j\wedge x_3\big)
\end{align*}
}
for some nonzero $\alpha, \beta, \gamma, \delta, \epsilon\in\CC$. Note that as a linear map $S^{d-2} U^\ast\otimes U\to U\otimes\wedge^2 U$, $\rank\phi(x_0^5)=3$ and $\rank\phi(f_2)=4\cdot3=12$. Thus, by Proposition~\ref{conormal_YF} $(\ker\phi(f_2))\cdot (\im\phi(f_2))^\perp$ produces a subspace of $\hat{N}^{\ast}_{[f_2]}\sigma_4(v_d(\P^3))$. 

For $d=3$, the expected dimension of $\hat{N}^{\ast}_{[f_2]}\sigma_4(v_3(\P^3))$ for the smoothness is ${3+3\choose3}-16=4$ and the corresponding four can be chosen as $y_0y_2y_3, y_1^2 y_2, y_1y_2y_3, y_1^2y_3$ in $\Sd^3 U^\ast$ which are given by the product of $\{ y_1\otimes x_0, y_2\otimes x_2, y_3\otimes x_3\}$ in $\ker\phi(f_2) \subset U^\ast\otimes U$ and 
\begin{align}\label{im_f2}
&\{~y_0\otimes y_2\wedge y_3, ~y_1\otimes y_1\wedge y_2,  ~y_1\otimes y_1\wedge y_3, ~y_1\otimes y_2\wedge y_3,\\ \nonumber
&~y_2\otimes y_0\wedge y_1,  ~y_2\otimes y_0\wedge y_3, ~y_2\otimes y_1\wedge y_3, ~y_3\otimes y_0\wedge y_1,  ~y_3\otimes y_0\wedge y_2, ~y_3\otimes y_1\wedge y_2~\}~
\end{align}
in $\im\phi(f_2)^\perp \subset U^\ast\otimes\wedge^2 U^\ast$. So, $\sigma_4$ is non-singular at $f_2$.

For any $d\ge4$, in $\ker\phi(f_2) \subset S^{d-2} U^\ast\otimes U$, one can find a subspace generated by
\begin{gather*}
  \{F\otimes x_i~|~F\in J_{d-2},~i=0,\ldots,3\},
\end{gather*}
where $J=\langle y_0y_2, y_0y_3, y_1^2, y_1y_2, y_1y_3, y_2y_3\rangle$ is an ideal in $\Sym U^\ast$. Also, in $\im\phi(f_2)^\perp \subset U^\ast\otimes\wedge^2 U^\ast$ there exists the same subspace as in (\ref{im_f2}). In this case, our $(\ker\phi(f_2))\cdot (\im\phi(f_2))^\perp$ contains the subspace of $\Sd^d U^\ast$ generated by
\begin{align*}
&\{y_0^{2}y_2^2, y_0^{2}y_2y_3, y_0^{2}y_3^2\} \cup\{y_0y_1^2 y_2, \ldots, y_0y_1y_3^2\} \cup\{y_0y_2^2 y_3,y_0 y_2 y_3^2\}\cup\{y_1^4,\ldots,y_1^2 y_3^{2}\} \cup\{y_1y_2^{2}y_3,y_1y_2y_3^{2},y_2^{2}y_3^2\}
\end{align*}
for $d=4$ and by
\begin{align*}
&\{y_0^{d-2}y_2^2, y_0^{d-2}y_2y_3, y_0^{d-2}y_3^2\} \cup\{y_0^{d-3}y_1^2 y_2, y_0^{d-3}y_1^2 y_3, \ldots, y_0^{d-3}y_3^3\}\\
& \cup\{y_0^{d-4}y_1^4, y_0^{d-4}y_1^3 y_2,\ldots, y_0^{d-4}y_3^4\} \cup\ldots \cup\{ y_0^2 y_1^{d-2},  y_0^2 y_1^{d-3}y_2, \ldots, y_0^2 y_3^{d-2}\}\\
& \cup\{y_0 y_1^{d-1},  y_0 y_1^{d-2}y_2, \ldots, y_0y_1 y_2^{d-2} \} \cup\{y_0 y_2^{d-2}y_3, \ldots, y_0y_2 y_3^{d-2} \}\\
& \cup\{y_1^d,\ldots,y_1^2y_2^{d-2},\ldots,y_1^2 y_3^{d-2}\} \cup\{y_1y_2^{d-2}y_3,\ldots,y_1y_2y_3^{d-2}\} \cup\{y_2^{d-2}y_3^2,\ldots,y_2^2 y_3^{d-2}\}
\end{align*} for any $d>4$
(note that the terms above are listed in the lexicographical order). In both cases, these monomial generators can be also represented as 
{\small
  \begin{align*}
    &\bigl(\{y_0^d,y_0^{d-1}y_1,\cdots,y_0 y_3^{d-1}\} \setminus\{y_0^d, y_0^{d-1}y_1, y_0^{d-1}y_2, y_0^{d-1}y_3, y_0^{d-2}y_1^2, y_0^{d-2}y_1 y_2, y_0^{d-2}y_1 y_3, y_0^{d-3}y_1^3, y_0y_2^{d-1}, y_0 y_3^{d-1}\}\bigr)\\
    & \cup\bigl(\{y_1^d, y_1^{d-1}y_2, \cdots, y_3^{d}\} \setminus\{y_1y_2^{d-1}, y_1 y_3^{d-1}, y_2^d, y_2^{d-1}y_3, y_2 y_3^{d-1}, y_3^d\}\bigr)~,
  \end{align*}
}
which implies that by Proposition~\ref{conormal_YF}
\begin{align*}
\dim\hat{N}^{\ast}_{[f_2]}\sigma_4(v_d(\P^3))&\ge \dim(\ker\phi(f_2))\cdot (\im\phi(f_2))^\perp \\
&\ge  \bigg\{{d-1+3\choose3}-10+{d+2\choose2}-6\bigg\}=\binom{d+3}{3}-16.
\end{align*}
Hence, $f_2$ is a smooth point of $\sigma_4$.
\smallskip

Case (iii) $f_3={x_0}^{d-1} x_1+{x_2}^{d-1} {x_3}$. Then, $\phi(f_3)$ is
{\small
\begin{align*}
&~\alpha x_0^{d-2}\otimes x_0\otimes\big(\sum y_j\otimes x_j\wedge x_1\big)+\beta x_0^{d-3}x_1\otimes x_0\otimes\big(\sum y_j\otimes x_j\wedge x_0\big)+\gamma x_0^{d-2}\otimes x_1\otimes\big(\sum y_j\otimes x_j\wedge x_0\big)\\
&+\delta x_2^{d-2}\otimes x_2\otimes\big(\sum y_j\otimes x_j\wedge x_3\big)+\epsilon x_2^{d-3}x_3\otimes x_2\otimes\big(\sum y_j\otimes x_j\wedge x_2\big)+\eta x_2^{d-2}\otimes x_3\otimes\big(\sum y_j\otimes x_j\wedge x_2\big)
\end{align*}
}
for some nonzero $\alpha, \beta, \gamma, \delta, \epsilon, \eta\in\CC$ so that $\rank\phi(f_3)=12$. For $d=3$, a subspace $\langle y_1\otimes x_0, y_3\otimes x_2\rangle$ in $\ker\phi(f_3) \subset U^\ast\otimes U$ and another subspace in $\im\phi(f_3)^\perp \subset U^\ast\otimes\wedge^2 U^\ast$
\begin{align}\label{im_f3}
&\langle~y_0\otimes y_2\wedge y_3, ~y_1\otimes y_1\wedge y_2,~  y_1\otimes y_1\wedge y_3, ~y_1\otimes y_2\wedge y_3,\\ \nonumber
&~y_2\otimes y_0\wedge y_1, ~y_3\otimes y_0\wedge y_1, ~y_3\otimes y_0\wedge y_3,  ~y_3\otimes y_1\wedge y_3~\rangle~
\end{align}
produces a desired 4-dimensional subspace $\langle~y_0y_3^2, y_1^2 y_2, y_1^2y_3, y_1y_3^2~\rangle$ in $\Sd^3 U^\ast$, which says that $\sigma_4$ is non-singular at $f_3$. 

Similarly, for the case of $d\ge4$, $(\ker\phi(f_3))\cdot (\im\phi(f_3))^\perp$ contains a subspace of $\hat{N}^{\ast}_{[f_3]}\sigma_4(v_d(\P^3)) \subset\Sd^d U^\ast$ which is generated by
{\Small
  \begin{align*}
&\bigl(\{y_0^d,y_0^{d-1}y_1,\cdots,y_0 y_3^{d-1}\} \setminus\{y_0^d, y_0^{d-1}y_1, y_0^{d-1}y_2, y_0^{d-1}y_3, y_0^{d-2}y_1^2, y_0^{d-2}y_1 y_2, y_0^{d-2}y_1 y_3, y_0^{d-3}y_1^3, y_0y_2^{d-1}, y_0 y_2^{d-2}y_3\}\bigr)\\
& \cup\bigl(\{y_1^d, y_1^{d-1}y_2, \cdots, y_3^{d}\} \setminus\{y_1y_2^{d-1}, y_1 y_2^{d-2}y_3, y_2^d, y_2^{d-1}y_3, y_2^{d-2} y_3^{2}, y_2^{d-3}y_3^3\}\bigr),
  \end{align*}
}
using a subspace $\langle~\{F\otimes x_i~|~F\in J_{d-2},~i=0,\ldots,3\}~\rangle$ in $\ker\phi(f_3)$, where $J$ is an ideal generated by $\{ y_0y_2, y_0y_3, y_1^2, y_1y_2, y_1y_3, y_3^2\}$ in $\Sym U^\ast$, and the same subspace in $\im\phi(f_3)^\perp$ as (\ref{im_f3}). Thus, $\dim\hat{N}^{\ast}_{[f_3]}\sigma_4(v_d(\P^3))\ge{d+3\choose3}-16$, which means that $f_3$ is also smooth. 
\smallskip

Case (iv) $f_4={x_0}^{d-2} {x_1}^2+{x_0}^{d-1} x_2+{x_3}^d$. For $d=3$, we have
{\small
\begin{align*}
\phi(f_4)&=2 x_0\otimes x_1\otimes\big(\sum y_j\otimes x_j\wedge x_1\big)+2 x_1\otimes x_0\otimes\big(\sum y_j\otimes x_j\wedge x_1\big)+2 x_1\otimes x_1\otimes\big(\sum y_j\otimes x_j\wedge x_0\big)\\
&+2 x_0\otimes x_0\otimes\big(\sum y_j\otimes x_j\wedge x_2\big)+2 x_0\otimes x_2\otimes\big(\sum y_j\otimes x_j\wedge x_0\big)+2 x_2\otimes x_0\otimes\big(\sum y_j\otimes x_j\wedge x_0\big)\\
&+6 x_3\otimes x_3\otimes\big(\sum y_j\otimes x_j\wedge x_3\big)
\end{align*}
}
and $\rank\phi(f_4)=12$. $\hat{N}^{\ast}_{[f_4]}\sigma_4(v_3(\P^3))$ contains a 4-dimensional subspace corresponding to $\langle -y_0y_2y_3+y_1^2 y_3, y_1 y_2y_3, y_2^3, y_2^2y_3\rangle$ which can be spanned by $\{y_2\otimes x_0, y_3\otimes x_3\}$ in $\ker\phi(f_4) \subset U^\ast\otimes U$ and  
\begin{align*}
&\{~y_3\otimes y_0\wedge y_1,  y_3\otimes y_0\wedge y_2, -y_1\otimes y_1\wedge y_2+y_2\otimes y_0\wedge y_2, -y_0\otimes y_2\wedge y_3+y_1\otimes y_1\wedge y_3\}
\end{align*}
in $\im\phi(f_4)^\perp \subset U^\ast\otimes\wedge^2 U^\ast$. So, $\sigma_4$ is non-singular at $f_4$.
\smallskip

For $d\ge4$, it holds that
{\small
\begin{align*}
\phi(f_4)&=2 x_0^{d-2}\otimes x_1\otimes\big(\sum y_j\otimes x_j\wedge x_1\big)+2(d-2) x_0^{d-3}x_1\otimes x_0\otimes\big(\sum y_j\otimes x_j\wedge x_1\big)\\
&+2(d-2) x_0^{d-3}x_1\otimes x_1\otimes\big(\sum y_j\otimes x_j\wedge x_0\big)+(d-2)(d-3) x_0^{d-4}x_1^{2}\otimes x_0\otimes\big(\sum y_j\otimes x_j\wedge x_0\big)\\
&+(d-1) x_0^{d-2}\otimes x_0\otimes\big(\sum y_j\otimes x_j\wedge x_2\big)+(d-1) x_0^{d-2}\otimes x_2\otimes\big(\sum y_j\otimes x_j\wedge x_0\big)\\&+(d-1)(d-2) x_0^{d-3}x_2\otimes x_0\otimes\big(\sum y_j\otimes x_j\wedge x_0\big)+d(d-1) x_3^{d-2}\otimes x_3\otimes\big(\sum y_j\otimes x_j\wedge x_3\big).
\end{align*}
}

In this case, $\rank\phi(f_4)$ is also $12$ and $\ker\phi(f_4)$ has a subspace $A_1$ which is generated by
\begin{align*}
&\{~\langle y_0 y_3, y_1 y_2, y_1y_3, y_2^2,y_2y_3\rangle_{d-2}\otimes x_i~(i=0,\ldots,3), ~\langle y_0 y_2\rangle_{d-2}\otimes x_0, ~\langle y_1^2\rangle_{d-2}\otimes x_0, \\ \nonumber
&~\langle y_3^2\rangle_{d-2}\otimes x_3,~\langle -2y_0 y_2+(d-1)y_1^2\rangle_{d-2}\otimes x_2 ~\}
\end{align*}
and $\im\phi(f_4)^\perp$ has a subspace $B_1$ spanned by
\begin{align*}
&\{~y_2\otimes y_1\wedge y_2, ~y_2\otimes y_1\wedge y_3,  ~y_2\otimes y_2\wedge y_3, ~y_3\otimes y_0\wedge y_1, ~y_3\otimes y_0\wedge y_2, ~ y_3\otimes y_1\wedge y_2, \\ \nonumber
&~-(d-1)y_1\otimes y_1\wedge y_2+2y_2\otimes y_0\wedge y_2,~ -2y_0\otimes y_2\wedge y_3+(d-1)y_1\otimes y_1\wedge y_3\}.
\end{align*}
Then, one can check that $A_1\cdot B_1$ produces a subspace of $\hat{N}^{\ast}_{[f_4]}\sigma_4(v_3(\P^3))$ in $\Sd^d U^\ast$ which is the degree $d$-part of an ideal $I_1$ generated by 19 quartics
{\small
  \begin{align*}
    \{&-4\underline{y_0^2 y_2^2}+(4d-4)y_0y_1^2 y_2-(d-1)^2 y_1^4,~-2\underline{y_0^2 y_2y_3}+(d-1)y_0 y_1^2y_3,~\underline{y_0^2y_3^2},~-2\underline{y_0y_1y_2^2}+(d-1)y_1^3 y_2,\\
      &\,-2\underline{y_0y_1 y_2y_3}+3 y_1^3y_3,
        ~\underline{y_0y_1y_3^2},~-2\underline{y_0y_2^3}+(d-1)y_1^2y_2^2,~-2\underline{y_0y_2^2y_3}+(d-1)y_1^2y_2y_3,~-2\underline{y_0y_2y_3^2}+(d-1)y_1^2y_3^2,\\
      &\,~\underline{y_1^3y_3},~\underline{y_1^2y_2^2},~\underline{y_1^2y_2y_3},~\underline{y_1^2y_3^2}, ~\underline{y_1y_2^3}, ~\underline{y_1y_2^2y_3},~\underline{y_1y_2y_3^2},~\underline{y_2^4},~\underline{y_2^3y_3},~\underline{y_2^2y_3^2}\ \}
  \end{align*}
}
(here, the \textit{underline} means the leading term with respect to the lexicographic order). Say $T=\Sym U^\ast$. $I_1$ has a minimal free resolution as
\begin{equation}\label{resoln_I1}
0\to T(-7)^4\to T(-6)^{22}\to T(-5)^{36}\to T(-4)^{19}\to I\to 0~,
\end{equation}
which shows that the Hilbert function of $I$ can be computed as
\begin{align*}
H(I,d)&=19{d-4+3\choose3}-36{d-5+3\choose3}+22{d-6+3\choose3}-4{d-7+3\choose3}\\
&={d+3\choose3}-16\quad(d\ge4).
\end{align*}
This implies that
\begin{gather*}
  {d+3\choose3}-16\ge\dim\hat{N}^{\ast}_{[f_4]}\sigma_4(v_d(\P^3))\ge H(I,d)={d+3\choose3}-16,
\end{gather*}
which means that $\sigma_4$ is also smooth at $f_4$.

Case (v) The final form $f_5=x_0^{d-3} x_1^3+x_0^{d-2} x_1 x_2+x_0^{d-1} x_3$. We begin with $d=3$. We have
{\small
\begin{align*}
\phi(f_5)&=6 x_1\otimes x_1\otimes\big(\sum y_j\otimes x_j\wedge x_1\big)+x_2\otimes x_0\otimes\big(\sum y_j\otimes x_j\wedge x_1\big)+x_2\otimes x_1\otimes\big(\sum y_j\otimes x_j\wedge x_0\big)\\&+x_1\otimes x_0\otimes\big(\sum y_j\otimes x_j\wedge x_2\big)+x_1\otimes x_2\otimes\big(\sum y_j\otimes x_j\wedge x_0\big)+x_0\otimes x_1\otimes\big(\sum y_j\otimes x_j\wedge x_2\big)\\
         &+x_0\otimes x_2\otimes\big(\sum y_j\otimes x_j\wedge x_1\big)+2x_3\otimes x_0\otimes\big(\sum y_j\otimes x_j\wedge x_0\big)+2 x_0\otimes x_0\otimes\big(\sum y_j\otimes x_j\wedge x_3\big)\\&+2 x_0\otimes x_3\otimes\big(\sum y_j\otimes x_j\wedge x_0\big)
\end{align*}
}
and $\rank\phi(f_5)=12$. The conormal space $\hat{N}^{\ast}_{[f_5]}\sigma_4(v_3(\P^3))$ contains a 4-dimensional subspace corresponding to $\langle -y_0y_3^2+4y_1 y_2 y_3-24y_2^3, -y_1 y_3^2+12y_2^2 y_3, y_2 y_3^2, y_3^3\rangle$ which can be spanned by $\{y_3\otimes x_0, 2y_1\otimes x_0+12y_2\otimes x_1+y_3\otimes x_2\}$ in $\ker\phi(f_5) \subset U^\ast\otimes U$ and  
\begin{equation}\label{f5_d=3}
\{~y_3\otimes y_1\wedge y_2, -2y_2\otimes y_1\wedge y_2+y_3\otimes y_0\wedge y_2, -2y_1\otimes y_2\wedge y_3+y_3\otimes y_0\wedge y_3\}
\end{equation}
in $\im\phi(f_5)^\perp \subset U^\ast\otimes\wedge^2 U^\ast$. So, $\sigma_4$ is non-singular at $f_5$.

For each $d\ge4$, the Young flattening is of the form
{\small
\begin{align*}
\phi(f_5)&=6 x_0^{d-3}x_1\otimes x_1\otimes\big(\sum y_j\otimes x_j\wedge x_1\big)+3(d-3) x_0^{d-4}x_1^2\otimes x_0\otimes\big(\sum y_j\otimes x_j\wedge x_1\big)\\
&+3(d-3) x_0^{d-4}x_1^2\otimes x_1\otimes\big(\sum y_j\otimes x_j\wedge x_0\big)+(d-3)(d-4) x_0^{d-5}x_1^{3}\otimes x_0\otimes\big(\sum y_j\otimes x_j\wedge x_0\big)\\
&+(d-2)(d-3) x_0^{d-4}x_1 x_2\otimes x_0\otimes\big(\sum y_j\otimes x_j\wedge x_0\big)+(d-2) x_0^{d-3}x_2\otimes x_0\otimes\big(\sum y_j\otimes x_j\wedge x_1\big)\\
&+(d-2) x_0^{d-3}x_2\otimes x_1\otimes\big(\sum y_j\otimes x_j\wedge x_0\big)+(d-2) x_0^{d-3}x_1\otimes x_0\otimes\big(\sum y_j\otimes x_j\wedge x_2\big)\\
&+(d-2) x_0^{d-3}x_1\otimes x_2\otimes\big(\sum y_j\otimes x_j\wedge x_0\big)+x_0^{d-2}\otimes x_1\otimes\big(\sum y_j\otimes x_j\wedge x_2\big)\\
&+x_0^{d-2}\otimes x_2\otimes\big(\sum y_j\otimes x_j\wedge x_1\big)+(d-1)(d-2) x_0^{d-3}x_3\otimes x_0\otimes\big(\sum y_j\otimes x_j\wedge x_0\big)\\
         &+(d-1) x_0^{d-2}\otimes x_0\otimes\big(\sum y_j\otimes x_j\wedge x_3\big)+(d-1) x_0^{d-2}\otimes x_3\otimes\big(\sum y_j\otimes x_j\wedge x_0\big)
\end{align*}
}
and $\rank\phi(f_5)$ is also $12$. Now, $\ker\phi(f_5)$ contains a subspace $A_2$ which is generated by
\begin{align*}
&\{~\langle ~y_1 y_3,~ y_2^2,~ y_2y_3,~ y_3^2~\rangle_{d-2}\otimes x_i~(i=0,\ldots,3), ~\langle -6y_0y_2+(d-2)y_1^{2}\rangle_{d-2}\otimes x_3,\\ \nonumber
&~\langle -y_0 y_3+(d-1)y_1y_2\rangle_{d-2}\otimes x_3~\}
\end{align*}
and $\im\phi(f_5)^\perp$ has a subspace $B_2$ spanned by
\begin{align*}
&\{~y_2\otimes y_2\wedge y_3, ~y_3\otimes y_1\wedge y_2,  ~y_3\otimes y_1\wedge y_3, ~y_3\otimes y_2\wedge y_3, ~-y_1\otimes y_2\wedge y_3+y_2\otimes y_1\wedge y_3, \\ \nonumber
&~-(d-1)y_2\otimes y_1\wedge y_2+y_3\otimes y_0\wedge y_2,~ -(d-1)y_1\otimes y_2\wedge y_3+y_3\otimes y_0\wedge y_3, \\ \nonumber
&~-y_0\otimes y_1\wedge y_2+y_1\otimes y_0\wedge y_2-y_2\otimes y_0\wedge y_1,~-y_0\otimes y_1\wedge y_3+y_1\otimes y_0\wedge y_3-y_3\otimes y_0\wedge y_1, \\ \nonumber
&~-6y_0\otimes y_2\wedge y_3+(d-2)y_1\otimes y_1\wedge y_3-6(d-1)y_2\otimes y_1\wedge y_2~\}~.
\end{align*}
Then, one can check that $A_2\cdot B_2$ produces a subspace of $\hat{N}^{\ast}_{[f_5]}\sigma_4(v_3(\P^3))$ in $\Sd^d U^\ast$  which is the degree $d$-part of an ideal $I_2$ generated by 19 quartics
{\small
  \begin{align*}
    \{\,
    &\,36y_0^2y_2^2-12(d-2)y_0y_1^2y_2+(d-2)^2 y_1^4 ,
      ~6y_0^2y_2y_3-(d-2)y_0y_1^2y_3-6(d-1)y_0y_1y_2^2+(d-1)(d-2)y_1^3y_2 ,\\
    &-y_0^2y_3^2+2(d-1)y_0y_1y_2y_3-(d-1)^2y_1^2y_2^2 , 
      ~-6y_0y_1y_2y_3+(d-2)y_1^3y_3 ,
      ~-y_0y_1y_3^2+(d-1)y_1^2y_2y_3 ,\\
    &-6y_0y_2^3 +(d-2)y_1^2y_2^2,
      ~y_0y_2^2y_3-(d-1)y_1y_2^3 ,
      ~y_0y_2y_3^2-(d-1)y_1y_2^2y_3,
      ~y_0y_3^3-(d-1)y_1y_2y_3^2,\\
    &\,(d-2)y_1^2y_2y_3-6(d-1)y_1y_2^3,
      ~y_1^2y_3^2,~y_1y_2^2y_3,~ y_1y_2y_3^2,~ y_1y_3^3,
      ~y_2^4 ,~ y_2^3y_3 , y_2^2y_3^2 ,~y_2y_3^3,~y_3^4
      \ \}.
  \end{align*}
}
Note that $I_2$ has the same minimal free resolution as $I_1$ in (\ref{resoln_I1}). Therefore, by the same argument, we conclude that $f_5$ is also a smooth point when $d\ge4$.
\end{proof}

As a direct consequence of the main results in the paper, we also obtain the following corollary on the (non-)singularity of subsecant loci in the 4-th secant variety.

\begin{corollary}[from subsecant loci]\label{sing_4th_Vero_sub}
  Let $v_d:\P^n \rightarrow \P^N$
  be the $d$-uple Veronese embedding
  with $n\ge3$, $d\ge3$, and $N={n+d\choose d}-1$.
  Then the following holds.
\begin{enumerate}
\item[(i)] A general point in $\sigma_4(v_d(\P^2)) \setminus\sigma_{3}(v_d(\P^n))$ is smooth for $d\ge4$. For $d=3$, $\sigma_4(v_3(\P^2))$ is a non-trivial singular locus for any $n\ge4$, while all points in $\sigma_4(v_3(\P^2)) \setminus\sigma_{3}(v_3(\P^3))$ are smooth for $n=3$.
\item[(ii)] $\sigma_4(v_d(\P^n))$ is smooth at each point in $\sigma_4(v_d(\P^1)) \setminus\sigma_{3}(v_d(\P^n))$ if $d\ge 7$. $\sigma_4(v_d(\P^1))$ is non-trivial singular locus when $d=6$ and $\sigma_4(v_d(\P^1)) \subset\sigma_{3}(v_d(\P^n))$ in case of $d\le 5$. 
\end{enumerate}
\end{corollary} 

\begin{proof}
  Since $k=4~ \&~ n\ge3$, the relevant range for an $m$-subsecant  locus in $\sigma_4(v_d(\Pn))$ is $1\le m\le 2$.

  (i) For $m=2$, Theorem~\ref{general_m} (ii) says that $\sigma_4(v_d(\P^m))$ is a non-trivial singular locus in $\sigma_4(v_d(\P^n))$ if $d=3, n\ge4$.
  The case $(d,n)=(3,3)$ is also discussed in Theorem~\ref{general_m} (iv).

When $d=4,5$ and $6$, we can say that a general point in $\sigma_4(v_d(\P^2)) \setminus\sigma_{3}(v_d(\Pn))$ is smooth by Theorem~\ref{general_m2} (i). For any $d\ge7$, the same conclusion follows from Theorem~\ref{general_m} (i).

(ii) This is given by Theorem~\ref{thm_m1} for the case $k=4, m=1$. 
\end{proof}

We add some remarks on Corollary \ref{sing_4th_Vero_sub}.

\begin{remark}\label{trivial_sing_4,2,2}
(a) For $d=2$, a subsecant variety $\sigma_4(v_d(\P^2))$ in $\sigma_4(v_d(\P^n))$ is a trivial singular locus, because $\sigma_4(v_d(\P^2))=\sigma_3(v_d(\P^2)) \subset\sigma_3(v_d(\P^n))$.

(b) As pointed out in Example \ref{sing_special_pt}, a singularity can occur at a \ti{special} point in $\sigma_4(v_d(\P^2)) \setminus\sigma_{3}(v_d(\P^n))$ even for $d\ge4$.
\end{remark}

Finally, we end this section by listing cases in which the same nice description for the singular locus of $\sigma_k(v_d(\PP^n))$ as in Example \ref{eg_smallest_4thsec} can be made.

\begin{corollary}\label{coro_for_cone_sing}
  Let $V$ be an $(n+1)$-dimensional complex vector space $(n\ge1)$ and $v_d(\PP V) \subset\PP^N$ be the image of the $d$-uple $(d\ge2)$ Veronese embedding of $\PP V$.
  Assume that $(k,d,n)$ satisfies one of the following conditions:
  \begin{enumerate}
  \item[(i)] $d=2$ and $n\ge k-1$,
  \item[(ii)] $k=2$, $d\ge2$ and $n\ge1$,
  \item[(iii)] $k=3$, $d=3$ and $n\ge2$, or $k=3$, $d=4$ and $n\ge3$,
  \item[(iv)] $k=4$, $d=3$ and $n\ge4$\,.
  \end{enumerate}  
Then the singular locus of $\sigma_k(v_d(\PP V))$ is given exactly as
\begin{equation}\label{nice_description_singloc}
  \big\{\,f\in \PP S^d V~|~\textrm{$f$ is any form which can be expressed using at most $k-1$ variables}\,\big\},\end{equation}
which is an irreducible locus of dimension
\begin{gather*}
  (k-1)(n-k+2)+{d+k-2\choose d}-1
\end{gather*}
and is equal to the maximum subsecant locus $\Sigma_{k,d}(\min\{k-1,n\}-1;\PP V)$.
\end{corollary}
\begin{proof} For case (i), the assertion is immediate since it corresponds to symmetric matrices. In case (ii), we draw the conclusion from the fact that $\Sing(\sigma_2(v_d(\PP V)))=v_d(\PP V)$ for every $d,n$ \cite{Kan}. 

For the remaining cases, we first claim that for any $3\le k\le n+1$ it holds
\begin{equation}\label{trivial in subsecant}
\sigma_{k-1}(v_d(\PP V)) \subset \bigcup_{\PP^{k-2} \subset\PP V}\lin{v_d(\PP^{k-2})}.
\end{equation}
We note that the right hand side of (\ref{trivial in subsecant}) is an irreducible and closed subvariety of $\sigma_{k}(v_d(\PP V))$,
since it coincides with a subvariety $\ds\bigcup_{\Lambda\in \mathrm{Im}\Phi}\Lambda$,
where a map
\begin{gather*}
  \Phi:\mathbb{G}(k-2,n)\to\mathbb{G}\left(\binom{d+k-2}{d}-1,N\right)
\end{gather*}
sending each subspace $L$ of dimension $k-2$ to the linear span $\lin{v_d(L)}$ in $\PP^N$ is regular (see e.g. \cite[Example 6.10, Proposition 6.13]{Harris}). Then, because a general element of the left hand side is of the form $\ell_1^d+\cdots+\ell_{k-1}^d$ for some linear forms $\ell_i$'s, it belongs to $\lin{v_d(\PP^{k-2})}$ for some $\PP^{k-2} \subset\PP V$ so that the closure is also contained in the subvariety $\ds\bigcup_{\PP^{k-2} \subset\PP V}\lin{v_d(\PP^{k-2})}$.

For case (iii), by \cite[Theorem 2.1, Remark 2.4(a), and Corollary 2.11]{H18} and Theorem~\ref{thm_m1}, and 
for case (iv), by Theorem~\ref{sing_4th_Vero}, we know that 
$$\Sing(\sigma_k(v_d(\PP V)))=\sigma_{k-1}(v_d(\PP V)) \cup\big\{\bigcup_{\PP^{k-2} \subset\PP V}\sigma_k(v_d(\PP^{k-2}))\big\},$$
which can also be written as
\begin{gather*}
  \sigma_{k-1}(v_d(\PP V)) \cup\Sigma_{k,d}(\min\{k-1,n\}-1;\PP V).
\end{gather*}
In both cases (iii) and (iv), we have $\sigma_k(v_d(\PP^{k-2}))=\lin{v_d(\PP^{k-2})}$. Thus, by the above claim, the singular locus is equal to 
\begin{gather*}
  \bigcup_{\PP^{k-2} \subset\PP V}\lin{v_d(\PP^{k-2})}=\Sigma_{k,d}(\min\{k-1,n\}-1;\PP V),
\end{gather*}
which is irreducible and can be described as written in the statement. The formula for the dimension is immediate from dimension counting.
\end{proof}

\section{Concluding remark}\label{sect_final_remark}

So far, we have reported results on singular loci of $\sigma_k(v_d(\Pn))$ coming from the subsecant loci. To the best of our knowledge, there is no general idea or clear consensus on the singular locus of an arbitrary higher secant variety of any Veronese variety yet.
From this point of view, the present paper contributes by providing a more visible picture on the singular locus via showing a generic smoothness of the subsecant loci for relatively low $k$'s and confirming the singularity of the same loci for other $k$'s.

As we mentioned in the introduction, each point $p\in\sigma_k(v_d(\Pn)) \setminus\sigma_{k-1}(v_d(\Pn))$ is located in $\sigma_k(v_d(\Pm)) \setminus\sigma_{k-1}(v_d(\Pn))$ for some $1\le m\le \min\{k-1,n\}$. To make the picture more complete, we have two future issues: (i) On the subsecant loci (i.e., $m<\min\{k-1,n\}$), one needs to check the (non-)singularity not only at a general point but also at \ti{every} point  and (ii) Points in the full-secant locus (i.e., $m=\min\{k-1,n\}$) should be treated.

Issue (i) is expected to be very complicated, because at some \ti{special} point a singularity can also occur even for a low $k$ as shown in Example \ref{sing_special_pt} (in fact, we can generate more examples using a similar idea).
For the points in the subsecant loci, in general, one could not hope to find some nice `normal forms' and the situation is expected to be \ti{wild} (in other words, the subsecant loci may not be covered with finitely many nice families of $\SL$-orbits).
But, still we can push on our viewpoint a bit further and, along the same spirit, we can refine a main result of this paper in the following manner. Based on the singularity results in Theorems~\ref{thm_m1}, \ref{general_m}, and \ref{general_m2} and using a similar estimation as in \S\ref{subsec_est}, more generally we have:

\begin{theorem}\label{Join_s_k_and_s_r}
  Suppose that $m=1$ and $k, d$ satisfy the condition of (ii) or (iii) of Theorem~\ref{thm_m1}, or suppose that $k,d,m$ satisfy the condition of (ii) or (iii) of Theorem~\ref{general_m} or Theorem~\ref{general_m2};
  in other words, the $m$-subsecant variety $\sigma_k(v_d(\Pm))$ is a singular locus in $\sigma_k(v_d(\Pn))$. Let $1\le m \le n-1$ and $r \leq n-m$. 
  Then, unless $\sigma_{k+r}(v_d(\mathbb{P}^n))$ fills up the ambient space $\mathbb{P}^N$, the following holds
  \begin{gather}\label{skvdPm-srvdPn-in-Sing}
    J(\sigma_k(v_d(\Pm)), \sigma_r(v_d(\Pn)))  \subset \Sing(\sigma_{k+r}(v_d(\Pn))),
  \end{gather}
  where $J(X,Y)$ denotes the `(embedded) join' of two subvarieties $X,Y$ in their ambient space. 
\end{theorem}
\begin{proof}
  Suppose that the inclusion~(\ref{skvdPm-srvdPn-in-Sing}) does not hold.
  Then, taking
  $x_1, \dots, x_{k}$ to be general points of $v_d(\Pm)$
  and
  $x_{k+1}, \dots, x_{k+r}$
  to be general points of $v_d(\Pn)$,
  we may assume that
  $x \notin \Sing(\sigma_{k+r}(v_d(\Pn)))$
  for a general $x \in \lin{x_1, \dots, x_{k+r}}$.
  By Terracini's lemma, we have
  \begin{equation}
    L_1 = \lin{\TT_{x_{1}}v_d(\Pn), \dots, \TT_{x_{k}}v_d(\Pn)}  \subset \TT_{x} \sigma_{k+r}(v_d(\Pn))
  \end{equation}
  and by the assumption on $k$ we know that $\dim L_1 > kn + k - 1$.

  On the other hand, since
  $x_{k+1}, \dots, x_{k+r}$
  are general points of $v_d(\Pn)$,
  \begin{gather*}
    L_2 = \lin{\TT_{x_{k+1}}v_d(\Pn), \dots, \TT_{x_{k+r}}v_d(\Pn)}  \subset \TT_{x} \sigma_{k+r}(v_d(\Pn))
  \end{gather*}
  and $L_2$ has dimension at least $rn +r - 1$.

  Moreover, we may assume $L_1 \cap L_2 = \emptyset$ as follows.
  Taking $\PP^{n-m-1}  \subset \Pn$
  such that $x_{k+1}, \dots, x_{k+r} \in \PP^{r-1}  \subset \PP^{n-m-1}$
  and $\Pm \cap \PP^{n-m-1} = \emptyset$,
  and changing coordinates $t_0, \dots, t_{m}, u_{1}, \dots, u_{m'}$ on $\Pn$ as in \S\ref{subsec_est},
  we may say that
  $\PP^{n-m-1}$ is the zero set of $t_0=\dots=t_{m}=0$ and $\PP^{m}$ is the zero set of $u_1=\dots=u_{m'}=0$. For a point $x'\in\Pm$, using the parameterization (\ref{eq:coordi-t0=1-tu}), the tangent space $\TT_{v_d(x')}v_d(\Pn)$ is spanned by the rows of the matrix of the form $[\ast:\mathsf{O}]$ as (\ref{eq:expr-TTx}). On the other hand, for a point $x''\in \PP^{n-m-1}$ and for an affine open set containing $x''$ we may take $u_{m'} = 1$ instead of $t_0=1$. Then, the only part on the parameterization of $v_d$ which contributes $\TT_{v_d(x'')}v_d(\Pn)$ is 
  \begin{gather*}
   \ t_0 \cdot \mo[u]{d-1},\ \dots,\ t_{m} \cdot \mo[u]{d-1}, \mo[u]{d}
  \end{gather*}
which corresponds to the tailing ``$*$'' part in (\ref{eq:coordi-t0=1-tu}) (recall that $\mo[u]{e}$ is the set of monomials of $\CC[u_1, \dots, u_{m'}]$ of degree $\le e$). Thus, a similar matrix whose rows span the other tangent space $\TT_{v_d(x'')}v_d(\Pn)$ has a form $[\mathsf{O}:\ast]$. This implies $L_1 \cap L_2 = \emptyset$.
 
  Hence $\dim \lin{L_1, L_2} > (k+r)n+(k+r) -1$, which is contrary to $\lin{L_1, L_2}  \subset \TT_{x} \sigma_{k+r}(v_d(\Pn))$.
\end{proof}

\begin{remark}[Partial subsecant locus]\label{p_subsecant}
This \ti{new} singular locus $J(\sigma_k(v_d(\Pm)), \sigma_r(v_d(\Pn)))$ in (\ref{skvdPm-srvdPn-in-Sing}) can be seen as a `partial version' of subsecant locus in this paper. In particular, it contains the $m$-subsecant variety $\sigma_{k+r}(v_d(\Pm))=J(\sigma_k(v_d(\Pm)), \sigma_r(v_d(\Pm)))$. So, let us call such a locus a \ti{partial subsecant locus} of $\sigma_{k+r}(v_d(\Pn))$. We note that the singularity of a specific form $f=x^2 y^2+z^4$ in Example \ref{sing_special_pt} can be explained using this notion; $f$ is a point of $\Sigma_{4,4}(2;\P^3)$ where only a generic smoothness is known by Theorems \ref{general_m2} (i), but $f$ also belongs to a partial subsecant locus $J(\sigma_3(v_4(\P^1)), \sigma_1(v_4(\P^3)))$ which is singular by Theorem~\ref{Join_s_k_and_s_r}.
\end{remark}

Therefore, one proper question on the singular locus of $\sigma_{k}(v_d(\Pn))$ here is probably such as
\begin{question}\label{question_smooth_gen}
Let $k-1\le n$ and $\mathcal{D}$ be the union of all possible (partial) subsecant loci of $\sigma_{k}(v_d(\Pn))$. Are the points of $\sigma_k(v_d(\PP^{n})) \setminus\big(\mathcal{D} \cup\sigma_{k-1}(v_d(\Pn))\big)$ all smooth in $\sigma_{k}(v_d(\Pn))$?
\end{question}
Note that the answer to Question \ref{question_smooth_gen} is affirmative in cases of $k=2$ classically and $k=3$ (by \cite{H18}) and $k=4$ (by Theorem~\ref{sing_4th_Vero_gen}). For a large value $k$ compared to $n$ (e.g. $n< k-1$), Question \ref{question_smooth_gen} may be answered negatively as in the following example.

\begin{example} Let us consider $\sigma_{14}(v_8(\PP^2))$,
  the $14$-th secant variety of the Veronese variety $v_{8}(\PP^2)$.
  Take general 14 points on $v_{8}(\PP^2)$.
  In \cite[remark 4.10]{AC19}, the authors presented a concrete point in the linear span of the 14 points which is a \ti{non-normal} point to $\sigma_{14}(v_8(\PP^2))$. Note that one can also check this singular point does not belong to $\mathcal{D}$, the locus of all partial subsecants. 
\end{example}

\begin{remark}\label{remk_for_next_project}
Finally, we would like to remark that the approach based on the same spirit of trichotomy pattern of (non-)singularity on subsecant loci still can be applied to the study of singular loci of higher secant varieties of other classical varieties such as Segre embeddings, Segre-Veronese varieties and Grassmannians. For instance, we can have a conjectural result like

\noindent\textit{`` For $\vec{\mbf{n}}=(n_1,n_2,\ldots,n_r)$, let $X$ be the Segre embedding $\P^{n_1}\times\P^{n_2}\times\cdots\times\P^{n_r} \subset\P^{\prod(n_i+1)-1}=:\P^{\beta(\vec{\mbf{n}})}$ and denote $\sigma_k(X)$ by $\sigma_k(\vec{\mbf{n}})$, the expected dimension of $\sigma_k(\vec{\mbf{n}})$ by $s_k(\vec{\mbf{n}})$. For every $\vec{\mbf{m}}=(m_1,m_2,\ldots,m_r)$ with $0\le m_i\le \min\{k-1,n_i-1\}$, besides a few exceptional cases, we have that the following holds 
\begin{enumerate}
\item[(i)] $\sigma_k(\vec{\mbf{n}})$ is smooth at a general point in $\sigma_k(\vec{\mbf{m}}) \setminus\sigma_{k-1}(\vec{\mbf{n}})$ if $\beta(\vec{\mbf{m}})\ge s_k(\vec{\mbf{m}})$
\item[(ii)] $\sigma_k(\vec{\mbf{m}})$ is singular in $\sigma_k(\vec{\mbf{n}})$, but $\sigma_k(\vec{\mbf{m}}) \not\subset\sigma_{k-1}(\vec{\mbf{n}})$ (i.e., non-trivial singular locus) if $s_{k-1}(\vec{\mbf{m}})<\beta(\vec{\mbf{m}})<s_k(\vec{\mbf{m}})$
\item[(iii)] $\sigma_k(\vec{\mbf{m}}) \subset\sigma_{k-1}(\vec{\mbf{n}})$ if $\beta(\vec{\mbf{m}})\le s_{k-1}(\vec{\mbf{m}})$ ''\,,
\end{enumerate}}
\noindent which can recover the result on the singular locus of the secant varieties of Segre embeddings (\cite[Corollary 7.17]{MOZ}) for $k=2$.
Note that, if we assume that everything is non-defective, then the ranges above can be computed as
\begin{align*}
\trm{(i)}\Leftrightarrow &\quad k\le\frac{\prod^{r}_{i=1}(m_i+1)}{\sum^{r}_{i=1}(m_i+1)-(r-1)}\\
\trm{(ii)}\Leftrightarrow &\quad \frac{\prod^{r}_{i=1}(m_i+1)}{\sum^{r}_{i=1}(m_i+1)-(r-1)}< k <\frac{\prod^{r}_{i=1}(m_i+1)}{\sum^{r}_{i=1}(m_i+1)-(r-1)}+1\\
\trm{(iii)}\Leftrightarrow &\quad k \ge\frac{\prod^{r}_{i=1}(m_i+1)}{\sum^{r}_{i=1}(m_i+1)-(r-1)}+1.
\end{align*}
We plan to deal with these cases in a forthcoming paper.
\end{remark}

\renewcommand{\MR}[1]{}
%\bibliography{biblio}

\end{document}